\documentclass[10pt]{amsart}
\usepackage{graphicx}
\usepackage{amsmath}
\usepackage{amssymb}
\usepackage{amsfonts}
\usepackage{verbatim}
\usepackage{scalefnt}
\usepackage{multirow}
\usepackage{mathtools}
\usepackage[margin=1.3in]{geometry}

\usepackage{pstricks,tikz}

\usepackage{fancyhdr}

\DeclareMathOperator*{\Int}{{\rm Int}}

\begin{document}

\def\COMMENT#1{}

\newtheorem{theorem}{Theorem}[section]
\newtheorem{lemma}[theorem]{Lemma}
\newtheorem{proposition}[theorem]{Proposition}
\newtheorem{corollary}[theorem]{Corollary}
\newtheorem{conjecture}[theorem]{Conjecture}
\newtheorem{claim}[theorem]{Claim}
\newtheorem{fact}[theorem]{Fact}

\numberwithin{equation}{section}

\def\eps{{\varepsilon}}
\newcommand{\cP}{\mathcal{P}}
\newcommand{\cT}{\mathcal{T}}
\newcommand{\cL}{\mathcal{L}}
\newcommand{\ex}{\mathbb{E}}
\newcommand{\eul}{e}
\newcommand{\pr}{\mathbb{P}}
\def\char{{\rm char}}
\newcommand{\bal}{{\rm bal}}

\title[Hamilton cycles in $3$-connected regular graphs]{Solution to a problem of Bollob\'as and H\"aggkvist on Hamilton cycles in regular graphs}
\author{Daniela K\"uhn, Allan Lo, Deryk Osthus and Katherine Staden}
\thanks{The research leading to these results was partially supported by the  European Research Council
under the European Union's Seventh Framework Programme (FP/2007--2013) / ERC Grant
Agreement n. 258345 (D.~K\"uhn and A.~Lo) and 306349 (D.~Osthus).}

\begin{abstract}
We prove that, for large $n$, every 3-connected $D$-regular graph on $n$ vertices with $D \geq n/4$ is Hamiltonian.
This is best possible and verifies the only remaining case of a conjecture posed independently by Bollob\'as and H\"aggkvist in the 1970s.
The proof builds on a structural decomposition result proved recently by the same authors.
\end{abstract}

\date{\today}
\maketitle


\section{Introduction}

In this paper we give an exact solution to a longstanding conjecture on Hamilton cycles in regular graphs, posed independently by Bollob\'as and H\"aggkvist: every sufficiently large $3$-connected regular graph on $n$ vertices with degree at least $n/4$ contains a Hamilton cycle.
The history of this problem goes back to Dirac's classical result that $n/2$ is the minimum degree threshold for Hamiltonicity.
This is certainly best possible -- consider e.g.~the almost balanced complete bipartite graph or the disjoint union of two equally-sized cliques.
The following natural question arises:
can we reduce the minimum degree condition by making additional assumptions on $G$?
The extremal examples above suggest that the family of regular graphs with some connectivity condition might have a lower minimum degree threshold for Hamiltonicity.
Indeed, Bollob\'as~\cite{egt} as well as H\"aggkvist (see~\cite{jackson}) independently made the following conjecture:
\emph{Every $t$-connected $D$-regular graph $G$ on $n$ vertices with $D \geq n/(t+1)$ is Hamiltonian.}%
\COMMENT{Bollob\'as's conjecture was stronger, with $D \geq n/(t+1)-1$.}
The case $t=2$ was first considered by Szekeres (see~\cite{jackson}), and after partial results by several authors including Nash-Williams~\cite{nashwilliams}, Erd\H{o}s and Hobbs~\cite{erdoshobbs} and Bollob\'as and Hobbs~\cite{bollobashobbs},
it was finally settled in the affirmative by Jackson~\cite{jackson}.
His result was extended by Hilbig~\cite{hilbig} who showed that one can reduce $D$ to $n/3-1$ unless $G$ is the Petersen graph $P$ or the $3$-regular graph $P'$ obtained by replacing one vertex of $P$ with a triangle.

However, Jung~\cite{jung} and independently Jackson, Li and Zhu~\cite{jlz}
found a counterexample to the conjecture for $t \geq 4$.
Until recently, the only remaining case $t=3$ was wide open.
Fan~\cite{fan} and Jung~\cite{jung} independently showed that every $3$-connected $D$-regular graph contains a cycle of length at least $3D$, or a Hamilton cycle.
Li and Zhu~\cite{lizhu} proved the conjecture for $t=3$ in the case when $D \geq 7n/22$%
\COMMENT{there's a mistake in Li's survey where this result is quoted -- it has $3n/22$ which is of course smaller than $n/4$!} and Broersma, van den Heuvel, Jackson and Veldman~\cite{bhjv} proved it for $D \geq 2(n+7)/7$.
In~\cite{jlz}, Jackson, Li and Zhu prove that if $G$ satisfies the conditions of the conjecture, 
any longest cycle $C$ in $G$ is dominating provided that $n$ is not too small.
(In other words, the vertices not in $C$ form an independent set.)
Recently, in~\cite{klos}, we proved an approximate version of the conjecture, namely that for all $\eps > 0$, whenever $n$ is sufficiently large, any $3$-connected $D$-regular graph on $n$ vertices with $D \geq (1/4 + \eps)n$ is Hamiltonian. 
Here, we prove the exact version (for large $n$).

\begin{theorem}\label{exact}
There exists $n_0 \in \mathbb{N}$ such that every $3$-connected $D$-regular graph on $n \geq n_0$ vertices with $D \geq n/4$ is Hamiltonian.
\end{theorem}

Our proof builds on the results in~\cite{klos}.
In particular, it relies on a structural decomposition result which holds for any dense regular graph: it gives a partition into (bipartite) robust expanders with few edges between these%
    \COMMENT{Deryk changed this sentence}
(see Section~\ref{sketch} and Theorem~\ref{structure}).~\cite{klos} also contains further applications of this partition result.

There are several natural analogues of these questions for directed and bipartite graphs. For example, the following conjecture of K\"uhn and Osthus~\cite{KOsurvey} is a directed analogue of Jackson's theorem~\cite{jackson}.
Further open problems are discussed in~\cite{klos}.
We say that a digraph $G$ is \emph{$D$-regular} if every vertex has both in- and out-degree $D$. 

\begin{conjecture}
Every strongly $2$-connected $D$-regular digraph on $n$ vertices with $D \geq n/3$ contains a Hamilton cycle.
\end{conjecture}

This paper is organised as follows.
In Section~\ref{example}, we discuss the extremal examples which show that Theorem~\ref{exact} is best possible.
Section~\ref{sketch} contains a sketch of the proof of Theorem~\ref{exact}.
Section~\ref{prelims} lists some notation, definitions and tools from~\cite{klos} which will be used throughout the paper.
The proof of Theorem~\ref{exact} is split into three cases, and these are considered in Sections~\ref{sec40}--\ref{sec21} respectively.
Finally, we derive Theorem~\ref{exact} in Section~\ref{proof}.

\section{The extremal examples}\label{example}

In this section we show that Theorem~\ref{exact} is best possible in the sense that neither the minimum degree condition nor the connectivity condition can be reduced.
The example of Jung~\cite{jung} and Jackson, Li and Zhu~\cite{jlz} shows that the minimum degree condition cannot be reduced for graphs with $n \equiv 1 \mod 8$ vertices; for completeness we extend this to all possible $n$ in the following proposition.
An illustration of their example may be found in Figure~\ref{fig:exactex}(i).

\begin{proposition} \label{extremalex}
Let $n \geq 5$%
\COMMENT{$|B|=D-1$ so $D-1=\lceil n/4 \rceil - 2 \geq 0$ so $n \geq 5$.}
and let $D$ be the largest integer such that $D \leq \lceil n/4 \rceil - 1$ and $nD$ is even.
Then there is an $(\lfloor n/8 \rfloor -1)$-connected $D$-regular graph $G_n$ on $n$ vertices which does not contain a Hamilton cycle.
\end{proposition}

\begin{proof}
Recall that a $D$-regular graph on $n$ vertices exists if and only if $n \geq D+1$ and $nD$ is even.
For each $n \geq 5$, we define a graph $G_n$ on $n$ vertices as follows.
Let $V_1,V_2,A,B$ be disjoint independent sets where $|A|=D$, $|B|=D-1$, and the other classes have sizes according to the table below.
Let $A_1,A_2$ be a partition of $A$ so that $\bigg|D/2-|A_1|\bigg|$ is minimal subject to the parity conditions below being satisfied:

\begin{center}
  \begin{tabular}{ l | l | l | l | l | l}
    $n$ & $D$ & $|V_1|$ & $|V_2|$ & $|A_1|$ & $|A_2|$\\ \hline 
    $8k+1$ & $2k$ & $2k+1$ & $2k+1$ & even & even\\
    $8k+2$ & $2k$ & $2k+2$ & $2k+1$ & even & even\\
    $8k+3$ & $2k$ & $2k+2$ & $2k+2$ & even & even\\
    $8k+4$ & $2k$ & $2k+3$ & $2k+2$ & even & even\\
    $8k+5$ & $2k$ & $2k+3$ & $2k+3$ & even & even \\
    $8k+6$ & $2k+1$ & $2k+3$ & $2k+2$ & odd & even\\
    $8k+7$ & $2k$ & $2k+4$ & $2k+4$ & even & even \\
    $8k+8$ & $2k+1$ & $2k+4$ & $2k+3$ & even & odd\\
  \end{tabular}
\end{center}
Note that $|V_i| \geq D+1$ for $i=1,2$.
Add every edge between $A$ and $B$.
First consider the cases when $D=2k$.
Then $|A_i|$ is even for $i=1,2$.
For each $i=1,2$, add edges so that $G_n[V_i]$ is $D$-regular.
Let $M_i$ be a matching of size $|A_i|/2$ in $G_n[V_i]$ and remove it.
Let $V_i' := V(M_i)$.
So $|V_i'|=|A_i|$.
Add a perfect matching between $V_i'$ and $A_i$.

Now consider the case when $D= 2k+1$.
Then, by our choice of $A_i$ and $V_i$ we have that $|A_i| \equiv |V_i| \mod 2$.
Fix $V_i' \subseteq V_i$ with $|V_i'| := |A_i|$.
Define the edge set of $G_n[V_i]$ so that for all $x \in V_i'$ we have $d_{V_i}(x) = D-1$ and for all $y \in V_i \setminus V_i'$ we have $d_{V_i}(y) = D$.%
\COMMENT{Case 1: $|V_i|$ odd. Since $|V_i| \geq D+1$ and we have that $|V_i| \geq D+2$.
Add every edge with both endpoints in $V_i$. Find a Hamilton decomposition (i.e.~$(|V_i|-1)/2 \geq (D+1)/2$ edge-disjoint HCs). Choose $(D+1)/2$ edge-disjoint Hamilton cycles $H_1, \ldots, H_{(D+1)/2}$. Each $H_i$ contains a matching of size $\lfloor |V_i|/2 \rfloor$. Let $M \subseteq H_{(D+1)/2}$ be a matching of size $(|V_i|-|A_i|)/2$. 
Let $G_n[V_i] := H_1 \cup \ldots \cup H_{(D-1)/2} \cup M$.
Case 2: $|V_i|$ even. Add every edge with both endpoints in $V_i$.
Find a 1-factorisation. Choose $D$ edge-disjoint perfect matchings $M_1, \ldots, M_{D}$ in this factorisation. Let $M \subseteq M_D$ have size $(|V_i|-|A_i|)/2$.
Let $G_n[V_i] := M_1 \cup \ldots \cup M_{D-1} \cup M$.}
Add a perfect matching between $V_i'$ and $A_i$.

Then $G_n$ has $n$ vertices, is $D$-regular and has connectivity $\min \lbrace |A_1|,|A_2| \rbrace \geq  \lfloor n/8 \rfloor - 1$.%
\COMMENT{AL: calculation changed slightly.
If $n \ne 8k +8$, $\min \lbrace |A_1|,|A_2| \rbrace \ge k-1 \ge \lfloor n/8 \rfloor-1$.
If $n = 8k +8$, $\min \lbrace |A_1|,|A_2| \rbrace  = k = n/8 -1$.}
Moreover, $G_n$ does not contain a Hamilton cycle because it is not $1$-tough ($G_n \setminus A$ contains more than $|A|$ components).
\end{proof}

\begin{center}
\begin{figure}[]
\begin{tikzpicture}[scale=0.8]
\tikzstyle{every node}=[font=\small]

\begin{scope}
\foreach \x in {-1.5,-1,-0.5,0,0.5,1,1.5}
\foreach \y in {-1.25,-0.75,...,1.25}
{\draw (3,\x) -- (5,\y);}

\draw[fill=gray!75] (0.5,1.5) circle (1);
\draw[fill=gray!75] (0.5,-1.5) circle (1);

\draw (0.5,2.25) -- (3,1.25);
\draw (0.5,2) -- (3,1);
\draw (0.5,1.75) -- (3,0.75);
\draw (0.5,1.5) -- (3,0.5);
\draw (0.5,1.25) -- (3,0.25);
\draw (0.5,1) -- (3,0);

\draw (0.5,-2.25) -- (3,-1.25);
\draw (0.5,-2) -- (3,-1);
\draw (0.5,-1.75) -- (3,-0.75);
\draw (0.5,-1.5) -- (3,-0.5);
\draw (0.5,-1.25) -- (3,-0.25);
\draw (0.5,-1) -- (3,0);

\draw[very thick,color=white] (0.5,2.25) -- (0.5,2);
\draw[very thick,color=white] (0.5,1.75) -- (0.5,1.5);
\draw[very thick,color=white] (0.5,1.25) -- (0.5,1);

\draw[very thick,color=white] (0.5,-2.25) -- (0.5,-2);
\draw[very thick,color=white] (0.5,-1.75) -- (0.5,-1.5);
\draw[very thick,color=white] (0.5,-1.25) -- (0.5,-1);

\foreach \w in {2.25,2,1.75,1.5,1.25,1}
{
\draw[fill=black] (0.5,\w) circle (1pt);
\draw[fill=black] (0.5,-\w) circle (1pt);
}

\draw[fill=white] (3,0) ellipse (0.5 and 1.5);
\draw[fill=white] (5,0) ellipse (0.5 and 1.25);
\draw[style=dashed] (2.5,0) -- (3.5,0);

\node at (5,-1.1) [draw=none,
label=below:{$2k-1$}] (){};

\node at (3,-1.3) [draw=none,
label=below:{$2k$}] (){};

\node at (0.5,0.7) [draw=none,
label=below:{$2k+1$}] (){};

\node at (0.5,-2.3) [draw=none,
label=below:{$2k+1$}] (){};

\node at (-0.4,1.5) [draw=none,
label=left:{$V_1$}] (){};

\node at (-0.4,-1.5) [draw=none,
label=left:{$V_2$}] (){};

\node at (3,1.3) [draw=none,
label=above:{$A$}] (){};

\node at (5,1.2) [draw=none,
label=above:{$B$}] (){};

\node at (3,0) [draw=none,
label=above:{$A_1$}] (){};
\node at (3,-1.2) [draw=none,
label=above:{$A_2$}] (){};

\node at (3,-2.9) [draw=none,
label=below:{(i)}] (){};
\end{scope}

\begin{scope}[xshift=12cm]
\tikzstyle{every node}=[font=\small]

\draw[fill=black] (0,2) circle (2pt);
\draw[fill=black] (0,-2) circle (2pt);

\node at (0,2) [draw=none,
label=above:{$a$}] (){};
\node at (0,-2) [draw=none,
label=below:{$b$}] (){};

\draw[fill=gray!75] (-3,0) circle (1);
\draw[fill=gray!75] (0,0) circle (1);
\draw[fill=gray!75] (3,0) circle (1);

\foreach \z in {-3,3}
\foreach \w in {-0.75,-0.25,0.25,0.75}
{
\draw (0,2) -- (\z+\w,0.25);
\draw (\z+\w,-0.25) -- (0,-2);
}

\foreach \z in {-3,3}
\foreach \w in {-0.75,-0.25,0.25,0.75}
{
\draw[thick,color=white] (\z+\w,0.25) -- (\z+\w,-0.25);
\draw[fill=black] (\z+\w,0.25) circle (1pt);
\draw[fill=black] (\z+\w,-0.25) circle (1pt);
}

\foreach \w in {-0.5,0,0.5}
{
\draw[thick,color=white] (\w,0.25) -- (\w,-0.25);
\draw (0,2) -- (\w,0.25);
\draw (\w,-0.25) -- (0,-2);
\draw[fill=black] (\w,0.25) circle (1pt);
\draw[fill=black] (\w,-0.25) circle (1pt);
}

\node at (0,-2.9) [draw=none,
label=below:{(ii)}] (){};
\end{scope}

\end{tikzpicture}
\caption{Extremal examples for Theorem~\ref{exact}.}\label{fig:exactex}
(i) is an illustration for the case $n=8k+1$. Here, each $V_i$ is a clique of order $2k+1$ with a matching of size $k$ removed.
\end{figure}
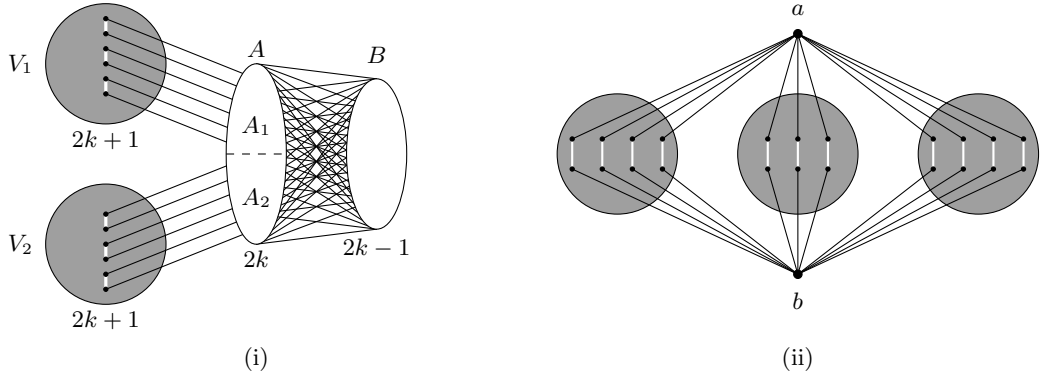
\end{center}

There also exist non-Hamiltonian $2$-connected regular graphs on $n$ vertices with degree close to $n/3$ (see Figure~\ref{fig:exactex}(ii)).
Indeed, we can construct such a graph $G$ as follows.
Start with three disjoint cliques on $3k$ vertices each.
In the $i$th clique choose disjoint sets $A_i$ and $B_i$ with $|A_i|=|B_i|$ and $|A_1|=|A_3|=k$ and $|A_2|=k-1$. Remove a perfect matching between $A_i$ and $B_i$ for each $i$. Add two new vertices $a$ and $b$, where $a$ is connected to all vertices in the sets $A_i$ and $b$ is connected to all vertices in all the sets $B_i$.
Then $G$ is a $(3k-1)$-regular $2$-connected graph on $n=9k+2$ vertices.
However, $G$ is not Hamiltonian because $G \setminus \lbrace a,b \rbrace$ has three components.
One can construct similar examples for all $n \in \mathbb{N}$.

Altogether this shows that none of the conditions --- degree or connectivity --- of Theorem~\ref{exact} can be relaxed.

\section{Sketch of the proof}\label{sketch}

\subsection{Robust partitions of dense regular graphs}

The main tool in our proof is a structural result on dense regular graphs that we proved in~\cite{klos}.
Roughly speaking, this allows us to partition the vertex set of such a graph $G$ into a small number of `robust components', each of which has strong expansion properties and sends few edges to the rest of the graph.
 
There are two types of robust components: \emph{robust expander components} and \emph{bipartite robust expander components}.
A robust expander component $G[U]$ is characterised by the following properties:
\begin{itemize}
\item for each $S \subseteq U$ which is neither too small nor too large, the `robust neighbourhood' $RN(S)$ of $S$ is significantly larger than $S$ itself;
\item $G$ contains few edges between $U$ and $V(G) \setminus U$.
\end{itemize}
Here the robust neighbourhood of $S$ is the set of all vertices in $U$ with linearly many neighbours in $S$.
A bipartite robust expander component $G[W]$ has slightly more structure: $G[W]$ can be made into a balanced bipartite graph by removing a small number of vertices and edges, and sets in the first class expand robustly into the second class.
More precisely, if $W$ has bipartition $A,B$ and $S \subseteq A$ is neither too large nor too small, then $RN(S) \cap B$ is significantly larger than $S$.
(Note that we do not require that sets in both vertex classes expand.)

We say that $\mathcal{V} = \lbrace V_1, \ldots, V_k, W_1, \ldots, W_\ell \rbrace$ is a \emph{robust partition of $G$ with parameters $k,\ell$} if it is a partition of $V(G)$ such that $G[V_i]$ is a robust expander component for all $1 \leq i \leq k$ and $G[W_j]$ is a bipartite robust expander component for all $1 \leq j \leq \ell$.
In~\cite{klos} we proved the following:

\begin{itemize}
\item[($\star$)] For all $r \in \mathbb{N}$ and $\eps > 0$ and $n$ sufficiently large, every $D$-regular graph $G$ on $n$ vertices with $D \geq (\frac{1}{r+1}+\eps)n$ has a robust partition with parameters $k,\ell$, where $k+2\ell \leq r$.
\end{itemize}

\noindent
In particular, the number of edges between robust components is $o(n^2)$ (see
Theorem~\ref{structure} for the precise statement).

\subsection{Finding a Hamilton cycle using a robust partition}

Now suppose that $G$ is a $D$-regular graph on $n$ vertices with $D \geq n/4$, where $n$ is sufficiently large.
Then ($\star$) applied with $r=4$ implies that $G$ has a robust partition $\mathcal{V}$ with parameters $k,\ell$, where $k+2\ell \leq 4$.
This gives eight possible structures, parametrised by $(k,\ell) \in S_{\leq 3} \cup S_4$, where
$$
S_{\leq 3} := \lbrace (1,0), (2,0), (3,0), (0,1), (1,1) \rbrace\ \ \mbox{ and }S_4 := \lbrace (4,0), (0,2), (2,1) \rbrace.
$$
Note that the extremal example in Figure~\ref{fig:exactex}(i) corresponds to the case $(2,1)$ and the one in (ii) corresponds to the case $(3,0)$.
Also note that when $D \geq (1/4 + \eps)n$, we have $k+2\ell \leq 3$ and so $(k,\ell) \in S_{\leq 3}$.
In~\cite{klos}, we proved that if $G$ is $3$-connected and has a robust partition $\mathcal{V}$ with parameters $k,\ell$ where $(k,\ell) \in S_{\leq 3}$, then $G$ is Hamiltonian.
In particular, this implies an approximate version of Theorem~\ref{exact}.
The proof proceeded by considering each possible structure separately.
Therefore, to prove Theorem~\ref{exact}, it remains to show that 
if $G$ is $3$-connected and has a robust partition $\mathcal{V}$ with parameters $k,\ell$ where $(k,\ell) \in S_4$, then $G$ is Hamiltonian (see Theorem~\ref{stability}).
So the current paper does not supersede our previous result but rather uses it as an essential ingredient.
Again, we consider each structure separately in Sections~\ref{sec40},~\ref{sec02} and~\ref{sec21} respectively.

In each case we adopt the following strategy.
Let $\mathcal{V}$ be a robust partition of $G$ with parameters $k,\ell$.
K\"uhn, Osthus and Treglown~\cite{kot} proved that every large robust expander $H$ with linear minimum degree contains a Hamilton cycle.
This can be strengthened (see~\cite{klos}) to show that one can cover all the vertices of a robust expander with a set of paths with prescribed endvertices.
More precisely, one can show that each robust expander component $G[V_i]$ is Hamilton $p$-linked for each small $p$ and all $1 \leq i \leq k$.
(Here a graph $H$ is \emph{Hamilton $p$-linked} if, whenever $X := \lbrace x_1, y_1, \ldots, x_p, y_p \rbrace$ is a collection of distinct vertices, there exist vertex-disjoint paths
$P_1, \ldots, P_p$ such that $P_j$ connects $x_j$ to $y_j$, and such that together the paths $P_1, \ldots, P_p$ cover all vertices of $H$.)
Balanced bipartite robust expanders have the same property, provided $X$ is distributed equally between the bipartition classes.
This means that we can hope to reduce the problem of finding a Hamilton cycle in $G$ to finding a suitable set of \emph{external edges} $E_{\rm ext}$, where an edge is external if it has endpoints in different members of $\mathcal{V}$.
We then apply the Hamilton $p$-linked property to each robust component to join up the external edges into a Hamilton cycle.
The assumption of $3$-connectivity is crucial for finding $E_{\rm ext}$.

However, several problems arise.
When $(k,\ell)=(4,0)$, we have four robust components and only the assumption of $3$-connectivity, which makes it difficult to find a suitable set $E_{\rm ext}$ joining all four components directly.
However, we can appeal to the dominating cycle result in~\cite{jlz} mentioned in the introduction, giving us a fairly short argument for this case.
Note that the condition that $D \geq n/4$ is essential in this case --- $3$-connectivity on its own is not sufficient.%
\COMMENT{Probably $D \geq n/4-1$ or something a little smaller than $n/4$ is essential.}

Now suppose that $\ell \geq 1$, i.e.~$\mathcal{V}$ contains a bipartite robust expander component.
These cases are challenging since a bipartite graph does not contain a Hamilton cycle if it is not balanced.
So as well as a suitable set $E_{\rm ext}$, we need to find a set $E_{\rm bal}$ of \emph{balancing edges} incident to the bipartite robust expander component.
Suppose for example that $(k,\ell)=(0,2)$ and $G$ consists of two bipartite robust expander components $W_1,W_2$ such that $W_i$ has vertex classes $A_i,B_i$ where $|A_1|=|B_1|$ and $|A_2|=|B_2|+1$.
Then we could choose $E_{\rm bal}$ to be a single edge with both endpoints in $A_2$.
A second example would be $E_{\rm bal} = \lbrace a_1a_2,b_1a_2' \rbrace$ where $a_1 \in A_1$, $b_1 \in B_1$ and $a_2,a_2' \in A_2$ are distinct. 
(Note that these are also external edges and in this case we can actually take $E_{\rm ext} \cup E_{\rm bal} = \lbrace a_1a_2,b_1a_2' \rbrace$.)%
\COMMENT{I added this example to explain what balancing edges are.}
Observe that we need at least $\bigg||A_1|-|B_1|\bigg|+\bigg||A_2|-|B_2|\bigg|$ balancing edges.%
\COMMENT{NEW big modulus}

Our robust partition guarantees that the vertex classes of any bipartite robust expander component differ by at most $o(n)$, so we must potentially find a similar number of balancing edges.
This must be done in such a way that $\mathcal{P} := E_{\rm ext} \cup E_{\rm bal}$ can be extended into a Hamilton cycle.
So in particular $\mathcal{P}$ must be a collection of vertex-disjoint paths.
We use the Hamilton $p$-linkedness of the (bipartite) robust expander components to find these edges which extend $\mathcal{P}$ into a Hamilton cycle.
Consider the second example above, with $\mathcal{P} = \lbrace a_1a_2,b_1a_2' \rbrace$.
Choose a neighbour $b_2$ of $a_2$ in $B_2$ and let $\mathcal{P}' := \lbrace a_1a_2b_2,b_1a_2' \rbrace$.
Then the Hamilton $1$-linkedness of $W_1,W_2$ implies that we can find a path $P_1$ with endpoints $a_1,b_1$ which spans $W_1$, and a path $P_2$ with endpoints $a_2',b_2$ which spans $W_2 \setminus \lbrace a_2 \rbrace$.
Then the edges of $P_1,P_2,\mathcal{P}'$ together form a Hamilton cycle.%
\COMMENT{continuing the previous example (optional?). A picky reader might not like the fact we use the Hamilton linkedness of $W_2$ not $W_2 \setminus \lbrace a_2 \rbrace$.}

It turns out that the condition that $D \geq n/4$ is crucial in the case when $(k,\ell)=(2,1)$ (see Section~\ref{example}) but its full strength is not required in the case when $(k,\ell)=(0,2)$.%
\COMMENT{We use $3$-connectivity in all cases. It may not be necessary when $(k,\ell)=(0,2)$ but we do use it. Should I mention this?}
A sketch of the proof in each of the three cases can be found at the beginning of Sections~\ref{sec40},~\ref{sec02} and~\ref{sec21} respectively.

\section{Notation, definitions and general tools}\label{prelims}

\subsection{General notation}\label{notation}

Given a graph $G$ and $X \subseteq V(G)$, complements are always taken within $G$, so that $\overline{X} := V(G) \setminus X$.
We write $G \setminus X$ to mean $G[V(G) \setminus X]$.
Given $H \subseteq V(G)$, we write $G \setminus E(H)$ for the graph with vertex set $V(G)$ and edge set $E(G) \setminus E(H)$.
We write $N(X) := \bigcup_{x \in X}{N(x)}$.
Given $x \in V(G)$ and $Y \subseteq V(G)$ we write $d_Y(x)$ for the number of edges $xy$ with $y \in Y$.

If $S,T$ are sets of vertices which are not necessarily disjoint and may not be subsets of $V(G)$, we write
$e_G(S)$ for the number of edges of $G$ with both endpoints in $S$, and $e_G(S,T)$ for the number of $ST$-edges of~$G$, i.e.~for the number of all edges with one endpoint
in $S$ and the other endpoint in $T$. Moreover, we set $G[S] := G[S \cap V(G)]$ and write
$G[S,T]$ for the bipartite graph with vertex classes $S\cap V(G)$, $T\cap V(G)$ whose edge set consists of all the $ST$-edges of $G$.
We omit the subscript $G$ whenever the graph $G$ is clear from the context.

Given%
\COMMENT{NEW: previously disjoint subsets}
subsets $X,Y$ of $V(G)$, we say that $P$ is an \emph{$XY$-path} if $P$ has one 
endpoint in $X$ and one endpoint in $Y$.
We call a vertex-disjoint collection of non-trivial paths a \emph{path system}.
We will often think of a path system $\mathcal{P}$ as a graph with edge set $\bigcup_{P \in \mathcal{P}}E(P)$, so that
e.g.~$V(\mathcal{P})$ is the union of the vertex sets of each path in $\mathcal{P}$, and $e_{\mathcal{P}}(X)$ denotes the number of edges on the paths in $\mathcal{P}$ having both endpoints in $X$, and $e_{\mathcal{P}}(X,Y)$ denotes the number of $XY$-edges in paths of $\mathcal{P}$.%
\COMMENT{I wrote `endpoints' here rather than `endvertices' since that is what we use previously. NEW def}
By slightly abusing notation, given two vertex sets $S$ and $T$ and a path system $\mathcal{P}$,
we write $\mathcal{P}[S]$ for the graph obtained from $\mathcal{P}[S]$ by deleting isolated vertices and define $\mathcal{P}[S,T]$
similarly.%
    \COMMENT{DK changed this sentence}
We say that a vertex $x$ is an \emph{endpoint} of $\mathcal{P}$ if $x$ is an endpoint of some path in $\mathcal{P}$. 
An \emph{Euler tour} in a (multi)graph is a closed walk that uses each edge exactly once.%
\COMMENT{I have removed the notions of $\mathcal{U}$-anchored and $\mathcal{U}$-extension.}

We write $\mathbb{N}$ for the set of positive integers and write $\mathbb{N}_0 := \mathbb{N} \cup \lbrace 0 \rbrace$.
$\mathbb{R}_{\geq 0}$ denotes the set of non-negative reals.
Throughout we will omit floors and ceilings where the argument is unaffected. The constants in the hierarchies used to state our results are chosen from right to left.
For example, if we claim that a result holds whenever $0<1/n\ll a\ll b\ll c\le 1$ (where $n$ is the order of the graph), then 
there is a non-decreasing function $f:(0,1]\to (0,1]$ such that the result holds
for all $0<a,b,c\le 1$ and all $n\in \mathbb{N}$ with $b\le f(c)$, $a\le f(b)$ and $1/n\le f(a)$. 
Hierarchies with more constants are defined in a similar way. 
Given $0 < \eps < 1$ and $x \in \mathbb{R}$, we write
$\lceil x \rceil_{\eps} := \lceil x - \eps \rceil$.


\subsection{Robust partitions of regular graphs} \label{sec:struct}

In this section we list the definitions which are required to state the structural result on dense regular graphs (Theorem~\ref{structure}) which is the main tool in our proof.
As already indicated in Section~\ref{sketch}, this involves the concept of `robust expansion'.

Given a graph $G$ on $n$ vertices, $0 < \nu < 1$ and $S \subseteq V(G)$, we define the \emph{$\nu$-robust neighbourhood} $RN_{\nu, G}(S)$  of $S$ to be the set of all those vertices with at least $\nu n$ neighbours in $S$.
Given $0 < \nu \leq \tau < 1$,
we say that $G$ is a \emph{robust $(\nu, \tau)$-expander} if, for all sets $S$ of vertices satisfying $\tau n \leq |S| \leq (1-\tau)n$, we have that $|RN_{\nu, G}(S) | \geq |S| + \nu n$.
For $S \subseteq X  \subseteq V(G)$ we write $RN_{\nu,X}(S) := RN_{\nu,G[X]}(S)$.

The next lemma (Lemma 4.8 in~\cite{klos}) states that robust expanders are indeed robust, in the sense that the expansion property cannot be destroyed by adding or removing a small number of vertices.

\begin{lemma} \label{expanderswallow}
Let $0 < \nu \ll \tau \ll 1$.
Suppose that $G$ is a graph and $U,U' \subseteq V(G)$ are such that $G[U]$ is a robust $(\nu,\tau)$-expander and $|U \triangle U'| \leq \nu|U|/2$.
Then $G[U']$ is a robust $(\nu/2,2\tau)$-expander.
\end{lemma}

We now introduce the concept of `bipartite robust expansion'.
Let $0 < \nu \leq \tau < 1$. Suppose that $H$ is a (not necessarily bipartite) graph on $n$ vertices and that $A,B$ is a partition of $V(H)$.
We say that $H$ is a \emph{bipartite robust} $(\nu, \tau)$-\emph{expander with bipartition $A,B$} if every $S \subseteq A$ with $\tau |A| \leq |S| \leq (1-\tau)|A|$ satisfies $|RN_{\nu, H}(S) \cap B| \geq |S| + \nu n$.
Note that the order of $A$ and $B$ matters here.
We do not mention the bipartition if it is clear from the context.

Note that for $0 < \nu' \leq \nu \leq \tau \leq \tau' < 1$, any robust $(\nu,\tau)$-expander is also a robust $(\nu',\tau')$-expander (and the analogue holds in the bipartite case).

Given $0 < \rho < 1$, we say that $U \subseteq V(G)$ is a \emph{$\rho$-component} of a graph $G$ on $n$ vertices if $|U| \geq \sqrt{\rho}n$ and $e(U,\overline{U}) \leq \rho n^2$.
We will need the following simple observation (Lemma 4.1 in~\cite{klos}) about $\rho$-components.

\begin{lemma} \label{comp}
Let $n,D \in \mathbb{N}$ and $\rho > 0$. Let $G$ be a $D$-regular graph on $n$ vertices and let $U$ be a $\rho$-component of $G$. Then $|U| \geq D - \sqrt{\rho}n$.
\end{lemma} 

Suppose that $G$ is a graph on $n$ vertices
and that $U \subseteq V(G)$.
We say that $G[U]$ is \emph{$\rho$-close to bipartite (with bipartition $U_1, U_2$)} if 
\begin{itemize}
\item[(C1)] $U$ is the union of two disjoint sets $U_1$ and $U_2$ with $|U_1|,|U_2| \geq \sqrt{\rho}n$;
\item[(C2)] $\bigg||U_1| - |U_2|\bigg| \leq \rho n$;
\item[(C3)] $e(U_1, \overline{U_2}) + e(U_2,\overline{U_1}) \leq \rho n^2$.
\end{itemize}

\noindent
(Recall that $\overline{U_1} = V(G) \setminus U_1$ and similarly for $U_2$.) 
Note that (C1) and (C3) together imply that $U$ is a $\rho$-component.
Suppose that $G$ is a graph on $n$ vertices
and that $U \subseteq V(G)$.
Let $0 < \rho \leq \nu \leq \tau < 1$.
We say that $G[U]$ is a \emph{$(\rho,\nu,\tau)$-robust expander component of $G$} if
\begin{itemize}
\item[(E1)] $U$ is a $\rho$-component;
\item[(E2)] $G[U]$ is a robust $(\nu,\tau)$-expander.
\end{itemize}
We say that $G[U]$ is a \emph{bipartite $(\rho,\nu,\tau)$-robust expander component (with bipartition $A,B$) of $G$} if
\begin{itemize}
\item[(B1)] $G[U]$ is $\rho$-close to bipartite with bipartition $A,B$;
\item[(B2)] $G[U]$ is a bipartite robust $(\nu,\tau)$-expander with bipartition $A,B$.
\end{itemize}
We say that $U$ is a \emph{$(\rho,\nu,\tau)$-robust component} if it is either a $(\rho,\nu,\tau)$-robust expander component or a bipartite $(\rho,\nu,\tau)$-robust expander component.

One can show that, after adding and removing a small number of vertices, a bipartite robust expander component is still a bipartite robust expander component, with slightly weaker parameters.
This appears as Lemma 4.10 in~\cite{klos} and the proof may be found in~\cite{thesis}.%
    \COMMENT{DK replaced 4.12 with 4.10}

\begin{lemma} \label{BREadjust}
Let $0 < 1/n \ll \rho \leq \gamma \ll \nu \ll \tau \ll \alpha < 1$ and suppose that $G$ is a $D$-regular graph on $n$ vertices where $D \geq \alpha n$.
Suppose that $G[A \cup B]$ is a bipartite $(\rho,\nu,\tau)$-robust expander component of $G$ with bipartition $A,B$. Let $A',B' \subseteq V(G)$ be such that $|A \triangle A'| + |B \triangle B'| \leq \gamma n$.
Then $G[A' \cup B']$ is a bipartite $(3\gamma,\nu/2,2\tau)$-robust expander component of $G$ with bipartition $A',B'$.
\end{lemma}

Let $k,\ell,D \in \mathbb{N}_0$ and $0 < \rho \leq \nu \leq \tau < 1$.
Given a $D$-regular graph $G$ on $n$ vertices, we say that $\mathcal{V}$ is a \emph{robust partition of $G$ with parameters $\rho,\nu,\tau,k,\ell$} if the following conditions hold.
\begin{itemize}
\item[(D1)] $\mathcal{V} = \lbrace V_1, \ldots, V_k, W_1, \ldots, W_\ell \rbrace$ is a partition of $V(G)$;
\item[(D2)] for all $1 \leq i \leq k$, $G[V_i]$ is a $(\rho,\nu,\tau)$-robust expander component of $G$;
\item[(D3)] for all $1 \leq j \leq \ell$, there exists a partition $A_j,B_j$ of $W_j$ such that $G[W_j]$ is a bipartite $(\rho,\nu,\tau)$-robust expander component with bipartition $A_j,B_j$;
\item[(D4)] for all $X,X' \in \mathcal{V}$ and all $x \in X$, we have $d_{X}(x) \geq d_{X'}(x)$. In particular, $d_X(x) \geq D/m$, where $m := k+\ell$;
\item[(D5)] for all $1 \leq j \leq \ell$ we have $d_{B_j}(u) \geq d_{A_j}(u)$ for all $u \in A_j$ and $d_{A_j}(v) \geq d_{B_j}(v)$ for all $v \in B_j$; in particular, $\delta(G[A_j,B_j]) \geq D/2m$;
\item[(D6)] $k + 2\ell \leq \left\lfloor (1+\rho^{1/3})n/D \right\rfloor$;
\item[(D7)] for all $X \in \mathcal{V}$, all but at most $\rho n$ vertices $x \in X$ satisfy $d_X(x) \geq D - \rho n$.
\end{itemize}
Note that (D7) implies that $|X| \geq D-\rho n$ for all $X \in \mathcal{V}$. 

The following structural result (Theorem 3.1 in~\cite{klos}) is our main tool.
It states that any dense regular graph has a remarkably simple structure: a partition into a small number of (bipartite) robust expander components.

\begin{theorem} \label{structure}
For all $\alpha,\tau > 0$ and every non-decreasing function $f: (0,1) \rightarrow (0,1)$, there exists $n_0 \in \mathbb{N}$ such that the following holds.
For all $D$-regular graphs $G$ on $n \geq n_0$ vertices where $D \geq \alpha n$, there exist $\rho,\nu$ with $1/n_0 \leq \rho \leq \nu \leq \tau$; $\rho \leq f(\nu)$ and $1/n_0 \leq f(\rho)$, and $k,\ell \in \mathbb{N}$ such that $G$ has a robust partition $\mathcal{V}$ with parameters $\rho,\nu,\tau,k,\ell$.
\end{theorem}

Let $k,\ell \in \mathbb{N}_0$ and $0 < \rho \leq \nu \leq \tau \leq \eta < 1$.
Given a graph $G$ on $n$ vertices, we say that $\mathcal{U}$ is a \emph{weak robust partition of $G$ with parameters $\rho,\nu,\tau,\eta,k,\ell$} if the following conditions hold.%
\COMMENT{no need to assume regularity. I have changed the definition to a \emph{weak robust partition of $G$} instead of a \emph{weak robust subpartition in $G$} as this is what we always have here.}
\begin{itemize}
\item[(D1$'$)] $\mathcal{U} = \lbrace U_1, \ldots, U_{k}, Z_1, \ldots, Z_{\ell} \rbrace$ is a partition of $V(G)$;
\item[(D2$'$)] for all $1 \leq i \leq k$, $G[U_i]$ is a $(\rho,\nu,\tau)$-robust expander component of $G$;
\item[(D3$'$)] for all $1 \leq j \leq \ell$, there exists a partition $A_j, B_j$ of $Z_j$ such that $G[Z_j]$ is a bipartite $(\rho,\nu,\tau)$-robust expander component with bipartition $A_j,B_j$;
\item[(D4$'$)] $\delta(G[X]) \geq \eta n$ for all $X \in \mathcal{U}$;
\item[(D5$'$)] for all $1 \leq j \leq \ell$, we have $\delta(G[A_j,B_j]) \geq \eta n/2$.
\end{itemize}
Using Lemma~\ref{comp} it is easy to check that
whenever $\rho\le \rho'\le \nu$ and $G$ is a $D$-regular graph on $n$ vertices with $D\ge 5\sqrt{\rho'}n$,
then any weak robust partition of $G$ with parameters $\rho,\nu,\tau,\eta, k,\ell$
is also a weak robust partition with parameters $\rho',\nu,\tau,\eta,k,\ell$. A similar statement holds for robust partitions.%
    \COMMENT{DK: new sentences, need $D\ge 5\sqrt{\rho'}n$ instead of $D\ge 2\sqrt{\rho'}n$ to check (C1)}

A weak robust partition $\mathcal{U}$ is weaker than a robust partition in the sense that the graph is not necessarily regular, and we can make small adjustments to the partition while still maintaining (D1$'$)--(D5$'$) with slightly worse parameters.
It is not hard to show the following (Proposition 5.1 in~\cite{klos}).

\begin{proposition} \label{WRSD-RD}
Let $k,\ell,D \in \mathbb{N}_0$ and suppose that $0 < 1/n \ll \rho \leq \nu \leq \tau \le \eta \le  \alpha^2/2 < 1$.
Suppose that $G$ is a $D$-regular graph on $n$ vertices where $D \geq \alpha n$. Let $\mathcal{V}$ be a robust partition of $G$ with parameters $\rho,\nu,\tau,k,\ell$. Then $\mathcal{V}$ is a weak robust partition of $G$ with parameters $\rho,\nu,\tau,\eta,k,\ell$.
\end{proposition}

We also proved the following stability result (Theorem 6.11 in~\cite{klos}).
This implies that any sufficiently large $3$-connected regular graph $G$ on $n$ vertices with degree at least a little larger than $n/5$ is either Hamiltonian, or has one of three very specific structures.

\begin{theorem} \label{stability}
For every $\eps,\tau > 0$ with $2\tau^{1/3} \leq \eps$ and every non-decreasing function $g : (0,1) \rightarrow (0,1)$, there exists $n_0 \in \mathbb{N}$ such that the following holds.
For all $3$-connected $D$-regular graphs $G$ on $n \geq n_0$ vertices where $D \geq (1/5 + \eps) n$, at least one of the following holds:
\begin{itemize}
\item[(i)] $G$ has a Hamilton cycle;
\item[(ii)] $D < (1/4 + \eps)n$ and there exist $\rho,\nu$ with $1/n_0 \leq \rho \leq \nu \leq \tau$; $1/n_0 \leq g(\rho)$; $\rho \leq g(\nu)$, and $(k,\ell) \in \lbrace (4,0), (2,1), (0,2) \rbrace$ such that $G$ has a robust partition $\mathcal{V}$ with parameters $\rho,\nu,\tau,k,\ell$.
\end{itemize}
\end{theorem}

\subsection{Path systems and $\mathcal{V}$-tours}\label{pathsystems}

Here we state some useful tools concerning path systems that we will need in our proof.
All of these were proved in~\cite{klos}.

A simple double-counting argument gives the following proposition (Proposition 6.4 in~\cite{klos}).
We use it to guarantee the existence of edges in certain parts within a regular graph.

\begin{proposition} \label{fact2}
Let $G$ be a $D$-regular graph with vertex partition $A,B,U$.
Then
\begin{itemize}
\item[(i)]
$
2(e(A) - e(B)) + e(A,U) - e(B,U) = (|A| - |B|)D.
$
\end{itemize}
In particular,
\begin{itemize}
\item[(ii)]
$
2e(A) + e(A,U) \geq (|A|-|B|)D.
$
\end{itemize}
\end{proposition}

Suppose that $G$ is a graph containing a path system $\mathcal{P}$, and that $\mathcal{V}$ is a partition of $V(G)$.
We define the \emph{reduced multigraph $R_{\mathcal{V}}(\mathcal{P})$ of $\mathcal{P}$ with respect to $\mathcal{V}$} to be the multigraph with vertex set $\mathcal{V}$ in which we add a distinct edge between $X,X' \in \mathcal{V}$ for every path in $\mathcal{P}$ with one endpoint in $X$ and one endpoint in $X'$.
So $R_{\mathcal{V}}(\mathcal{P})$ might contain loops and multiple edges.

Given a graph $G$ containing a path system $\mathcal{P}$, and $A \subseteq V(G)$, we write
\begin{equation}\label{F}
F_{\mathcal{P}}(A) := (a_1,a_2)
\end{equation}
when $a_i$ is the number of vertices in $A$ of degree $i$ in $\mathcal{P}$ for $i=1,2$.%
\COMMENT{`F' for forbidden. This is a measure of how hard it is to extend $\mathcal{P}$ using edges in $A$ (so it's still an Euler tour, etc.}
Note that, if $e_{\mathcal{P}}(A) = 0$, then%
\COMMENT{So if this quantity is big, it is hard to add edges in $A$. But we need fewer such edges because the contribution from $\mathcal{P}$ itself is greater.}
\begin{equation}\label{edgecount}
e_{\mathcal{P}}(A,\overline{A}) = a_1+2a_2.
\end{equation}

\noindent
The following lemma (Lemma 6.3 in~\cite{klos}) is used in the case $(k,\ell)=(4,0)$.
An extension (Proposition~\ref{BC}) is used in the case $(k,\ell)=(2,1)$.

\begin{lemma}\label{cliquetour}
Let $G$ be a $3$-connected graph and let $\mathcal{V}$ be a partition of $V(G)$ into at most three parts, where $|V| \geq 3$ for each $V \in \mathcal{V}$.
Then $G$ contains a path system $\mathcal{P}$ such that
\begin{itemize}
\item[(i)] $e(\mathcal{P}) \leq 4$ and $\mathcal{P} \subseteq \bigcup_{V \in \mathcal{V}}G[V,\overline{V}]$;
\item[(ii)] $R_{\mathcal{V}}(\mathcal{P})$ has an Euler tour;
\item[(iii)] for each $V \in \mathcal{V}$, if $F_{\mathcal{P}}(V)=(c_1,c_2)$, then $c_1+2c_2 \in \lbrace 2,4 \rbrace$ and $c_2 \leq 1$.
\end{itemize}
\end{lemma}

Let $k,\ell \in \mathbb{N}_0$, let $0 < \rho \leq \nu \leq \tau \leq \eta < 1$ and let $0 < \gamma < 1$.
Suppose that $G$ is a graph on $n$ vertices with a weak robust partition $\mathcal{V} = \lbrace V_1, \ldots, V_k, W_1, \ldots, W_\ell \rbrace$ with parameters $\rho,\nu,\tau,\eta,k,\ell$, so that the bipartition of $W_j$ specified by (D3$'$) is $A_j, B_j$.
We say that a path system $\mathcal{P}$ is a \emph{$\mathcal{V}$-tour with parameter $\gamma$} if
\begin{itemize}
\item $R_{\mathcal{V}}(\mathcal{P})$ has an Euler tour;
\item for all $X \in \mathcal{V}$ we have $|V(\mathcal{P}) \cap X| \leq \gamma n$;
\item for all $1 \leq j \leq \ell$ we have $|A_j\setminus V(\mathcal{P})|=|B_j\setminus V(\mathcal{P})|$.%
   \COMMENT{DK: previously had "$A_j,B_j$ contain the same number of vertices of $\mathcal{P}$."}
Moreover, $A_j,B_j$ contain the same number of endpoints of $\mathcal{P}$ and this number is positive.
\end{itemize}
We will often think of $R_{\mathcal{V}}(\mathcal{P})$ as a walk rather than a multigraph.%
\COMMENT{I don't think there's any reason to define $(A,B)$-balanced (and certainly not ${\rm End}_{\mathcal{P}}(U), {\rm Int}_{\mathcal{P}}(U)$) since they are only needed to state the definition of a $\mathcal{V}$-tour.}
So in particular, we will often say that `$R_{\mathcal{V}}(\mathcal{P})$ is an Euler tour'.

We will use the following lemma (a special case of Lemma 6.8 in~\cite{klos})%
\COMMENT{there we had a weak robust subpartition}
to extend a path system into one that satisfies the third property above for all $A,B$ forming a bipartite robust expander component.

\begin{lemma} \label{balextend}
Let $n,k,\ell \in \mathbb{N}_0$ and $0 < 1/n \ll \rho \ll \nu \ll \tau \ll \eta < 1$.%
\COMMENT{$\nu$ and $\tau$ are superfluous and only needed to define the WRSP} Let $G$ be a graph on $n$ vertices and suppose that $\mathcal{V} := \lbrace V_1, \ldots, V_k, W_1, \ldots, W_\ell \rbrace$ is a weak robust partition of $G$ with parameters $\rho,\nu,\tau,\eta,k,\ell$.
For each $1 \leq j \leq \ell$, let $A_j,B_j$ be the bipartition of $W_j$ specified by \emph{(D3$'$)}.
Let $\mathcal{P}$ be a path system such that for each $1 \leq j \leq \ell$, 
\begin{equation} \label{bal2}
2e_{\mathcal{P}}(A_j) - 2e_{\mathcal{P}}(B_j) + e_{\mathcal{P}}(A_j,\overline{W_j}) - e_{\mathcal{P}}(B_j,\overline{W_j}) = 2(|A_j|-|B_j|).
\end{equation}
Suppose further that $|V(\mathcal{P}) \cap X| \leq \rho n$ for all $X \in \mathcal{V}$, and that $R_{\mathcal{V}}(\mathcal{P})$ is an Euler tour.
Then $G$ contains a path system $\mathcal{P}'$ that is a $\mathcal{V}$-tour with parameter $9\rho$.
\end{lemma}

The last result of this section (a special case of Lemma 5.2 in~\cite{klos})%
\COMMENT{there we had a weak robust subpartition}
says that, in order to find a Hamilton cycle, it is sufficient to find a $\mathcal{V}$-tour.

\begin{lemma} \label{HES}
Let $k,\ell,n \in \mathbb{N}_0$ and suppose that $0 < 1/n \ll \rho, \gamma \ll \nu \ll \tau \ll \eta < 1$.
Suppose that $G$ is a graph on $n$ vertices and that $\mathcal{V}$ is a weak robust partition of $G$ with parameters $\rho,\nu,\tau,\eta,k,\ell$.
Suppose further that $G$ contains a $\mathcal{V}$-tour $\mathcal{P}$ with parameter $\gamma$.
Then $G$ contains a Hamilton cycle.
\end{lemma}

\section{(4,0): Four robust expander components}\label{sec40}

The aim of this section is to prove the following lemma.

\begin{lemma} \label{(4,0)}
Let $D, n \in \mathbb{N}$ and $0 < 1/n \ll \rho \ll \nu \ll \tau \ll 1$.
Suppose that $G$ is a $3$-connected $D$-regular graph on $n$ vertices with $D \geq n/4$. 
Suppose further that $G$ has a robust partition $\mathcal{V}$ with parameters $\rho,\nu,\tau,4,0$.
Then $G$ contains a $\mathcal{V}$-tour with parameter $33/n$.
\end{lemma}

We will find a $\mathcal{V}$-tour $\mathcal{P}$ as follows.
Let $\mathcal{V} := \lbrace V_1, \ldots, V_4 \rbrace$.
Suppose that there are $1 \leq i < j \leq 4$ such that $G[V_i,V_j]$ contains a large matching $M$.
We can use $3$-connectivity with the tripartition $\mathcal{V}' := \mathcal{V} \cup \lbrace V_i \cup V_j \rbrace \setminus \lbrace V_i, V_j \rbrace$ to obtain a path system $\mathcal{P}'$ such that $R_{\mathcal{V}'}(\mathcal{P}')$ is a $\mathcal{V}'$-tour.
Then $\mathcal{P}'$ together with some suitable edges of $M$ will form a $\mathcal{V}$-tour.

Suppose instead that for all $1 \leq i < j \leq 4$, every matching in $G[V_i,V_j]$ is small.
In this case, we appeal to the result of Jackson, Li and Zhu~\cite{jlz} mentioned in the introduction: any longest cycle in $G$ is dominating.
Thus $C$ visits all the $V_i$.
Moreover, since there are very few edges between the $V_i$ it follows that most of the edges of $C$ lie within some $V_i$.
If we remove all such edges, what remains is a $\mathcal{V}$-tour.%

\medskip
Let $\mathcal{V}'$ be a partition of $V(G)$ into three parts such that $\mathcal{V}$ is a refinement of $\mathcal{V}'$.
Then, by Lemma~\ref{cliquetour}, we can easily find a collection of paths $\mathcal{P}'$ such that $R_{\mathcal{V}'}(\mathcal{P}')$ is an Euler tour.
The following result will enable us to `extend' $\mathcal{P}'$ into $\mathcal{P}$ such that $R_{\mathcal{V}}(\mathcal{P})$ is an Euler tour.

\begin{proposition} \label{plusmatching}
Let $\mathcal{U}$ be a partition of $V(G)$.
Let $U,V \in \mathcal{U}$ and let $\mathcal{U}' := \mathcal{U} \cup \lbrace U \cup V \rbrace \setminus \lbrace U,V \rbrace$.
Suppose that $G$ contains a path system $\mathcal{P}'$ such that $R_{\mathcal{U}'}(\mathcal{P}')$ is an Euler tour.
Suppose further that $G[U,V]$ contains a matching $M$ of size at least $|V(\mathcal{P}') \cap (U \cup V)| +2$.
Then $G$ contains a path system $\mathcal{P}$ with $E(\mathcal{P}) \supseteq E(\mathcal{P}')$ such that $R_{\mathcal{U}}(\mathcal{P})$ is an Euler tour and $|V(\mathcal{P}) \cap X| \leq |V(\mathcal{P}') \cap X| + 2$ for all $X \in \mathcal{U}$.
\end{proposition}

\begin{proof}
Note that there are at least two edges $e,e'$ of $M$ which are vertex-disjoint from $\mathcal{P}'$.
Let $R' := R_{\mathcal{U}}(\mathcal{P}')$ and $R'' := R_{\mathcal{U}'}(\mathcal{P}')$.
We have that
$d_{R'}(U) + d_{R'}(V) = d_{R''}(U \cup V)$ is even since $R''$ is an Euler tour.
Moreover, $d_{R'}(X)=d_{R''}(X)$ for all $X \in \mathcal{U}' \cap \mathcal{U}$.

If both $d_{R'}(U)$ and $d_{R'}(V)$ are odd,
let $\mathcal{P} := \mathcal{P}' \cup \lbrace e \rbrace$.%
\COMMENT{Let $R := R_{\mathcal{U}}(\mathcal{P})$.
Then $d_R(X) = d_{R'}(X)$ for all $X \in \mathcal{U}' \cap \mathcal{U}$. Moreover $d_R(U) = d_{R'}(U)+1$, and similarly for $V$.
Therefore every vertex in $R$ has even positive degree.
Furthermore, $R''$ is connected and so $R$ is connected.
Note that $U$ might be isolated in $R'$.
Therefore $R$ is an Euler tour.
}
Otherwise, both $d_{R'}(U)$ and $d_{R'}(V)$ are even (but one could be zero).
In this case, let $\mathcal{P} := \mathcal{P}' \cup \lbrace e,e' \rbrace$.%
\COMMENT{
Then $d_R(X) = d_{R'}(X)$ for all $X \in \mathcal{U}' \cap \mathcal{U}$. Moreover $d_R(U) = d_{R'}(U)+2$, and similarly for $V$.
The same reasoning as above implies that $R$ is an Euler tour.
We also have that $\Int_{\mathcal{P}}(X) = \Int_{\mathcal{P}'}(X)$ for all $X \in \mathcal{U}$, and $\Delta(R) \leq \Delta(R)+2$.
}
It is straightforward to check that in both cases $R_{\mathcal{U}}(\mathcal{P})$ is an Euler tour.
\end{proof}

A subgraph $H$ of a graph $G$ is said to be \emph{dominating} if $G \setminus V(H)$ is an independent set. 
In our proof of Lemma~\ref{(4,0)} we will use the following theorem of Jackson, Li and Zhu.

\begin{theorem} \label{jlzdom}\cite{jlz}
Let $G$ be a $3$-connected $D$-regular graph on $n$ vertices with $D \geq n/4$.
Then any longest cycle in $C$ is dominating.
\end{theorem}

\medskip
\noindent
\emph{Proof of Lemma~\emph{\ref{(4,0)}}.}
Let $C$ be a longest cycle in $G$.
Then Theorem~\ref{jlzdom} implies that $C$ is dominating.
We consider two cases according to the number of edges in $C$ between classes of $\mathcal{V}$.

\medskip
\noindent
\textbf{Case 1.}
\emph{$e_C(U,V) \geq 12$ for some distinct $U,V \in \mathcal{V}$.}

\medskip
\noindent
Since $C$ is a cycle we have that $\Delta(C[U,V]) \leq 2$.
K\"onig's theorem implies that $C[U,V]$ has a proper edge-colouring with at most two colours, and thus $C[U,V]$ contains a matching of size at least $e_C(U,V)/2 \geq 6$.

Let $\mathcal{V}' := \mathcal{V} \cup \lbrace U \cup V \rbrace \setminus \lbrace U,V \rbrace$.
So $\mathcal{V}'$ is a tripartition of $V(G)$, and certainly $|V| \geq 3$ for each $V \in \mathcal{V}'$.
Apply Lemma~\ref{cliquetour}
to obtain a path system $\mathcal{P}'$ in $G$ such that the consequences (i)--(iii) hold.%
\COMMENT{NEW `consequences of' everywhere...}
Then $R_{\mathcal{V}'}(\mathcal{P}')$ is an Euler tour and (iii) implies that $|V(\mathcal{P}') \cap X| \leq 4$ for all $X \in \mathcal{V}'$.

Now Proposition~\ref{plusmatching} with $\mathcal{V}, \mathcal{V}'$ playing the roles of $\mathcal{U}, \mathcal{U}'$ implies that $G$ contains a path system $\mathcal{P}$ such that $R_{\mathcal{V}}(\mathcal{P})$ is an Euler tour, and $|V(\mathcal{P}) \cap X| \leq 6$ for all $X \in \mathcal{V}$.
So $\mathcal{P}$ is a $\mathcal{V}$-tour with $6/n$ playing the role of $\gamma$.

\medskip
\noindent
\textbf{Case 2.}
\emph{$e_C(U,V) \leq 11$ for all distinct $U,V \in \mathcal{V}$.}

\medskip
\noindent
Let $\mathcal{P}$ be the collection of disjoint paths with edge set $E(C) \setminus \bigcup_{V \in \mathcal{V}}E(C[V])$.
For each $V \in \mathcal{V}$, let $\mathcal{P}_V := \bigcup_{U \in \mathcal{V} \setminus \lbrace V \rbrace}\mathcal{P}[U,V]$. Then
\begin{align} \label{eCV}
e \left( \mathcal{P}_V  \right) = \sum\limits_{U \in \mathcal{V} \setminus \lbrace V \rbrace}e_C(U,V)
\leq 33.
\end{align}
Suppose that $|V(C) \cap V| < D - 2\rho^{1/3}n$.
Let $X := V \setminus V(C)$.
So $X$ is an independent set in $G$.
Moreover, (D7) implies that, for all but at most $\rho n$ vertices in $x \in V$, we have $d_V(x) \geq D - \rho n$.
In particular, $|V| \geq D - \rho n$ and so $|X| \geq \rho^{1/3}n$.
Thus there is some $x \in X$ such that $d_V(x) \geq D - \rho n$.
Therefore $x$ has a neighbour in $X$, a contradiction.

Thus $|V(C) \cap V| \geq D - 2\rho^{1/3}n$ for all $V \in \mathcal{V}$.
But
$$
2|V(C) \cap V| = \sum\limits_{v \in V}d_C(v) = 2e_C(V) + e(\mathcal{P}_V)$$
and hence
$$
e_C(V) = |V(C) \cap V| - \frac{1}{2}e(\mathcal{P}_V) \geq D - 2\rho^{1/3}n - 33/2 > 0.
$$
Thus $E(C[V]) \neq \emptyset$ for all $V \in \mathcal{V}$.
It is straightforward to check that this implies that $R_{\mathcal{V}}(\mathcal{P})$ is an Euler tour.%
\COMMENT{
Let $R := R_{\mathcal{V}}(\mathcal{P})$.
We have that $\mathcal{P}$ contains every edge of $C$ which does not lie within some cluster of the partition.
Now $C$ meets every cluster of $\mathcal{V}$, so for any $V,V' \in \mathcal{V}$, there is a path $P \subseteq C$ between a vertex of $V$ and a vertex of $V'$.
Then $\mathcal{P}[V(P)]$ is a path system whose corresponding edges in $R$ form a path between $V$ and $V'$.
Thus $R$ is connected.
Moreover each $V \in \mathcal{V}$ has even degree in $R$.
Therefore $R$ is an Euler tour.}
Finally, note that, for each $V \in \mathcal{V}$, (\ref{eCV}) implies that we have $|V(\mathcal{P}) \cap V| \leq e(\mathcal{P}_V) \leq 33$.
So $\mathcal{P}$ is a $\mathcal{V}$-tour with parameter $33/n$.
\hfill$\square$


\section{(0,2): Two bipartite robust expander components}\label{sec02}

The aim of this section is to prove the following lemma.

\begin{lemma} \label{(0,2)}
Let $D, n \in \mathbb{N}$, let $0 < 1/n \ll \rho \ll \nu \ll \tau \ll \alpha < 1$ and let $D \geq \alpha n$.
Suppose that $G$ is a $3$-connected $D$-regular graph on $n$ vertices and that $\mathcal{V}$ is a robust partition of $G$ with parameters $\rho,\nu,\tau,0,2$.
Then $G$ contains a $\mathcal{V}$-tour with parameter $\rho^{1/3}$.
\end{lemma}

We first give a brief outline of the argument.

\subsection{Sketch of the proof of Lemma~\ref{(0,2)}}

Let $\mathcal{V} := \lbrace W_1,W_2 \rbrace$ be as above and let $A_i,B_i$ be a bipartition of $W_i$ such that $|A_i|\ge |B_i|$ and $G[W_i]$ is a bipartite robust expander component with bipartition $A_i,B_i$ or $B_i,A_i$.
(To be precise, $G[W_i]$ is a bipartite robust expander component with bipartition $A_{W_i},B_{W_i}$, where $\{A_{W_i},B_{W_i}\} = \{A_i,B_i\}$.)\COMMENT{NEW29/5}

To prove Lemma~\ref{(0,2)}, our aim is to find a `balancing' path system $\mathcal{P}$ to which we can apply Lemma~\ref{balextend}
and hence obtain a $\mathcal{V}$-tour.
In other words, the path system has to `compensate for' the differences in the sizes of the vertex classes $A_i$ and $B_i$ and has to `join up' $W_1$ and $W_2$.
(This also justifies why we do not specify whether $G[W_i]$ is a bipartite robust expander component with bipartition $A_i,B_i$ or $B_i,A_i$.)\COMMENT{NEW29/5}

One could try to first find a path system which balances $W_1$, and then add additional edges so that $W_2$ is also balanced; however these additional edges may cause $W_1$ to become unbalanced.
So one must find a path system~$\mathcal{P}$ which simultaneously balances both components.

This is not too difficult if both $A_1$ and $A_2$ contain sufficiently large matchings $M_1$ and~$M_2$ (see Lemma~\ref{ensureconnected}).
In this case, we use the $3$-connectivity of $G$ to modify $M_1 \cup M_2$ to obtain~$\mathcal{P}$.

So suppose that this is not the case.
Then (see Lemmas~\ref{removeedges} and~\ref{rounding}) we show that we can choose $C_i \in \lbrace A_i,B_i \rbrace$ for each $i=1,2$
such that Vizing's and K\"onig's theorems on edge-colourings guarantee the following:
$G[C_1], G[C_2], G[W_1,A_2]$ contain matchings $M_{1}, M_{2}, M_{1,2}$ respectively, such that the union $\mathcal{R}$ of these matchings balances both $W_1$ and $W_2$.
However, two problems can arise: $\mathcal{R}$ may not connect $W_1$ and $W_2$ (it might contain no $W_1W_2$-path) and it might contain cycles.

Therefore the bulk of the proof of Lemma~\ref{(0,2)} is devoted to choosing $M_1,M_2$ and $M_{1,2}$ carefully to avoid these problems.
Observe that since we use Vizing's and K\"onig's theorems to find matchings, we can actually find much larger matchings in $H \subseteq G$ when $\Delta(H)$ is small, and thus choosing a `good' matching is easier in this case.
So most of the difficulty in the proof arises from the presence of vertices of high degree.%
\COMMENT{NEW 2nd paragraph onwards}

\subsection{Balanced subgraphs with respect to a partition}

Consider a graph $G$ with vertex partition $\mathcal{V} := \lbrace W_1,W_2 \rbrace$, where $W_i$ has bipartition $A_i,B_i$ for $i=1,2$.
Write $\mathcal{V}^*$ for the ordered partition $(A_1,B_1,A_2,B_2)$.
Given $D \in \mathbb{N}$, we say that $G$ is \emph{$D$-balanced (with respect to $\mathcal{V}^*$)} if both of the following hold.
\begin{align}
\label{balancing}
2e(A_1) - 2e(B_1) + e(A_1,W_2) - e(B_1,W_2) &= D(|A_1|-|B_1|);\\
\nonumber 2e(A_2) - 2e(B_2) + e(A_2,W_1) - e(B_2,W_1) &= D(|A_2|-|B_2|).
\end{align}

Proposition~\ref{fact2}(i) easily implies that any $D$-regular graph with arbitrary ordered partition $\mathcal{V}^*$ is $D$-balanced.%
\COMMENT{Proof of Prop~\ref{regbal}: Let $W_i := A_i \cup B_i$.
Apply Prop~\ref{fact2}(i) with $A_1,B_1,W_2$ playing the roles of $A,B,U$ to get the first $D$-balanced condition.
Apply Prop~\ref{fact2}(i) with $A_2,B_2,W_1$ playing the roles of $A,B,U$ to get the second $D$-balanced condition.}

\begin{proposition}\label{regbal}
Suppose that $G$ is a $D$-regular graph and let $A_1,B_1,A_2,B_2$ be a partition of $V(G)$.
Then $G$ is $D$-balanced with respect to $(A_1,B_1,A_2,B_2)$.
\end{proposition}

The next proposition shows that, to prove Lemma~\ref{(0,2)}, it suffices to find a path system~$\mathcal{P}$ which is $2$-balanced with respect to $\mathcal{V}^*$, contains a $W_1W_2$-path, and does not have many edges.

\begin{proposition}\label{sufficient}
Let $n,D \in \mathbb{N}$ and $0 < 1/n \ll \rho \leq \gamma \ll \nu \ll \tau \ll \alpha < 1$.
Let $G$ be a $D$-regular graph on $n$ vertices with $D \geq \alpha n$.
Suppose further that $G$ has a robust partition $\mathcal{V} := \lbrace W_1, W_2 \rbrace$ with parameters $\rho,\nu,\tau,0,2$.
For each $i=1,2$, let $A_i,B_i$ be the bipartition of $W_i$ such that $|A_i| \ge |B_i|$ and $G[W_i]$ is a bipartite $(\rho,\nu,\tau)$-robust expander component with bipartition $A_i,B_i$ or $B_i,A_i$.
\COMMENT{NEW29/5}
Let $\mathcal{P}$ be a $2$-balanced path system with respect to $(A_1,B_1,A_2,B_2)$ in $G$.
Suppose that $e(\mathcal{P}) \leq \gamma n$ and that $\mathcal{P}$ contains at least one $W_1W_2$-path.
Then $G$ contains a $\mathcal{V}$-tour with parameter $18\gamma$.
\end{proposition}

\begin{proof}
Let $p$ be the number of $W_1W_2$-paths in $\mathcal{P}$.
Any $W_1W_2$-path in $\mathcal{P}$ contains an odd number of $W_1W_2$-edges.
Since $\mathcal{P}$ is $2$-balanced with respect to $(A_1,B_1,A_2,B_2)$, we have that $e_{\mathcal{P}}(W_1,W_2) = e_{\mathcal{P}}(A_1,W_2) - e_{\mathcal{P}}(B_1,W_2) + 2e_{\mathcal{P}}(B_1,W_2)$ is even.
Hence $p$ is even. 
Since $p>0$, we have that $R_{\mathcal{V}}(\mathcal{P})$ is an Euler tour.

The hypothesis $e(\mathcal{P}) \leq \gamma n$ implies that $|V(\mathcal{P}) \cap V| \leq 2\gamma n$ for all $V \in \mathcal{V}$.
Proposition~\ref{WRSD-RD} implies that $\mathcal{V}$ is a weak robust partition with parameters $2\gamma,\nu,\tau,\alpha^2/2,0,2$.
Thus we can apply Lemma~\ref{balextend}
with $\mathcal{V},0,2,W_i,\{A_i,B_i\},\mathcal{P},2\gamma$ playing the roles of $\mathcal{U},k,\ell,W_j,\{A_j,B_j\},\mathcal{P},\rho$ to find a $\mathcal{V}$-tour $\mathcal{P}'$ with parameter $18\gamma$.\COMMENT{NEW29/5}
\end{proof}

The next lemma shows that we can find a $D$-balanced subgraph of $G$ which only contains edges in some of the parts of $G$.
(Recall the definition of $\lceil \cdot \rceil_{\eps}$ from the end of Subsection~\ref{notation}.)

\begin{lemma} \label{removeedges}
Let $D \in \mathbb{N}$ be such that $D \geq 20$.
Let $G$ be a graph and let $\mathcal{V}^* := (A_1,B_1,A_2,B_2)$ be an ordered partition of $V(G)$
with $0\le |A_i|-|B_i| \le D/2$ for $i=1,2$.%
\COMMENT{Deryk added this and $D\ge 20$, the latter ensures $\lceil e(A)/5 \rceil_{1/4} \geq \lceil 2e(A)/D \rceil$}
Suppose that $e_G(A_1,B_2) \leq e_G(B_1,A_2)$ and $\Delta(G[A_i]) \leq D/2$ for $i=1,2$.%
\COMMENT{Added min deg condition for (ii).}
Suppose further that $G$ is $D$-balanced with respect to $\mathcal{V}^*$.
Then one of the following holds:
\begin{itemize}
\item[(i)] for $i=1,2$, $G[A_i]$ contains a matching $M_i$ of size $|A_i|-|B_i| \leq \lceil e_G(A_i)/5 \rceil_{1/4}$;
\item[(ii)] there exists a spanning subgraph $G'$ of $G$ which is $D$-balanced with respect to $\mathcal{V}^*$ and $E(G') \subseteq E(G[C_1]) \cup E(G[C_2]) \cup E(G[A_1 \cup B_1,A_2])$,
where $C_1 \in \lbrace A_1,B_1 \rbrace$ and $C_2 \in \lbrace A_2,B_2 \rbrace$.
\end{itemize}
\end{lemma}

\begin{proof}
Observe that the graph obtained by removing $E(G[A_i,B_i])$ from $G$ for $i=1,2$ is $D$-balanced.
So we may assume that $E(G[A_i,B_i]) = \emptyset$ for $i=1,2$.
Consider each of the pairs
\begin{align*}
\lbrace G[A_1], G[B_1] \rbrace, \lbrace G[A_2], G[B_2] \rbrace, \lbrace G[A_1,A_2], G[B_1,B_2] \rbrace, \lbrace G[A_1,B_2], G[B_1,A_2] \rbrace
\end{align*}
of induced subgraphs.
For each such pair $\lbrace J,J' \rbrace$, remove
$\min \lbrace e_G(J),e_G(J') \rbrace$ arbitrary edges from each of $J,J'$ in $G$.
Let $H$ be the subgraph obtained from $G$ in this way.
Then $H$ is $D$-balanced and for each pair $\lbrace J,J' \rbrace$, we have that $E(H[V(J)]) = \emptyset$ whenever $e_G(J) \leq e_G(J')$ (and vice versa).
In particular, $e_H(A_1,B_2)=0$.
Suppose that we cannot take $G' := H$ so that (ii) holds.
Then $H \subseteq
G[C_1] \cup G[C_2] \cup G[B_1,A_2 \cup B_2]$ for some $C_1 \in \lbrace A_1,B_1 \rbrace$ and $C_2 \in \lbrace A_2,B_2 \rbrace$ with $e_H(B_1,B_2) \geq 1$.
So $e_H(A_1,A_2) = 0$.
Let $v_i := D(|A_i| - |B_i|) \geq 0$.
Since $H$ is $D$-balanced we have that
$2e_H(A_1) - 2e_H(B_1) - e_H(B_1,A_2 \cup B_2) = v_1 \geq 0$.
In particular, $e_{H}(A_1) \geq e_H(B_1)$.
So $e_H(B_1)=0$.
Let $t := e_H(B_1,A_2)$.
Thus
\begin{align}
\label{H*ineq1} 2e_{H}(A_1) &\geq v_1+t+1\ \ \mbox{ and similarly}\\
\nonumber 2e_{H}(A_2) &\geq v_2-t+1.
\end{align}

Suppose first that $t \geq v_2$.
Then $2e_H(A_1) \geq v_1+v_2+1$.
Since $G$ is $D$-balanced, summing the two equations in (\ref{balancing}) implies that $v_1+v_2$ is even.
Let $H_{B_1A_2}$ consist of $v_2$ arbitrary edges in $H[B_1,A_2]$ and let $H_{A_1}$ consist of $(v_1+v_2)/2$ arbitrary edges in $H[A_1]$.
In this case, we let $G' := H_{A_1} \cup H_{B_1A_2}$. 
So (ii) holds.

Suppose instead that $t < v_2$.
First consider the case when $t=0$.%
\COMMENT{Previously the proof was wrong because we neglected this case.}
Then (\ref{H*ineq1}) implies that $2e_G(A_i) \geq 2e_H(A_i) \geq v_i+1$ for $i=1,2$.
Since $\Delta(G[A_i]) \leq D/2$, Vizing's theorem implies that $G[A_i]$ contains a matching $M_i$ of size%
\COMMENT{Deryk changed this}
$$
\left\lceil \frac{e_G(A_i)}{D/2+1} \right\rceil \ge
\left\lceil \frac{D(|A_i|-|B_i|)/2}{D/2+1} \right\rceil
\ge |A_i|-|B_i| -\lfloor D/(D+2) \rfloor =|A_i|-|B_i|.
$$
Note that the right hand side  is at most $\lceil e(A_i)/5 \rceil_{1/4}$.%
\COMMENT{LATE CHANGE: This is clear if $e(A_i) \geq 10$, say. Note that $e(A_i) \geq D(|A_i|-|B_i|)/2$. So if $e(A_i) < 10$ we have $e(A_i)=0$. In this case the assertion is also clear. (It is necessary to observe this since the assertion would not be true if, for example, $e(A_i)=1$.)}
So (i) holds.

Therefore we may assume that $t>0$. Recall that $v_1 \equiv v_2 \mod 2$.
We will choose $H_{B_1A_2} \subseteq H[B_1,A_2]$ and $H_{A_i} \subseteq H[A_i]$ for $i=1,2$ by arbitrarily choosing edges according to the relative parities of $v_1$ and $t$, such that the following hold:
\begin{itemize}
\item if $v_1+t$ is even then choose $e(H_{B_1A_2})=t$, $2e(H_{A_1})=v_1+t$, $2e(H_{A_2})=v_2-t$;
\item if $v_1+t$ is odd then choose $e(H_{B_1A_2})=t-1$, $2e(H_{A_1})=v_1+t-1$, $2e(H_{A_2})=v_2-t+1$.
\end{itemize}
These choices are possible by (\ref{H*ineq1}). 
We let $G' := H_{A_1} \cup H_{A_2} \cup H_{B_1A_2}$. 
Observe that $G'$ is $D$-balanced.
So (ii) holds.
\end{proof}

Observe that the subgraph $M_1 \cup M_2$ of $G$ guaranteed by Lemma~\ref{removeedges}(i) is a $2$-balanced path system.
The next lemma shows that, when $G$ is $3$-connected, one can modify such a path system into one which also contains paths between $A_1 \cup B_1$ and%
           \COMMENT{This lemma has been moved forward (before it was the last of the section).
Then the reader can always assume that we are in case (ii) of Lemma~\ref{removeedges} in the remainder of the section.}
$A_2 \cup B_2$.

\begin{lemma}\label{ensureconnected}
Let $n,D \in \mathbb{N}$ and $0 < 1/n \ll \gamma \ll 1$.
Let $G$ be a $3$-connected $D$-regular graph on $n$ vertices.
Let $W_1,W_2$ be a partition of $V(G)$ and let $A_i,B_i$ be a partition of $W_i$ for $i=1,2$, where $|A_i| \geq |B_i|$.
Suppose that there exist matchings $M_1, M_2$ in $G[A_1], G[A_2]$ respectively so that $|A_i|-|B_i| = e(M_i) \leq \lceil e(A_i)/5 \rceil_{1/4}$ and $e(M_i) \leq \gamma n$ for $i=1,2$.
Then $G$ contains a path system $\mathcal{P}$ which is $2$-balanced with respect to $(A_1,B_1,A_2,B_2)$ and contains a $W_1W_2$-path, and $e(\mathcal{P}) \leq 3\gamma n$.
\end{lemma}

\begin{proof}
Proposition~\ref{regbal} implies that $G$ is $D$-balanced with respect to $(A_1,B_1,A_2,B_2)$.
Suppose that there exist edges $e \in E(G[A_1,A_2])$ and $e' \in E(G[B_1,B_2])$.
Then we can take $\mathcal{P} := M_1 \cup M_2 \cup \lbrace e,e' \rbrace$.
We are similarly done if there exist edges $f \in E(G[A_1,B_2])$ and $f' \in E(G[B_1,A_2])$.
If either of these two hold then we say that $G$ contains a \emph{balanced matching}.
So we may assume that $G$ does not contain a balanced matching.
The $3$-connectivity of $G$ implies that there is a matching $N$ of size at least three in $G[W_1,W_2]$.
Since $G$ does not contain a balanced matching, $e_N(C_1,C_2) \geq 2$ for some $C_i \in \lbrace A_i, B_i \rbrace$.
So we can choose a matching $N'$ of size two in $G[C_1,C_2]$.
Let $D_i$ be such that $\lbrace C_i, D_i \rbrace := \lbrace A_i,B_i \rbrace$.
Note that $e_G(D_1,D_2) = 0$ or $G$ would contain a balanced matching.
Without loss of generality, we may assume that $e(M_1) \leq e(M_2)$.

\medskip
\noindent
\textbf{Case 1.}
$e(M_2)>0$.

\medskip
\noindent
Note that $1 \leq e(M_2) \leq e_G(A_2)/5+3/4$.
Thus%
\COMMENT{$\lceil a \rceil_{\eps} \leq a+1-\eps$}
$e_G(A_2)-e(M_2) \geq 4e_G(A_2)/5 - 3/4 > 0$.
So we can always
choose an edge $e_2 \in E(G[A_2]) \setminus E(M_2)$.
If possible, let $f_2$ be the edge of $M_2$ spanned by $V(N') \cap A_2$.
If there is no such edge, let $f_2$ be an arbitrary edge in $M_2$.
Let
$$
M_2' := \begin{cases}  
M_2 \setminus \lbrace f_2 \rbrace
&\mbox{if } C_2=A_2\\
M_2 \cup \lbrace e_2 \rbrace 
&\mbox{if } C_2=B_2.\\
\end{cases}
$$

\medskip
\noindent
\textbf{Case 1.a.}
$e(M_1)>0$.

\medskip
\noindent
Define $e_1,f_1$ and hence $M_1'$ analogously to $e_2,f_2, M_2'$.
It is straightforward to check that
$\mathcal{P} := N' \cup M_1' \cup M_2'$ is as required in the lemma.

\medskip
\noindent
\textbf{Case 1.b.}
$e(M_1)=0$.

\medskip
\noindent
We have $|A_1|=|B_1|$.
Without loss of generality we may suppose that $C_1 = A_1$ or we can swap $A_1,B_1$.
So $e_G(A_1,W_2) \geq e_N(C_1,C_2) \geq 2$.
Since $G$ is $D$-balanced and $e_G(B_1,C_2)=e_G(B_1,W_2)$,
this in turn implies
that $2e_G(B_1)+e_G(B_1,C_2) \geq 2$.
If $e_G(B_1) > 0$ let $e \in E(G[B_1])$ be arbitrary
and define
$\mathcal{P} := N' \cup M_2' \cup \lbrace e \rbrace$. 
Otherwise, there exists $e_{12} \in E(G[B_1,C_2])$.
Let $e_{12}' \in E(N')$ be vertex-disjoint from $e_{12}$.
If possible, let $f_2' \in E(M_2)$ be the edge spanning the endpoints of $e_{12}, e_{12}'$ which lie in $A_2$; otherwise, let $f_2' \in E(M_2)$ be arbitrary. 
If $C_2=A_2$, let $\mathcal{P} := M_2 \cup \lbrace e_{12},e_{12}' \rbrace \setminus \lbrace f_2' \rbrace$.
If $C_2=B_2$, let $\mathcal{P} := M_2 \cup \lbrace e_{12},e_{12}' \rbrace$.
It is straightforward to check that in all cases $\mathcal{P}$ is as required in the lemma.%
\COMMENT{Deryk changed the last 3 sentences}

\medskip
\noindent
\textbf{Case 2.}
$e(M_2)=0$.

\medskip
\noindent
So $e(M_1)=0$ and $|A_i|=|B_i|$ for $i=1,2$.
Without loss of generality, we may assume that $C_i := A_i$ (and hence $D_i := B_i$).
Write $\lbrace i,j \rbrace = \lbrace 1,2 \rbrace$.
Since $G$ is $D$-balanced we have that
$$
2e_G(A_i) - 2e_G(B_i) + e_G(A_i,A_j) + e_G(A_i,B_j) - e_G(B_i,A_j) = 0.
$$
So
$
2e_G(B_i) + e_G(B_i,A_j) \geq e_N(A_1,A_2) \geq 2
$.
Therefore either $e_G(B_i)>0$ or $e_G(B_i,A_j)>0$ (or both).
So for $i=1,2$, either we can find $e_i \in E(G[B_i])$ or $e_{ij} \in E(G[B_i,A_j])$ (or both).
Note that not both $e_G(B_1,A_2), e_G(A_1,B_2)$ can be positive since $G$ does not contain a balanced matching.

Suppose that $e_G(B_1),e_G(B_2)>0$.
Let $\mathcal{P} := N' \cup \lbrace e_1, e_2 \rbrace$ as required.
If  $e_G(B_1) =  0$ or $e_G(B_2) = 0$, then we may assume without loss of generality that
$e_G(B_1)>0$ and $e_G(B_2,A_1)>0$.
Let $e_{12}' \in N'$ be vertex-disjoint from $e_{21}$.
Let $\mathcal{P} := \lbrace e_1, e_{12}', e_{21} \rbrace$.
It is straightforward to check that in both cases $\mathcal{P}$ is as required in the lemma.
\end{proof}

\subsection{Tools for finding matchings}

Given any bipartite graph $G$, K\"onig's theorem on edge-colourings guarantees that we can find a matching of size at least $\lceil e(G)/\Delta(G) \rceil$.
The following lemma shows that, given any matching $M$ in $G$, we can find a matching $M'$ of at least this size such that $V(M) \subseteq V(M')$.

\begin{lemma} \label{matchingextend}
Let $G$ be a bipartite graph with vertex classes $V,W$ such that $\Delta(G) \leq \Delta$.
Let $M$ be a matching in $G$ with $e(M)\le \lceil e(G)/\Delta \rceil$.
Then there exists a matching $M'$ in $G$ such that $e(M')=\lceil e(G)/\Delta \rceil$ and%
   \COMMENT{DK: reformulated lemma and proof so that $e(M')=\lceil e(G)/\Delta \rceil$
instead of $e(M')\ge \lceil e(G)/\Delta \rceil$, since this is what we need later on}
$V(M) \subseteq V(M')$.
\end{lemma}

\begin{proof}
Let $M'$ be a matching in $G$ such that $V(M) \subseteq V(M')$ and $e(M')\le \lceil e(G)/\Delta \rceil$
is maximal with this property. Suppose that $e(M')< \lceil e(G)/\Delta \rceil$.
Since, by K\"onig's theorem on edge-colourings, $G$ contains a matching of size $\lceil e(G)/\Delta \rceil$,
this means that $M'$ is not a maximum matching. So, by Berge's lemma,
$G$ contains an augmenting path $P$ for $M'$, i.e.~a path with endpoints not in $V(M')$ which
alternates between edges in $E(M')$  and edges outside of $E(M')$. But then 
$P \setminus E(M')$ is a matching contradicting the maximality of $e(M')$.
\end{proof}

We now show that given a bipartite graph $G = (U,Z)$ and any partition $V,W$ of $Z$, we can find a large matching in
$G$ which has the `right' density in each of $G[U,V]$ and $G[U,W]$.

\begin{lemma}\label{spreadmatching}
Let $G$ be a bipartite graph with vertex classes $U,V \cup W$, where $V,W$ are disjoint.
Suppose that $\Delta(G) \leq \Delta$.
Let $b_V, b_W$ be non-negative integers such that $b_V+b_W \leq \lceil e(G)/\Delta \rceil$, $b_V \leq \lceil e_G(U,V)/\Delta \rceil$ and $b_W \leq \lceil e_G(U,W)/\Delta \rceil$.
Then $G$ contains a matching $M$ such that 
$e_{M}(U,V) = b_V$ and $e_M(U,W) = b_W$.%
\COMMENT{Can't have ceilings for both. But that's okay.
Fact: if have $b'+c'$ s.t. $b'+c' = \lceil b+c \rceil$ and $b' \leq \lceil b \rceil$, $c' \leq \lceil c \rceil$, then either $b' = \lceil b\rceil$ or $c' = \lceil c \rceil$, or both.}
\end{lemma}

\begin{proof}
By increasing $b_V,b_W$ if necessary,
we may assume that $b_V + b_W = \lceil e(G)/\Delta \rceil$.
Note that either $b_V = \lceil e_G(U,V)/\Delta \rceil$, or $b_W = \lceil e_G(U,W)/\Delta \rceil$, or both.
Suppose without loss of generality that $b_V = \lceil e_G(U,V)/\Delta \rceil$.
Choose a matching $M'$ in $G$ of size $\lceil e(G) /\Delta \rceil$.
Let $m_V := e_{M'}(U,V)$ and let $m_W := e_{M'}(U,W)$.
Let $k := b_V - m_V$.
Then
$$
m_W = \lceil e(G)/\Delta \rceil - m_V = b_V + b_W - m_V = b_W + k.
$$
If $k = 0$ we are done, so suppose first that $k>0$.
Apply Lemma~\ref{matchingextend} to obtain a matching $J_V$ in $G[U,V]$ such that $e(J_V) = b_V$ and $V(J_V) \supseteq V(M'[U,V])$.
So $|(V(J_V) \setminus V(M'[U,V])) \cap U| = k$.
Thus we can choose a submatching $J_W$ of $M'[U,W]$ of size $m_W-k=b_W$ 
that is vertex-disjoint from $J_V$.
Let $M := J_V \cup J_W$.

Otherwise, $k<0$.
Apply Lemma~\ref{matchingextend} to obtain a matching $J_W$ in $G[U,W]$ such that $e(J_W) = b_W$ and $V(J_W) \supseteq V(M'[U,W])$.
As above, we can choose a submatching $J_V$ of $M'[U,V]$ of size $b_V$ that is vertex-disjoint from $J_W$.
Let $M := J_V \cup J_W$.
\end{proof}

\subsection{Acyclic unions of matchings}

The next lemma shows that, in a graph with low maximum degree, we can find a large matching that does not completely span a given set of vertices.%
\COMMENT{Our aim is to find a $2$-balanced path system $\mathcal{P}$ in $G$ that consists of matchings in and between $A_1,B_1,A_2,B_2$.
We require that $\mathcal{P}$ has a $W_1W_2$-path.
Suppose that we have added a matching $N$ between $A_1$ and $W_2$ to $\mathcal{P}$, and we now wish to add an additional matching $M$ in $G[A_1]$.
Then $M \cup N$ contains a $W_1W_2$-path unless $M[V(N) \cap A_1]$ is a perfect matching.
}

\begin{proposition}\label{sparsematching}

Let $0 < 1/\Delta \ll \eta \ll 1$.
Let $G$ be a graph with $\Delta(G) \leq \eta \Delta$ and suppose that $e(G) \geq 2\eta \Delta$.
Suppose that $K \subseteq V(G)$.
Then there exists a matching $M$ in $G$ such that
$e(M) =  
\lceil e(G)/\Delta \rceil
$ 
and $M[K]$ is not a perfect matching.
\end{proposition}

\begin{proof}
By Vizing's theorem, $G$ contains a matching $M'$ of size
$$
\left\lceil \frac{e(G)}{\Delta(G)+1} \right\rceil \geq \left\lceil \frac{e(G)}{3\eta\Delta/2} \right\rceil \geq \left\lceil \frac{e(G)}{\Delta} \right\rceil +1.
$$
Delete edges so that $M'$ has size $\lceil e(G)/\Delta \rceil + 1$.
If $M'$ contains an edge with both endpoints in $K$, remove this edge to obtain $M$.
Otherwise, obtain $M$ from $M'$ by removing an arbitrary edge.
\end{proof}

Proposition~\ref{sparsematching} and the following observation will be used to guarantee that, given a matching $M$ in $G[W_1,A_2]$, we can find a suitable matching $N$ in $G[A_2]$ such that the path system $M \cup N$ contains a $W_1A_2$-path.

\begin{fact}\label{obvious}
Let $G$ be a graph with vertex partition $U,V$ and let $M$ be a non-empty matching between $U$ and $V$.
Let $K := V(M) \cap V$ and let $M'$ be a matching in $G[V]$ such that $M'[K]$ is not a perfect matching.
Then $M \cup M'$ is a path system containing a $UV$-path.
\end{fact}

Given a graph $G$ with low maximum degree, vertex partition $U,V$ and a non-empty matching $M$ in $G[U,V]$, the next lemma shows that we can find matchings in $G[U],G[V]$ which extend $M$ into a path system $\mathcal{P}$ containing a $UV$-path.

\begin{lemma}\label{sparsethreematchings}
Let $0 < 1/\Delta \ll \eta \ll 1$.
Let $G$ be a graph with partition $U,V$ and suppose that $\Delta(G) \leq \eta \Delta$.
Let $M$ be a matching between $U$ and $V$.
Suppose further that $e_G(U) \leq e_G(V) \leq \eta\Delta^2$.
Then there exist matchings $M_U, M_V$ in $G[U], G[V]$ respectively such that
\begin{itemize}
\item[(i)] $\mathcal{P} := M \cup M_U \cup M_V$ is a path system;
\item[(ii)] $e(M_U) \leq \lceil e_G(U)/\Delta \rceil$ with equality if $e_G(U) \geq \sqrt{\eta}\Delta$; and $e(M_V) \leq \lceil e_G(V)/\Delta \rceil$ with equality if $e_G(V) \geq \sqrt{\eta}\Delta$;
\item[(iii)] if $M \neq \emptyset$, then $\mathcal{P}$ contains a $UV$-path.
\end{itemize}
\end{lemma}

\begin{proof}
If $M = \emptyset$ then Vizing's theorem implies that we can find matchings $M_U, M_V$ of size $\lceil e_G(U)/\Delta \rceil$, $\lceil e_G(V)/\Delta \rceil$ respectively.
Then the consequences (i)--(iii) hold.
So we may assume that $M \neq \emptyset$.
If $e_G(U) \leq e_G(V) < \sqrt{\eta}\Delta$, then we are done by taking $M_U,M_V := \emptyset$.
Suppose instead that $e_G(U) < \sqrt{\eta}\Delta \leq e_G(V)$.
Apply Proposition~\ref{sparsematching}
with $G[V], V(M) \cap V$ playing the roles of $G,K$
to obtain a matching $M_V$ in $G[V]$ such that $e(M_V) = \lceil e_G(V)/\Delta \rceil$ and $M_V[V(M) \cap V]$ is not a perfect matching.
Fact~\ref{obvious} implies that we are done by taking $M_U = \emptyset$.

Therefore we may assume that $\sqrt{\eta}\Delta \leq e_G(U) \leq e_G(V)$.
Apply Proposition~\ref{sparsematching} with $G[U], V(M) \cap U$ playing the roles of $G,K$ to obtain a matching $M_U$ in $G[U]$ of size $\lceil e_G(U)/\Delta \rceil$ such that $M_U[V(M) \cap U]$ is not a perfect matching.
Let $\mathcal{P}_U$ be the path system with edge set $E(M) \cup E(M_U)$.
So Fact~\ref{obvious} implies that $\mathcal{P}_U$ contains at least one $UV$-path $P$.
Let $u_0 \in U$ and $v_0 \in V$ be the endpoints of $P$.
Let $Y$ be the set of all those vertices in $V$ which are endpoints of a $VV$-path in $\mathcal{P}_U$.
Now
\begin{equation}\label{|Y|}
|Y| \leq 2e(M_U) = 2\lceil e_G(U)/\Delta \rceil \leq 2\lceil e_G(V)/\Delta \rceil.
\end{equation}
Obtain $G'$ from $G[V]$ by removing every edge incident with $Y \cup \lbrace v_0 \rbrace$. So%
\COMMENT{LATE CHANGE: $e(V) - \eta\Delta(|Y|+1) \geq e(V) - \eta\Delta( 2\lceil e(V)/\Delta \rceil +1) \geq e(V) - \eta\Delta(2e(V)/\Delta+3) = e(V)(1-2\eta)-3\eta\Delta \geq e(V)(1-2\eta) - 3\sqrt{\eta}e(V) \geq e(V)(1-4\sqrt{\eta})$.}
\begin{align*}
e(G') &\geq e_G(V) - \eta\Delta(|Y|+1) \stackrel{(\ref{|Y|})}{\geq} (1-4\sqrt{\eta})e_G(V) \geq e_G(V)/2.  
\end{align*}
So $G'$ contains a matching of size
$$
\lceil e(G')/(\eta\Delta+1) \rceil \geq \lceil e(G')/2\eta\Delta \rceil \geq \lceil e_G(V)/4\eta\Delta \rceil \geq \lceil e_G(V)/\Delta \rceil.
$$
Let $M_V$ be an arbitrary submatching of this matching of size $\lceil e_G(V)/\Delta \rceil$.
Let $\mathcal{P} := M \cup M_U \cup M_V$.

Clearly (ii) holds.
Observe that $\mathcal{P}$ has a $UV$-path, namely $P$.
Hence (iii) holds.
To show (i), it is enough to show that $\mathcal{P}$ is acyclic.
Suppose not and let $C$ be a cycle in $\mathcal{P}$.
Now $C$ contains at least one edge $e \in E(M_V)$.
Then both endpoints%
\COMMENT{Allan, at least one rather than both? Deryk:`at least one' would make sense too, but I'd leave it as it is}
of this edge belong to $Y$, and hence $e \notin E(G')$, a contradiction.
\end{proof}

The following is a version of Lemma~\ref{sparsethreematchings} for sparse graphs which may have a small number of vertices with high degree.

\begin{lemma}\label{threematchings}
Let $0 < 1/\Delta \ll \rho \ll 1$.
Let $G$ be a graph with vertex partition $U,V$ and suppose that $\Delta(G[U]),\Delta(G[V]) \leq \Delta$.
Let $M$ be a matching between $U$ and $V$ such that $e(M) \leq \rho\Delta$.
Suppose further that $e_G(U),e_G(V) \leq \rho\Delta^2$.
Then, for any integers $0 \leq a_U \leq \lceil e_G(U)/\Delta \rceil_{1/4}$ and $0 \leq a_V \leq \lceil e_G(V)/\Delta \rceil_{1/4}$, $G$ contains a path system $\mathcal{P}$ such that
\begin{itemize}
\item[(i)] $\mathcal{P}[U,V] = M$ and both of $\mathcal{P}[U]$,$\mathcal{P}[V]$ are matchings;
\item[(ii)] $e_{\mathcal{P}}(U) = a_U$, $e_{\mathcal{P}}(V) = a_V$;
\item[(iii)] if $M \neq \emptyset$, then $\mathcal{P}$ contains a $UV$-path.
\end{itemize}
\end{lemma}

\begin{proof}
By removing edges in $G[U]$ and $G[V]$ we may assume without loss of generality that $a_U = \lceil e_G(U)/\Delta \rceil_{1/4}$ and $a_V = \lceil e_G(V)/\Delta \rceil_{1/4}$.
Choose $\eta$ with $\rho \ll \eta \ll 1$.
Let $U' := \lbrace u \in U: d_U(u) \geq \eta\Delta \rbrace$ and define $V'$ analogously.
Then $2e_G(U) \geq \sum_{u \in U'}d_U(u) \geq |U'|\eta \Delta$ and similarly for $V'$, so
\begin{equation}\label{U'}
|U'|, |V'| \leq \sqrt{\rho}\Delta.
\end{equation}
Let $U_0 := U \setminus U'$ and $V_0 := V \setminus V'$.
Let $H$ be the graph with vertex set $V(G)$ and edge set $E(G[U_0]) \cup E(G[V_0]) \cup M$.
So $E_H(U)=E_G(U_0)$ and $E_H(V)=E_G(V_0)$.%
\COMMENT{edges of $M$ may be incident to vertices in $U' \cup V'$ so we apply the sparse lemma to this $H$ rather than $G[U_0] \cup G[V_0]$}
Moreover, $\Delta(H) \leq 2\eta\Delta$.
Note that
\begin{equation}\label{eG0}
e_G(U_0) \geq e_G(U) - \Delta |U'|\ \ \mbox{ and }\ \ e_G(V_0) \geq e_G(V) - \Delta|V'|.
\end{equation}
Assume without loss of generality that $e_G(U_0) \leq e_G(V_0)$.
Apply Lemma~\ref{sparsethreematchings} with $H,M,U,V,2\eta$ playing the roles of $G,M,U,V,\eta$ to obtain matchings
$M_{U_0},M_{V_0}$ in $H[U_0]=G[U_0], H[V_0]=G[V_0]$ respectively such that
$\mathcal{P}_0 := M \cup M_{U_0} \cup M_{V_0}$
is a path system 
satisfying the consequences~(i)--(iii) of Lemma~\ref{sparsethreematchings}.
So $\mathcal{P}_0$ contains a $UV$-path if $M \neq \emptyset$. 
Moreover, $e(M_{U_0}) \leq \lceil e_G(U_0)/\Delta \rceil$ with equality if $e_G(U_0) \geq \sqrt{2\eta}\Delta$,
and $e(M_{V_0}) \leq \lceil e_G(V_0)/\Delta \rceil$ with equality if $e_G(V_0) \geq \sqrt{2\eta}\Delta$.
Thus
\begin{equation}\label{VP0}
|V(\mathcal{P}_0)| \leq 2e(\mathcal{P}_0) \leq 2\left( e(M) + \lceil e_G(U)/\Delta \rceil + \lceil e_G(V)/\Delta \rceil \right) \leq \sqrt{\rho}\Delta. 
\end{equation}
For every $u \in U'$ and $v \in V'$ we have that
\begin{equation*}
d_{U_0 \setminus V(\mathcal{P}_0)}(u), d_{V_0 \setminus V(\mathcal{P}_0)}(v) \stackrel{(\ref{VP0})}{\geq}
\eta\Delta/2 \stackrel{(\ref{U'})}{>} |U'|,|V'|.
\end{equation*}%
\COMMENT{recall that $U' \cap U_0 = \emptyset$. Deryk changed calculation slightly}
So for each $u \in U'$, we may choose a distinct neighbour $w_u \in U_0 \setminus V(\mathcal{P}_0)$ of $u$.
Let $M_{U'} := \lbrace uw_u : u \in U' \rbrace \subseteq G[U',U_0 \setminus V(\mathcal{P}_0)]$.
Define a matching $M_{V'}$ in $G[V',V_0 \setminus V(\mathcal{P}_0)]$ (which covers $V'$) similarly.

Let $\mathcal{P} := \mathcal{P}_0 \cup M_{U'} \cup M_{V'}$.
Note that $\mathcal{P}$ is a path system since $\mathcal{P}_0$ is.
Certainly $\mathcal{P}[U,V] = \mathcal{P}_0[U,V] = M$, so (i) holds.
Suppose that $e_G(U_0) \geq \sqrt{2\eta}\Delta$.
Then
\begin{align*}
e_{\mathcal{P}}(U) &= e(M_{U_0}) + e(M_{U'}) = \lceil e_G(U_0)/\Delta \rceil + |U'| \stackrel{(\ref{eG0})}{\geq} \lceil e_G(U)/\Delta - |U'| \rceil + |U'|\\
&= \lceil e_G(U)/\Delta \rceil \geq \lceil e_G(U)/\Delta \rceil_{1/4}.
\end{align*}
Suppose instead that $e_G(U_0) < \sqrt{2\eta}\Delta$.
Then
$$
e_{\mathcal{P}}(U) \geq |U'| \stackrel{(\ref{eG0})}{\geq} \lceil e_G(U)/\Delta - \sqrt{2\eta} \rceil
\geq \lceil e_G(U)/\Delta \rceil_{1/4}
$$
since $\sqrt{2\eta} < 1/4$.
Analogous statements are true for $e_{\mathcal{P}}(V)$.
So by removing edges in $e_{\mathcal{P}}(U), e_{\mathcal{P}}(V)$ if necessary, we may assume that (ii) holds.
Note that $\mathcal{P}$ has a $UV$-path if $\mathcal{P}_0$ does (there is a one-to-one correspondence between the $UV$-paths in $\mathcal{P}$ and the $UV$-paths in $\mathcal{P}_0$).%
\COMMENT{The $UV$-paths in $\mathcal{P}$ are not precisely the $UV$-paths in $\mathcal{P}_0$. E.g.~if $uPv$ is a $UV$-path in $\mathcal{P}_0$ and $u \in U'$ (a vertex of large degree), then $w_u uPv$ is a (sub)path in $\mathcal{P}$.}
\end{proof}

\subsection{Rounding}

Given a small collection of reals which sum to an integer, the following lemma shows that we can suitably round these reals so that their sum is unchanged.%
\COMMENT{it is very important that we are able to round $a_i$ to at most $\lceil a_i \rceil_{\eps}$ (rather than just an integer which is at most $\lceil a_i \rceil$). Then, if $e_G(A_i)$ is very small, are not required to find any path-system edges in $G[A_i]$ (and indeed we cannot necessarily find any such edges).}
Lemmas~\ref{spreadmatching} and~\ref{threematchings} together enable us to find three matchings, one in each of $G[W_1],G[W_2]$
and $G[W_1,W_2]$, each of which is not too large, such that their union is a path system $\mathcal{P}$.
Lemma~\ref{rounding} will allow us to choose the size of each matching correctly, so that $\mathcal{P}$ is $2$-balanced.

\begin{lemma}\label{rounding}
Let $0 < \eps < 1/2$.
Let $a_1, a_2, b, c \in \mathbb{R}$ with $b,c \geq 0$
and let $x_1, x_2 \in \mathbb{N}_0$.
Suppose that
$$
2a_1 + b - c = 2x_1\ \ \mbox{ and }\ \ 2a_2 + b + c = 2x_2.
$$
Then there exist integers $a_1',a_2',b',c'$ such that
$$
2a_1' + b' - c' = 2x_1\ \ \mbox{ and }\ \ 2a_2' + b' + c' = 2x_2,
$$
where $0 \leq b' \leq \lceil b \rceil$, $0 \leq c' \leq \lceil c \rceil$, $b'+c' \leq \lceil b+c \rceil$; and for $i=1,2$, $|a_i'| \leq \lceil |a_i| \rceil_{\eps}$; and 
finally $a_i' \geq 0$ if and only if $a_i \geq 0$.
\end{lemma}

\begin{proof}
Note that
\begin{equation}\label{rightsum}
\lfloor 2a_1 \rfloor + \lceil b - c \rceil = 2x_1\ \ \mbox{ and }\ \ \lfloor 2a_2 \rfloor + \lceil b + c \rceil = 2x_2.
\end{equation}
In particular, either $\lfloor 2a_1 \rfloor$, $\lceil b-c \rceil$ are both odd, or both even.
The same is true for the pair $\lfloor 2a_2 \rfloor, \lceil b+c \rceil$.
Let $A_i := \lfloor 2 a_i \rfloor/2$ for $i=1,2$.
Let also
$$
B := \frac{\lceil b+c\rceil + \lceil b-c \rceil}{2}\ \ \mbox{ and }\ \ C := \frac{\lceil b+c \rceil - \lceil b-c \rceil}{2}.
$$
Observe that $\lbrace A_1,A_2,B,C \rbrace \subseteq \mathbb{Z} \cup (\mathbb{Z} + 1/2)$.
Let $i \in \lbrace 1,2 \rbrace$.
Suppose first that $a_i \geq 0$ (and so $A_i \geq 0$).
If $a_i - \lfloor a_i \rfloor \leq \eps$ then 
$
2\lceil a_i \rceil_{\eps} = 2\lfloor a_i \rfloor = \lfloor 2a_i \rfloor = 2A_i.
$
If $a_i - \lfloor a_i \rfloor > \eps$ then
$
2\lceil a_i \rceil_{\eps} = 2\lceil a_i \rceil \geq \lfloor 2a_i \rfloor = 2A_i.
$
Therefore $\lceil A_i \rceil \leq \lceil a_i \rceil_{\eps}$.
Suppose now that $a_i < 0$ (and so $A_i< 0$).
If $a_i - \lfloor a_i \rfloor < 1-\eps$ then $2\lfloor a_i + \eps \rfloor = 2\lfloor a_i \rfloor \leq \lfloor 2a_i \rfloor = 2A_i$.
If $a_i - \lfloor a_i \rfloor \geq 1-\eps$ then $2\lfloor a_i + \eps \rfloor = 2\lfloor a_i \rfloor +2 = \lfloor 2a_i \rfloor +1 = 2A_i+1$ since $1-\eps \geq 1/2$. 
Since
$
-\lceil -a_i \rceil_{\eps} = \lfloor a_i + \eps \rfloor
$, this shows that
$-\lceil -a_i \rceil_{\eps} \leq \lceil A_i \rceil$.
Altogether this implies that
\begin{align}\label{Ai}
|A_i| \leq \lceil |a_i| \rceil_{\eps} \ \ &\mbox{ when }A_i \in \mathbb{Z},\ \ \mbox{ and }\\
\nonumber |A_i + 1/2| \leq \lceil |a_i| \rceil_{\eps}\ \ &\mbox{ when }A_i \in \mathbb{Z} + 1/2.
\end{align}
We also have that
\begin{equation}\label{B+C}
B+C = \lceil b+c \rceil\ \ \mbox{ and }\ \ B-C = \lceil b-c \rceil.
\end{equation}
Note that
\begin{align}\label{Bbound}
&\lceil 2b \rceil = \lceil b+c+b-c \rceil \leq 2B \leq \lceil b+c + (b-c) \rceil + 1 = \lceil 2b \rceil + 1 \leq 2\lceil b \rceil + 1;\\
\nonumber &\lceil 2c \rceil - 1 = \lceil b+c - (b-c) \rceil - 1 \leq 2C \leq \lceil b+c-(b-c) \rceil = \lceil 2c \rceil \leq 2\lceil c \rceil.
\end{align}
It is straightforward to check that these equations (together with the definition of $C$) imply the following:
\begin{align}\label{ineqs}
0 \leq B \leq \lceil b \rceil\ \ &\mbox{ when }B \in \mathbb{Z}\\
\nonumber 0 \leq B - 1/2 \leq \lceil b \rceil\ \ &\mbox{ when }B \in \mathbb{Z} + 1/2\\
\nonumber 0 \leq C \leq \lceil c \rceil\ \ &\mbox{ when }C \in \mathbb{Z}\\
\nonumber 0 \leq C - 1/2 < C + 1/2 \leq \lceil c \rceil\ \ &\mbox{ when }C \in \mathbb{Z} + 1/2.
\end{align}
Finally, note that (\ref{rightsum}) and (\ref{B+C}) together imply that
\begin{equation}\label{hope}
2A_1 + B - C = 2x_1\ \ \mbox{ and }\ \ 2A_2 + B + C = 2x_2.
\end{equation} 
We choose $a_1',a_2',b',c'$ as follows:

\medskip
\begin{center}
  \begin{tabular}{l | l | l | l | l | l }
     & $a_1'$ & $a_2'$ & $b'$ & $c'$ &  \\ \hline
    (i) & $A_1$ & $A_2$ & $B$ & $C$ & if $\lceil b+c \rceil$, $\lceil b-c \rceil$ both even;\\
    (ii) & $A_1+1/2$ & $A_2$ & $B-1/2$ & $C+1/2$ & if $\lceil b+c \rceil$ even, $\lceil b-c \rceil$ odd;\\
    (iii) & $A_1$ & $A_2+1/2$ & $B-1/2$ & $C-1/2$ & if $\lceil b+c \rceil$ odd, $\lceil b-c \rceil$ even;\\
    (iv) & $A_1+1/2$ & $A_2+1/2$ & $B-1$ & $C$ & if $b > 0$ and $\lceil b+c \rceil$, $\lceil b-c \rceil$ both odd;\\
    (v) & $A_1-1/2$ & $A_2+1/2$ & $B$ & $C-1$ & if $b=0$ and $\lceil b+c \rceil$, $\lceil b-c \rceil$ both odd.
  \end{tabular}
\end{center}

\medskip
\noindent
By the definition of $A_i$ we have for each $i=1,2$ that $a_i' \geq 0$ if and only if $a_i \geq 0$.
Then $\lbrace a_1',a_2',b',c' \rbrace \subseteq \mathbb{Z}$ and (\ref{hope}) implies that
$$
2a_1' + b' - c' = 2x_1\ \ \mbox{ and }\ \ 2a_2' + b' + c' = 2x_2.
$$
Moreover, $b'+c' \leq B+C= \lceil b+c \rceil$.
We claim that $0 \leq b' \leq \lceil b \rceil$ and $0 \leq c' \leq \lceil c \rceil$ and $|a_i'| \leq \lceil |a_i| \rceil_{\eps}$ for $i=1,2$ respectively in all cases (i)--(v).
To see this,
suppose first that we are in case (iv).
Since $b > 0$, (\ref{Bbound}) implies that $B \geq \lceil 2b \rceil/2 > 0$, so, since $B \in \mathbb{Z}$, $B - 1 \geq 0$ in this case.

Suppose now that we are in case (v).
Then $\lceil c \rceil, \lceil -c \rceil = -\lfloor c \rfloor$ are both odd.
Therefore $\lceil c \rceil, \lfloor c \rfloor$ are both odd so $\lceil c \rceil = \lfloor c \rfloor = c$.
So $c \in \mathbb{N}_0$ is odd, $B=0$ and $C = c$. Thus $C-1 \geq 0$.
Moreover $c=2A_1-2x_1$,
so $2A_1$ is odd and positive, which implies that $A_1-1/2 \geq 0$.
Then (\ref{Ai}) implies that $|A_1-1/2| \leq \lceil |a_i| \rceil_{\eps}$.%
    \COMMENT{DK has changed the last 3 sentences slightly}

In all cases (i)--(v), these last deductions together with (\ref{Ai})--(\ref{ineqs}) complete the proof of the lemma.
\end{proof}

\subsection{Proof of Lemma~\ref{(0,2)}}

Before we can prove Lemma~\ref{(0,2)}, we need one more preliminary result which%
   \COMMENT{DK changed this para and replaced $\mathcal{P}[W_1,W_2] \neq \emptyset$ by $e_{\mathcal{P}}(W_1,W_2)>0$
in the lemma below (and in the proof of Lemma~\ref{(0,2)}}
guarantees a path system $\mathcal{P}$ that can balance out the vertex class sizes of the bipartite graphs induced by the $W_i$.
If $e_{\mathcal{P}}(W_1,W_2)=0$, then we 
will use $3$-connectivity (via Lemma~\ref{ensureconnected}) to modify $\mathcal{P}$ into a balanced path system which also links up the~$W_i$.

\begin{lemma}\label{2balanced}
Let $0 < 1/n \ll \rho \ll \nu \ll \tau \ll \alpha < 1$ and let $G$ be a $D$-regular graph on $n$ vertices with $D \geq \alpha n$.
Suppose that $G$ has a robust partition $\mathcal{V} := \lbrace W_1,W_2 \rbrace$ with parameters $\rho,\nu,\tau,0,2$.
For each $i=1,2$, let $A_i,B_i$ be the bipartition of $W_i$ such that $|A_i| \ge |B_i|$ and $G[W_i]$ is a bipartite $(\rho,\nu,\tau)$-robust expander component with bipartition $A_i,B_i$ or $B_i,A_i$.
\COMMENT{NEW29/5}
Then
\begin{itemize}
\item[(i)] $G$ contains a path system $\mathcal{P}$ which is $2$-balanced with respect to $(A_1,B_1,A_2,B_2)$ such that $e(\mathcal{P}) \leq \sqrt{\rho}n$;
\item[(ii)] if $e_{\mathcal{P}}(W_1,W_2)>0$ then $\mathcal{P}$ contains a $W_1W_2$-path;
\item[(iii)] for $i=1,2$, $\mathcal{P}[W_i]$ consists either of a matching in $G[A_i]$ of size at most $\lceil e_G(A_i)/5 \rceil_{1/4}$,
or a matching in $G[B_i]$ of size at most $\lceil e_G(B_i)/5 \rceil_{1/4}$.
\end{itemize}
\end{lemma}

\begin{proof}
Write $\mathcal{V}^* := (A_1,B_1,A_2,B_2)$.
Let $\Delta := D/2$ and 
note that
$$
\Delta(G[A_i]), \Delta(G[B_i]), \Delta(G[W_1,W_2]) \leq \Delta
$$
for $i=1,2$ by (D4) and (D5).
Without loss of generality, we may suppose that $e_G(A_1,B_2) \leq e_G(B_1,A_2)$.
Note that $G$ is $D$-balanced with respect to $\mathcal{V}^*$ by Proposition~\ref{regbal}.
Apply Lemma~\ref{removeedges} to $G$.
Suppose that the consequence~(i) of Lemma~\ref{removeedges} holds.
Then $G[A_i]$ contains a matching $M_i$ of size $|A_i|-|B_i|\le \lceil e_G(A_i)/5 \rceil_{1/4}$ for $i=1,2$.%
   \COMMENT{DK added $\le \lceil e_G(A_i)/5 \rceil_{1/4}$}
Set $\mathcal{P} := M_1 \cup M_2$.
So (iii) holds, (D3) and (C2) imply that (i) holds, and (ii) is vacuous.%
\COMMENT{LATE CHANGE: justification for (i).}

So we may assume that the consequence~(ii) of Lemma~\ref{removeedges} holds.
Let $H$ be a spanning subgraph of $G$ which is $D$-balanced with respect to $\mathcal{V}^*$ such that
$E(H) \subseteq E(G[C_1]) \cup E(G[C_2]) \cup E(G[W_1,A_2])$ for some $C_1 \in \lbrace A_1, B_1 \rbrace$ and $C_2 \in \lbrace A_2, B_2 \rbrace$.
Observe that
\begin{equation}\label{eH}
e(H) \leq \sum\limits_{i=1,2}\left( e_G(A_i, \overline{B_i}) + e_G(B_i,\overline{A_i})\right) \stackrel{{\rm (D3)},{\rm (C3)}}{\leq} 2\rho n^2.
\end{equation}
For each $H' \subseteq H$ and $i=1,2$, define
\begin{equation}\label{fdef}
f_i(H') = e_{H'}(A_i)-e_{H'}(B_i).
\end{equation}
Now (\ref{balancing}) implies that, for any $t \in \mathbb{N}_0$, $H'$ is $t$-balanced if
\begin{align}\label{fbal}
2f_i(H') + e_{H'}(A_i,W_j) - e_{H'}(B_i,W_j) = t(|A_i|-|B_i|)
\end{align}
for $\lbrace i,j \rbrace = \lbrace 1,2 \rbrace$.
Observe that
$
e_H(C_i) = e_H(W_i) = |f_i(H)|
$.
For $i=1,2$, let
\begin{equation}\label{ai}
a_i := f_{i}(H)/\Delta.
\end{equation}
Then the $D$-balancedness of $H$ and (\ref{fbal}) imply that
\begin{align*}
2a_1 + \frac{e_{H}(A_1,A_2)}{\Delta} - \frac{e_{H}(B_1,A_2)}{\Delta} &= 2(|A_1|-|B_1|)\\
\nonumber\mbox{ and }
\ \
2a_2 + \frac{e_{H}(A_1,A_2)}{\Delta} + \frac{e_{H}(B_1,A_2)}{\Delta} &= 2(|A_2|-|B_2|).
\end{align*}
Apply Lemma~\ref{rounding} with $a_1,a_2, e_{H}(A_1,A_2)/\Delta, e_{H}(B_1,A_2)/\Delta, |A_1|-|B_1|, |A_2|-|B_2|,1/4$ playing the roles of $a_1,a_2,b,c,x_1,x_2,\eps$ to obtain integers $a_1',a_2',b',c'$ with%
   \COMMENT{DK added (\ref{aiai'}) and referred to it later on}
\begin{equation}\label{ai'}
|a_i'| \leq \lceil |a_i| \rceil_{1/4} = \lceil e_{H}(C_i)/\Delta \rceil_{1/4}\ \ \mbox{ for }i=1,2;
\end{equation}
\begin{equation}\label{aiai'}
a'_i\ge 0 \ \ \mbox{ if and only if  } \ \ a_i\ge 0;
\end{equation}
$0 \leq b' \leq \lceil e_{H}(A_1,A_2)/\Delta \rceil$; $0 \leq c' \leq \lceil e_{H}(B_1,A_2)/\Delta \rceil$ and
\begin{equation}\label{b'+c'}
b'+c' \leq \lceil e_{H}(W_1,A_2)/\Delta \rceil;
\end{equation}
\begin{align} \label{rounded1}
2a_1' + b' - c' = 2(|A_1|-|B_1|) \ \ \mbox{ and } \ \ 
2a_2' + b' + c' = 2(|A_2|-|B_2|).
\end{align}
Apply Lemma~\ref{spreadmatching} with $H[W_2,W_1],W_2,A_1,B_1$ playing the roles of $G,U,V,W$ to obtain a matching $M$ in $H[W_2,W_1]$ such that
\begin{align}\label{wheresM}
e_M(A_1,A_2) &= e_M(A_1,W_2) = b',\ \ e_M(B_1,A_2) = e_M(B_1,W_2) = c'\\
\nonumber \mbox{ and }\ \ e_M(W_1,B_2) &= 0.
\end{align}
Then (\ref{eH})%
    \COMMENT{DK referred to (\ref{eH}) instead of (C3)}
and (\ref{b'+c'}) imply that $e(M) = b'+c' \leq \lceil e(H)/\Delta \rceil \leq \sqrt{\rho}\Delta$.%
\COMMENT{$\rho n^2/\Delta = 2\rho n^2/D \leq 2\rho n/\alpha \leq 2\rho D/\alpha^2 = 4\rho \Delta/\alpha^2 \leq \sqrt{\rho}\Delta/2$}
By (\ref{eH}) and (\ref{ai'}), we can apply Lemma~\ref{threematchings} to $H$ with $\sqrt{\rho},M, \Delta, W_1,W_2,|a_1'|,|a_2'|$ playing the roles of $\rho, M,\Delta,U,V,a_U,a_V$
to obtain a path system $\mathcal{P}$ such that
\begin{align}
\label{Medges} \mathcal{P}[W_1,W_2] &= M;\\
\label{eP'Y} e_{\mathcal{P}}(W_i) = e_{\mathcal{P}}(C_i) &= |a_i'|\ \ \mbox{ for }i=1,2;
\end{align}
$\mathcal{P}[C_i]$ is a matching for $i=1,2$, and if $M \neq \emptyset$, then $\mathcal{P}$ contains a $W_1W_2$-path.
So (ii) holds.
(Note that~(\ref{eP'Y}) follows from the fact that $H[W_i]=H[C_i]$.)
Moreover, (\ref{ai'}) and (\ref{eP'Y}) imply that the matching $\mathcal{P}[C_i]$ has size  
at most $\lceil e_{H}(C_i)/\Delta \rceil_{1/4} \leq \lceil e_{G}(C_i)/\Delta \rceil_{1/4} \leq \lceil e_G(C_i)/5 \rceil_{1/4}$.
So (iii) holds.
Equations (\ref{fdef}), (\ref{ai}), (\ref{aiai'}) and (\ref{eP'Y}) imply that
\begin{equation}\label{sumai}
f_i(\mathcal{P})=a_i'.
\end{equation}
Furthermore, by (\ref{wheresM}) and (\ref{Medges}) we have
\begin{align}
\nonumber e_{\mathcal{P}}(A_1,W_2) - e_{\mathcal{P}}(B_1,W_2)
=
b'-c'\ \ \mbox{ and }\ \
e_{\mathcal{P}}(W_1,A_2) - e_{\mathcal{P}}(W_1,B_2) = b'+c'.
\end{align}
Together with (\ref{fbal}), (\ref{rounded1}) and (\ref{sumai}), this implies that
$\mathcal{P}$ is $2$-balanced with respect to $\mathcal{V}^*$.
Finally,
$$
e(\mathcal{P}) = |a_1'|+|a_2'| + b'+c' \stackrel{(\ref{ai'}),(\ref{b'+c'})}{\leq} e(H)/\Delta + 3 \stackrel{(\ref{eH})}{\leq} \sqrt{\rho}n,
$$
as required.
\end{proof}

\medskip
\noindent
\emph{Proof of Lemma~\emph{\ref{(0,2)}}.}
Let $\mathcal{V} := \lbrace W_1, W_2 \rbrace$ and for $i=1,2$, let $A_i,B_i$ be the bipartition of $W_i$ such that $|A_i| \ge |B_i|$ and $G[W_i]$ is a bipartite $(\rho,\nu,\tau)$-robust expander component with bipartition $A_i,B_i$ or $B_i,A_i$.%
\COMMENT{NEW 1/6}
Apply Lemma~\ref{2balanced} to obtain a path system $\mathcal{P}$ which is $2$-balanced with respect to $(A_1,B_1,A_2,B_2)$ such that $e(\mathcal{P}) \leq \sqrt{\rho}n$.

Suppose first that $e_{\mathcal{P}}(W_1,W_2) >0$.
Then $\mathcal{P}$ contains a $W_1W_2$-path by the consequence~(ii) of Lemma~\ref{2balanced}.
So we are done by Proposition~\ref{sufficient}.
Therefore we may assume that $e_{\mathcal{P}}(W_1,W_2) =0$.
The consequence~(iii) of Lemma~\ref{2balanced} implies that, for each $i=1,2$, at least one of $\mathcal{P}[A_i], \mathcal{P}[B_i]$ is empty, and the other is a matching of size at most $\lceil e_G(B_i)/5 \rceil_{1/4}, \lceil e_G(A_i)/5 \rceil_{1/4}$ respectively.%
\COMMENT{LATE CHANGE: swapped second $A_i$ and $B_i$.}
The $2$-balancedness of $\mathcal{P}$ implies that $e_{\mathcal{P}}(A_i) - e_{\mathcal{P}}(B_i) = |A_i|-|B_i| \geq 0$.
So $\mathcal{P} = M_1 \cup M_2$ for some matchings $M_i \subseteq G[A_i]$.
Apply Lemma~\ref{ensureconnected} to obtain a path system $\mathcal{P}'$ which is $2$-balanced with respect to $(A_1,B_1,A_2,B_2)$ and contains a $W_1W_2$-path, and $e(\mathcal{P}) \leq 3\sqrt{\rho}n$.
Again, we are done by Proposition~\ref{sufficient}.
\hfill$\square$
\medskip

\section{(2,1) : Two robust expander components and one bipartite robust expander component}\label{sec21}

The aim of this section is to prove the following lemma.

\begin{lemma}\label{(2,1)}
Let $0 < 1/n \ll \rho \ll \nu \ll \tau \ll 1$.
Let $G$ be a $3$-connected $D$-regular graph on $n$ vertices where $D \geq n/4$.
Let $\mathcal{X}$ be a robust partition of $G$ with parameters $\rho,\nu,\tau,2,1$.
Then $G$ contains a Hamilton cycle.
\end{lemma}

This --- the final case ---  is the longest and most difficult.
This is perhaps unsurprising given that the extremal example in Figure~\ref{fig:exactex}(i) has precisely this structure.
Moreover, the presence of a bipartite robust expander component means that the path system we find to join the robust components needs to be balanced with respect to the bipartite component -- the regularity of $G$ is essential to achieve this.
On the other hand, since we have to join up three components, the $3$-connectivity of $G$ is essential too.
The main challenge is to find a path system which satisfies both requirements simultaneously, i.e.~one that is both balanced and joins up the three components.
We need to invoke the degree bound $D \geq n/4$ for this.
We begin by giving a brief outline of the argument.

\subsection{Sketch of the proof of Lemma~\ref{(2,1)}}

Let $\mathcal{X} := \lbrace V_1',V_2',W' \rbrace$,\COMMENT{NEW entire proof sketch Sec 7.1} where $G[V_i']$ is a robust expander component for $i=1,2$, and $G[W']$ is a bipartite robust expander component.
Let $A',B'$ be a bipartition of $W$ such that $|A'|\ge |B'|$ and $G[W]$ is a bipartite robust expander component with bipartition $A',B'$ or $B',A'$.
(To be precise, $G[W']$ is a bipartite robust expander component with bipartition $A_{W'},B_{W'}$, where $\{A_{W'},B_{W'}\} = \{A',B'\}$.)\COMMENT{NEW29/5}
To prove Lemma~\ref{(2,1)}, Lemmas~\ref{balextend} and~\ref{HES} imply that it is sufficient to find an $\mathcal{X}$-tour $\mathcal{P}$ such that
\begin{itemize}
	\item[\rm (X1)] $\mathcal{P}$ contains few edges;
	\item[\rm (X2)] $R_{\mathcal{X}}(\mathcal{P})$ is an Euler tour;
	\item[\rm (X3)] $
2e_{\mathcal{P}}(A') - 2e_{\mathcal{P}}(B') + e_{\mathcal{P}}(A',U') - e_{\mathcal{P}}(B',U') = 2(|A'|-|B'|)$ holds, where $U' := V_1' \cup V_2'$.%
\end{itemize}
Note that (X1)--(X3) are independent of whether $G[W']$ is a bipartite robust expander component with bipartition $A',B'$ or $B',A'$.\COMMENT{NEW29/5}
Note that $\mathcal{P}$-edges in $G[A',U'] \cup G[A']$ count `positively' towards the goal of (X3), edges in $G[B',U'] \cup G[B']$ count `negatively', and all other edges are `neutral'.  
Therefore a natural approach to construct an $\mathcal{X}$-tour is to find two matchings $M_{A',U'}$ in $G[A',U']$ and $M_{A'}$ in $G[A']$ such that $\mathcal{P}_{\rm{match}} := M_{A',U'} \cup M_{A'}$ is an $\mathcal{X}$-tour (note that $\mathcal{P}_{\rm{match}}$ is always a path system).
Unfortunately, this may be impossible.
However, we can hope that there exists an $\mathcal{X}$-tour $\mathcal{P}$ \emph{most} of whose edges lie in the union of two such matchings.
In other words, we aim to construct matchings $M_{A'}$ and $M_{A',U'}$ which come as close as possible to satisfying (X1)--(X3).

Note that with the above approach, the requirement (X3) translates to $|M_{A',U'}| + 2|M_{A'}| = 2(|A'|-|B'|)$.
By Proposition~\ref{fact2}, \emph{any} partition $\lbrace U^*,A^*,B^* \rbrace$ of $V(G)$ satisfies
\begin{equation}\label{ABU*}
2e_{G}(A^*) - 2e_{G}(B^*) + e_{G}(A^*,U^*) - e_{G}(B^*,U^*) = D(|A^*|-|B^*|).
\end{equation}
To find $M_{A',U'}$ and $M_{A'}$ we will use Vizing's theorem on edge colourings, which guarantees a matching of size
$e(H)/(\Delta(H)+1)$ in a graph $H$, and K\"{o}nig's theorem, which guarantees a matching of size $e(H)/\Delta(H)$ in a bipartite graph~$H$.
Suppose first that
\begin{equation}\tag{MaxDeg}\label{maxdeg}
\Delta(G[A',U']), \Delta(G[A']) \leq D/2.
\end{equation}
This then implies that we can find $M_{A',U'}$ and $M_{A'}$ such that 
\begin{equation*}
|M_{A',U'}| + 2|M_{A'}| \ge \frac{ e_G(A',U')}{ D/2 } +  \frac{ 2 e_{G}(A') }{D/2+1}.
\end{equation*}
which is nearly at least $2(|A'|-|B'|)$ by~\eqref{ABU*}.
So by removing edges of $M_{A',U'}$ and $M_{A'}$ if necessary, we can ensure that $|M_{A',U'}|$ and $|M_{A'}|$ are
very close to the correct sizes.
Unfortunately $M_{A',U'} \cup M_{A'}$ may not satisfy~(X2).
In this case, we will modify $M_{A',U'} \cup M_{A'}$ to obtain the desired $\mathcal{P}$.

The above illustrates that (\ref{maxdeg}) is an important constraint, which we would like to achieve and apply. 
However, our robust partition $\mathcal{X}$ does not necessarily satisfy \eqref{maxdeg}: by (D4), $\Delta(G[A',U'])$ could be as large as
$2D/3$, for example when $a \in A'$ satisfies $d_{V_1'}(a),d_{V_2'}(a),d_{B'}(a) = D/3$.
For this reason, we will adjust the partition $\mathcal{X}$ slightly by moving a small number of vertices to obtain a weak robust partition $\mathcal{V} := \lbrace V_1,V_2,W := A \cup B \rbrace$ (where each part corresponds to its primed counterpart, and $|A| \geq |B|$) such that $\mathcal{V}$ \emph{does} satisfy (\ref{maxdeg}).
By Lemmas~\ref{HES} and~\ref{balextend} it is still sufficient to find $\mathcal{P}$ with the properties above, with $\mathcal{V}$ replacing $\mathcal{X}$.

We prove Lemma~\ref{(2,1)} separately in each of the following four cases: 
\begin{itemize}
\item[(i)] $|A|-|B| \geq 2$ and $e_G(A,\overline{W})$ is at least a little larger than $3D/2$ (Subsection~\ref{dense}); 
\item[(ii)] $|A|-|B| \geq 2$ and $e_G(A,\overline{W})$ is at most a little larger than $3D/2$ (Subsection~\ref{sparse}); 
\item[(iii)] $|A|-|B|=1$ (Subsection~\ref{+1});
\item[(iv)] $|A|=|B|$ (Subsection~\ref{equal}).
\end{itemize}
The reason for these distinctions will be discussed at the end of Subsection~\ref{tools}.
The full strength of the minimum degree bound $D \geq n/4$ is only used in the last two cases.

\subsubsection{A remark on a different approach}

Suppose for instance that we have $\Delta(G[A',U']), \Delta(G[A']) \leq D/100$ instead of \eqref{maxdeg}.
Then we could find much larger matchings $M_{A',U'}$ and $M_{A'}$, and would have more freedom when choosing suitable edges from them to add to $\mathcal{P}$.
By~(D7), there are very few vertices in each component with many neighbours in other components, so moving these vertices would not reduce our expansion parameters by much.
So one could hope to proceed as follows:

If there exists $a \in A'$ with, say, at least $D/100$ neighbours in $A'$, move $a$ to $B'$.
If, for $i \in \lbrace 1,2 \rbrace$, there exists $a \in A'$ with at least $D/200$ neighbours in $V'_i$, move $a$ to $V'_i$.
If there exists $v \in V'_i$ with at least $D/100$ neighbours in $A'$, move $v$ to $B'$.   
After every step, we still have a weak robust partition.
Continue until we have a weak robust partition for which one of the following holds:
\begin{itemize}
\item[(a)] $\Delta(G[A',U']),\Delta(G[A']) \leq D/100$ and $|A'|-|B'| \geq 2$;
\item[(b)] $|A'|-|B'| \in \lbrace 0,1\rbrace$ and (\ref{maxdeg}) holds.
\end{itemize}
The idea would be that one could then replace cases (i) and (ii) above by the easier case~(a).

However, such a process runs into difficulties as illustrated by the following example.
Suppose that we have started the process above with partition $A'_{\text{old}},B'_{\text{old}},U'_{\text{old}}$ and have arrived
at a partition $A',B',U'$ which satisfies the following properties: $G[A']$ contains a triangle $a_1a_2a_3 \in A'$
such that $d(a_i,A')=2$, $d (a_i,B') = D/2-2$ and $d (a_i,U') = D/2$ for each $i\le 3$. Moreover,
$|A'| = |B'|+2$, $e_G(A') = e_G(B') + D/4$, $e_G(A',U')=3D/2$ and $e_G(B',U') = 0$.
(Note that $\eqref{ABU*}$ holds and $\Delta(G[A',U']) = D/2$.)
Thus we aim to move one or two vertices from $A'$.
If we move $a_1$ to $B'$, then $\delta(G[A' \setminus \{a_1\}, B' \cup \{a_1\}]) = 2$ and so we no longer have a weak robust partition.
If we move $a_1$ to $U'$, then $d_G(a_2,U' \cup \{a_1\}) = D/2+1$, and so \eqref{maxdeg} fails.
A similar argument holds if we move two vertices from~$A'$ to other classes.
If $\Delta(G[B']) \le D/2$, then we could try to avoid this issue by first moving two or three of the $a_i$ to $U'$, and then
swapping $A'$ and~$B'$ to obtain a partition satisfying~(b).
However, since we might have moved several vertices $a \in A'_{\text{old}}$ (which might for instance satisfy $d(a,A'_{\text{old}})= D/5$ and $d (a,B'_{\text{old}}) = 4D/5$) to $B'$ at some earlier steps, we do not have any control on $\Delta(G[B'])$.
Thus it is not clear that we can swap $A'$ and~$B'$.
Moreover, one can modify the example to violate~\eqref{maxdeg} by a larger number, which is $o(n)$ say.

\subsection{Notation}\label{notation3}

Throughout%
     \COMMENT{DK: rewrote this para. Formally we would need to be more careful, for example, we need that Lemma~\ref{2,2} also holds if the
partition classes are slightly smaller than $n/5$, so that we can apply it later on to the partition $\mathcal{V}'$ obtained from
$\mathcal{V}$ by moving a few vertices (eg in the proof of Lemma~\ref{1,3}). But I'd suggest to gloss over this...}
the remainder of the paper, whenever we say that a graph $G$ has vertex partition $\mathcal{V} = \lbrace V_1,V_2,W := A \cup B \rbrace$,
we assume that $V(G)$ has a partition into parts $V_1,V_2,W$, each of size at least $|V(G)|/100\ge 100$,
that $A$ and $B$ are disjoint and that $|A|\ge |B|$. 
Moreover, we will say that $G$ has a weak robust partition $\mathcal{V} = \lbrace V_1,V_2,W := A \cup B \rbrace$
(for some given parameters) if $\mathcal{V}$ satisfies the above properties and is a weak robust partition of $G$
such that $G[V_1], G[V_2]$ are two robust expander components and $G[W]$ is a bipartite robust expander component,
and the bipartition of $W$ as specified by (D3$'$) is $A,B$ or $B,A$.\COMMENT{NEW29/5}
We will use a similar notation when $\mathcal{V}$ is a robust partition of~$G$.

Given $0 < \eps < 1$ and $\Delta > 0$, consider any graph $G$ with vertex partition $U,A,B$ such that $\Delta(G[A]),\Delta(G[A,U]) \leq \Delta$. We say that%
   \COMMENT{DK had $\char_{\Delta,\eps}(G) = \char_{\eps}(G):= (\ell,m)$ before, but I don't think we ever use $\char_{\eps}(G)$}
\begin{equation}\label{character}
\char_{\Delta,\eps}(G) 
:= (\ell,m)
\end{equation}
when
$\ell := \lceil e_G(A)/\Delta \rceil_{\eps}$ and $m$ is the largest even integer less than or equal to $\lceil e_G(A,U)/\Delta \rceil_{\eps}$.
(Recall the definition of $\lceil \cdot \rceil_{\eps}$ from the end of Subsection~\ref{notation}.)
Given any path system $\mathcal{P}$ in $G$, we write
\begin{equation}\label{bal}
\bal_{AB}(\mathcal{P}) := e_{\mathcal{P}}(A) - e_{\mathcal{P}}(B) + (e_{\mathcal{P}}(A,U) - e_{\mathcal{P}}(B,U))/2.
\end{equation}%
\COMMENT{we need $\bal_{AB}(\mathcal{P}) = |A|-|B|$ in order to apply Lemma~\ref{balextend} and modify $\mathcal{P}$ into a $\mathcal{V}$-tour. }
When $\mathcal{V} = \lbrace V_1,V_2, W:= A \cup B \rbrace$ is a vertex partition of $G$, we take $U := V_1 \cup V_2$ in the definitions of $\char_{\Delta,\eps}$ and $\bal_{AB}$.

It may be helpful to motivate these two crucial pieces of notation.
We think of `char' as being short for `character'.
The character of $G$ encodes what sort of $\mathcal{V}$-tour $\mathcal{P}$ we can hope to find.
Typically, when $G$ has character $(\ell,m)$, a $\mathcal{V}$-tour will closely resemble the union of a matching of size $\ell$ in $G[A]$, and a matching of size $m$ in $G[A,U]$.
(Recall that, in a $\mathcal{V}$-tour $\mathcal{P}$, we have that $e_{\mathcal{P}}(W,U)$ is even.)
The character of $G$ together with Vizing's and K\"onig's theorems guarantee that we can find such matchings.
The notion `bal' is a measure of the `balancedness' of a path system $\mathcal{P}$.
One of our aims will be to find $\mathcal{P}$ with $\bal_{AB}(\mathcal{P})=|A|-|B|$ (see (P2) below).%
\COMMENT{NEW paragraph}

Given $0 < \eps < 1$, $\Delta > 0$ and a graph $G$ with partition $\mathcal{V} = \lbrace V_1,V_2,W := A \cup B \rbrace$ and $\char_{\Delta,\eps}(G)=(\ell,m)$, we will find a path system satisfying the following properties:%
\COMMENT{Note that $\mathcal{P}$ satisfying (P1)--(P3) is \emph{not} a $\mathcal{V}$-tour since we need to apply Lemma~\ref{balextend} first.}

\begin{itemize}
\item[(P1)] $e(\mathcal{P}) \leq \ell+m + 6$;%
\COMMENT{we do not need to explicitly parametrise the robust partition.}
\item[(P2)] $\bal_{AB}(\mathcal{P}) = |A|-|B|$;
\item[(P3)] $R_{\mathcal{V}}(\mathcal{P})$ is an Euler tour.
\end{itemize}


\subsection{Preliminaries and a reduction}\label{reduction}

In this subsection we show that, in order to prove Lemma~\ref{(2,1)}, it is sufficient to prove Lemma~\ref{aim} below.
We then state some tools which will be used in the next subsections to do so.
The following observation provides us with a convenient check for a path system $\mathcal{P}$ to be such that $R_{\mathcal{V}}(\mathcal{P})$ is an Euler tour.%
   \COMMENT{DK had "into at most three parts" before, but if $|\mathcal{V}|=1$ then the fact is not necessarily true}

\begin{fact}\label{eulertour}
Let $G$ be a graph with vertex partition $\mathcal{V}$ into three parts.%
\COMMENT{We were claiming that a graph (cf. reduced graph) is connected if and only if each each vertex has non-zero degree. This is only true for graphs with at most three vertices.}
Then, for a path system $\mathcal{P}$ in $G$, \emph{(P3)} is equivalent to the following.
For each $X \in \mathcal{V}$,
$e_{\mathcal{P}}(X,\overline{X})$ is even and there exists $X' \in \mathcal{V} \setminus \lbrace X \rbrace$ such that $\mathcal{P}$ contains an $XX'$-path.
\end{fact}

The remainder of Section~\ref{sec21} is devoted to the proof of the following lemma, which states that $G$ contains a path system satisfying (P1)--(P3) (when the partition $\mathcal{V}$ and the parameters involved are suitably defined).

\begin{lemma}\label{aim}
Let $n,D \in \mathbb{N}$ and $\ell,m \in \mathbb{N}_0$.
Let $0 < 1/n \ll \rho \ll \nu \ll \tau \ll \eps \ll 1$.
Let $G$ be a $3$-connected $D$-regular graph on $n$ vertices where $D \geq n/4$.
Suppose that $G$ has a weak robust partition~$\mathcal{V} = \lbrace V_1,V_2, W:= A \cup B \rbrace$ with parameters $\rho,\nu,\tau,1/16,2,1$
such that $|V_1|,|V_2|\ge D/2$ and $|A|\ge |B|$.%
    \COMMENT{DK: added $|V_1|,|V_2|\ge D/2$. This of course follows since $G[V_i]$ is a robust component.
But now we don't have to refer to Lemma~\ref{comp} whenever we want to use this fact.
Added the same condition in the Claim of the proof of Lemma~\ref{(2,1)} below.}
Suppose further that $\Delta(G[A,V_1 \cup V_2]) \leq D/2$, $d_{V_i}(x_i) \geq d_{V_j}(x_i)$ for all $x_i \in V_i$ and all $\lbrace i,j \rbrace =\lbrace 1,2 \rbrace$, and $d_A(a) \leq d_B(a)$ for all $a \in A$.
Let $\char_{D/2,\eps}(G)=(\ell,m)$.
Then $G$ contains a path system $\mathcal{P}$ satisfying \emph{(P1)--(P3)}.
\end{lemma}

The following proposition gives bounds on $\ell$ and $m$ when $\char_{\Delta,\eps}(G) = (\ell,m)$.

\begin{proposition}\label{ell}
Let $n,D \in \mathbb{N}$ and $\ell,m \in \mathbb{N}_0$.
Let $0 < 1/n \ll \rho \ll \nu \ll \tau \ll \eps,\eta \ll 1$ and suppose $D \geq n/4$.
Let $G$ be a graph on $n$ vertices with weak robust partition $\mathcal{V} = \lbrace V_1, V_2, W:= A \cup B \rbrace$ with parameters $\rho,\nu,\tau,\eta,2,1$.
Suppose further that $\Delta(G[A]),\Delta(G[A,V_1 \cup V_2]) \leq D/2$ and that $\char_{D/2,\eps}(G)=(\ell,m)$.
Then $\ell, m \leq 12\rho n$.
\end{proposition}

\begin{proof}
(D3$'$)%
    \COMMENT{DK: new sentence}
implies that $G[W]$ is $\rho$-close to bipartite with bipartition $A,B$. So $e_G(A) + e_G(A,V_1 \cup V_2) \leq \rho n^2$.
Thus
$
\ell = \lceil 2e_G(A)/D \rceil_{\eps} \leq 3\rho n^2/D \leq 12\rho n
$.
An almost identical calculation gives the same bound for $m$.
\end{proof}

We now show that, to prove Lemma~\ref{(2,1)}, it suffices to prove Lemma~\ref{aim}.

\medskip
\noindent
\emph{Proof of Lemma~\emph{\ref{(2,1)}} (assuming Lemma~\emph{\ref{aim}}).}
Choose $\eps$ with $\tau \ll \eps \ll 1$.%
\COMMENT{Need to choose $\eps$ here since it is not defined in the statement of Lemma~\ref{(2,1)}.}
Let $\mathcal{X} = \lbrace U_1,U_2,W' := A' \cup B' \rbrace$ be a robust partition of $G$ with parameters $\rho,\nu,\tau,2,1$,
where $G[U_1], G[U_2]$ are $(\rho,\nu,\tau)$-robust expander components and $G[W']$ is a bipartite $(\rho,\nu,\tau)$-robust expander component with bipartition $A',B'$ as guaranteed by (D3).
We will alter $\mathcal{X}$ slightly so that it is a weak robust partition and that additionally the degree conditions of Lemma~\ref{aim} hold.

\medskip
\noindent
\textbf{Claim.}
\emph{There exists a weak robust partition $\mathcal{V} = \lbrace V_1,V_2,W := A \cup B \rbrace$ of $G$ with parameters $\rho^{1/3},\nu/2,2\tau,1/16,2,1$
such that $|V_1|,|V_2|\ge D/2$, $|A|\ge |B|$, $\Delta(G[A,V_1 \cup V_2]) \leq D/2$, $d_{V_i}(x_i) \geq d_{V_j}(x_i)$ for all $x_i \in V_i$ and $\lbrace i,j \rbrace =\lbrace 1,2 \rbrace$, and $d_A(a) \leq d_B(a)$ for all $a \in A$.}

\medskip
\noindent
To prove the claim, for $i=1,2$, let $X_i$ be the collection of vertices $x \in U_i$ with $d_{\overline{U_i}}(x) > \rho n$.
Then (D7) implies that $|X_i| \leq \rho n$.
Let $Y_i := U_i \setminus X_i$.
Then each $y \in Y_i$ satisfies
\begin{equation}\label{dYi}
d_{Y_i}(y) = d(y) - d_{\overline{U_i} \cup X_i}(y) \geq d(y) - \rho n - |X_i| \geq d(y) - 2\rho n.
\end{equation}
Let $A_0$ be the collection of vertices $a \in A'$ such that $d_{\overline{B'}}(a) \geq \sqrt{\rho} n$.
Let $A_1 := A' \setminus A_0$.
Define $B_0,B_1$ analogously.
By (D3), $G[W']$ is $\rho$-close to bipartite with bipartition $A',B'$.%
   \COMMENT{DK: had $W$, $A,B$ before}
Therefore (C3) holds, from which one can easily derive that $|A_0|,|B_0| \leq 2\sqrt{\rho}n$.%
\COMMENT{If say $|A_0| > 2\sqrt{\rho}n$, then $e_G(A,\overline{B}) \geq |A_0|\sqrt{\rho}n/2 > \rho n^2$, contradicting (C3).}
Similarly as in (\ref{dYi}), for each $a \in A_1$ and $b \in B_1$ we have
\begin{equation}\label{dA1}
d_{B_1}(a) \geq d(a) - 3\sqrt{\rho}n\ \ \mbox{ and } \ \ d_{A_1}(b) \geq d(b) - 3\sqrt{\rho}n.
\end{equation}

Let $V_0 := X_1 \cup X_2 \cup A_0 \cup B_0$.
Then
\begin{equation}\label{V0}
|V_0| \leq 5\sqrt{\rho}n.
\end{equation}
Among all partitions $X_1',X_2',A_0',B_0'$ of $V_0$, choose one such that $e(A \cup B,V_1 \cup V_2)$ is minimised; and subject to $e(A \cup B,V_1 \cup V_2)$ being minimal we have that $e(V_1, V_2)+e(A)+e(B)$ is minimal,%
\COMMENT{this has been simplified} where $V_i := Y_i \cup X_i'$, $A := A_1 \cup A_0'$ and $B := B_1 \cup B_0'$.
It is easy to see that $d_{A \cup B}(w) \geq d_{V_1 \cup V_2}(w)$ for all $w \in A_0' \cup B_0'$; $d_{V_1 \cup V_2}(v) \geq d_{A \cup B}(v)$
for all $v \in X_1' \cup X_2'$; $d_{V_i}(v_i) \geq d_{V_j}(v_i)$ for all $v_i \in X_i'$ and $\lbrace i,j \rbrace = \lbrace 1,2 \rbrace$; $d_A(a) \leq d_B(a)$ for all $a \in A_0'$; and $d_B(b) \leq d_A(b)$ for all $b \in B_0'$.
If $v_i \in Y_i$, then (\ref{dYi}) implies that $d_{V_i}(v_i) \geq d_{Y_i}(v_i) \geq d(v_i)-2\rho n \geq d(v_i)/2$.
So $d_{V_i}(v_i) \geq d_{A \cup B}(v_i), d_{V_j}(v_i)$ for $\{i,j\}=\{1,2\}$.
Similarly, (\ref{dA1}) implies that, for all $w \in A_1 \cup B_1$ we have $d_{A \cup B}(w) \geq d_{V_1 \cup V_2}(w)$; for all $a \in A_1$ we have $d_{A}(a) \leq d_B(a)$ and for all $b \in B_1$ we have $d_B(b) \leq d_A(b)$.
Observe that (\ref{dYi}), (\ref{dA1}) imply that $|V_i| \geq D - 2\rho n$ and $|A|,|B| \geq D - 3\sqrt{\rho}n$ respectively.%
\COMMENT{LATE CHANGE: swapped sentences round because we need $|V_i| \geq D/2$ in addition.}
Then the degree conditions required in the claim hold, and they also hold with $B,A$ playing the roles of $A,B$ respectively.%
\COMMENT{NEW 30/5}
We now prove that $\mathcal{V} := \lbrace V_1, V_2, W := A \cup B \rbrace$ is a weak robust partition with parameters $\rho^{1/3},\nu/2,2\tau,1/16,2,1$.
Property (D1$'$) is clear. We now prove (D2$'$).
Observe that%
   \COMMENT{DK: added $+D|X_i|$ to count $e(Y_i,\overline{V_i}\cap X_i)$ in next inequality}
$$
e(V_i,\overline{V_i}) \leq e(U_i,\overline{U_i}) +D|X_i|+ D|X_i'| \leq (\rho + 6\sqrt{\rho})n^2 \leq \rho^{1/3}n^2.
$$
Therefore each $V_i$ is a $\rho^{1/3}$-robust component of $G$.
Note also that
$$
|V_i \triangle U_i| \leq |V_0| \stackrel{(\ref{V0})}{\leq} 5\sqrt{\rho}n \leq \nu|U_i|/2.
$$
Lemma~\ref{expanderswallow} implies that $G[V_i]$ is a $(\nu/2,2\tau)$-robust expander.
Therefore $G[V_i]$ is a $(\rho^{1/3},\nu/2,2\tau)$-robust expander component for $i=1,2$, so (D2$'$) holds.
To prove (D3$'$), note that
$|A \triangle A'| + |B \triangle B'| \leq 2|V_0| \leq \rho^{1/3}n/3$ where the final inequality follows from (\ref{V0}).
Now Lemma~\ref{BREadjust} implies that $G[A \cup B]$ is a bipartite $(\rho^{1/3},\nu/2,2\tau)$-robust expander component of $G$ with bipartition $A,B$.
Thus (D3$'$) holds.
Finally, (D4$'$) and (D5$'$) are clear from the degree conditions we have already obtained.
Finally, if necessary, relabel $A$ and $B$ so that $|A| \geq |B|$.
Then, as previously remarked, the degree conditions of the claim hold.%
\COMMENT{NEW 30/5}
This completes the proof of the claim.

\medskip
\noindent
Given%
   \COMMENT{DK: previously para started with "Let $\Delta := D/2$.
Now (D3$'$) implies that $G[A \cup B]$ is $\rho^{1/3}$-close to bipartite, and therefore $e_G(A) + e_G(A,V_1 \cup V_2) \leq \rho^{1/3} n^2$."}
the partition $\mathcal{V}$ of $V(G)$, let $\ell,m$ satisfy $\char_{D/2,\eps}(G) = (\ell,m)$.
Let $\mathcal{P}$ be a path system in $G$ guaranteed by Lemma~\ref{aim}, i.e.~$\mathcal{P}$ satisfies (P1)--(P3).
Note%
   \COMMENT{DK: new sentence}
that $\mathcal{V}$ is also a weak robust partition with parameters $\rho^{1/3},\nu/2,2\tau,\eps,2,1$.
So (P1) and Proposition~\ref{ell} with $\rho^{1/3},\eps$ playing%
   \COMMENT{DK: had $1/16$ instead of $\eps$ before}
the roles of $\rho,\eta$ imply that $e(\mathcal{P}) \leq 25\rho^{1/3}n$.
Then, for each $X \in \mathcal{V}$ we have that $|V(\mathcal{P}) \cap X| \leq |V(\mathcal{P})| \leq 2e(\mathcal{P}) \leq 50\rho^{1/3}n \leq \rho^{1/4}n/9$.
So Lemma~\ref{balextend} applied with $2,1,W,\{A,B\},\mathcal{P},\rho^{1/4}/9$ playing the roles of $k,\ell,W_j,\{A_j,B_j\},\mathcal{P},\rho$\COMMENT{NEW29/5}
implies that $G$ contains a path system $\mathcal{P}'$ that is a $\mathcal{V}$-tour with parameter $\rho^{1/4}$.
Now Lemma~\ref{HES} with $\mathcal{P}',\rho^{1/3},\rho^{1/4},\nu/2,2\tau,1/16,2,1$ playing the roles of $\mathcal{P},\rho,\gamma,\nu,\tau,\eta,k,\ell$ implies that $G$ contains a Hamilton cycle.
\hfill$\square$

\subsection{Tools}\label{tools}

In this section we gather some useful tools which will be used repeatedly in the sections to come.
We will often use the following lower bounds for $e_G(A),e_G(A,U)$ implied by $\char_{\Delta,\eps}(G)$.%
    \COMMENT{DK deleted $\Delta'/\Delta \ll 1$ in the next prop since we don't need it}

\begin{proposition}\label{charedges}
Let $\Delta,\Delta' \in \mathbb{N}$ and $\ell,m \in \mathbb{N}_0$.  
Let $\Delta'/\Delta \leq \eps < 1$.
Suppose that $G$ is a graph with vertex partition $U,A,B$ such that $\Delta(G[A]),\Delta(G[A,U]) \leq \Delta$ and $\char_{\Delta,\eps}(G)=(\ell,m)$.
Then $e_G(A) \geq (\ell-1)\Delta+\Delta'$ and $e_G(A,U) \geq (m-1)\Delta+\Delta'$.
\end{proposition}

\begin{proof}
We have that
$
\ell = \lceil e_G(A)/\Delta \rceil_{\eps} = \lceil e_G(A)/\Delta - \eps \rceil$
so $\ell-1 < e_G(A)/\Delta - \eps \leq (e_G(A)-\Delta')/\Delta$, as required.
A near identical calculation proves the second assertion.
\end{proof}

The path system we require will contain edges in $G[A]$ and $G[V_1 \cup V_2,A]$, and will `roughly look like' a matching within each of these subgraphs.
The following lemma allows us to find a structure which in turn contains a large matching even if certain vertices need to be avoided.%
   \COMMENT{DK: replaced $1/\Delta \ll 1/\Delta' \ll 1$ with $\Delta'/\Delta, 1/\Delta' \ll 1$ in the prop and made similar changes
   in the statements of the other lemmas (since both $\Delta$ and $\Delta'$ will later be linear in $n$, we don't have that $1/\Delta \ll 1/\Delta'$. We
need $ 1/\Delta' \ll 1$ in Lemma~\ref{goodmatching2} since this doesn't follow from $\ell/\Delta'\ll 1$ in the case when $\ell=0$.}

\begin{lemma}\label{goodmatching2}
Let $\Delta,\Delta' \in \mathbb{N}$ and $\ell \in \mathbb{N}_0$ be such that $\ell/\Delta',\Delta'/\Delta, 1/\Delta' \ll 1$.
Let $G$ be a graph with $\Delta(G) \leq \Delta$, and let $e(G) \geq (\ell-1)\Delta+\Delta'$.
Then $G$ contains one of the following:%
\COMMENT{LATE CHANGE: new wording.}
\begin{itemize}
\item[(i)] a matching $M$ of size $\ell+1$ and $uv \in E(G)$ with $u \notin V(M)$;
\item[(ii)] $\ell$ vertices each with degree at least $\Delta'$.
\end{itemize}
Moreover, if $\ell \geq 1$ and $e(G) \geq \ell\Delta+1$; or $\ell=0$ and $e(G) \geq 2$, then \emph{(i)} holds.
\end{lemma}

\begin{proof}
We will use induction on $\ell$ in order to show that either (i) or (ii) holds.
The cases $\ell=0,1$ are trivial.
Suppose now that $\ell \geq 2$.%
\COMMENT{We cannot only suppose $\ell \geq 1$ for the next part of the argument to work.}
Suppose first that $\Delta(G) \leq \Delta'$.
Then, by Vizing's theorem, $E(G)$ 
can be properly coloured with at most $\Delta'+1$ colours.%
\COMMENT{
Now $(\ell-1)\Delta-\Delta' > (\ell+1)(\Delta'+1)$.
This wasn't in a comment before.}
Therefore $G$ contains a matching of size
$$
\left\lceil \frac{e(G)}{\Delta'+1} \right\rceil \geq \left\lceil \frac{(\ell-1)\Delta + \Delta'}{\Delta'+1} \right\rceil \geq \ell+2.
$$
So (i) holds. 
Thus we may assume that there exists $x \in V(G)$ with $d(x) \geq \Delta'$.
Let $G^- := G \setminus \lbrace x \rbrace$.
Then $e(G^-) \geq e(G)-\Delta \geq (\ell-2)\Delta + \Delta'$.
By induction, $e(G^-)$ contains either a matching $M^-$ of size $\ell$ and $uv \in E(G^-)$ with $u \notin V(M^-)$, or $\ell-1$ vertices of degree at least $\Delta'$.
In the first case, choose $y \in N(x) \setminus V(M^-)$ with $y \neq u$ and let $M := M^- \cup \lbrace xy \rbrace$.
Then (i) holds.
In the second case, $x$ is our $\ell$th vertex of degree at least $\Delta'$ in $G$, so (ii) holds.

For the moreover part, suppose now that $\ell \geq 1$ and $e(G) \geq \ell\Delta+1$.
Suppose that (i) does not hold.
Let $x_1,\ldots,x_{\ell}$ be $\ell$ distinct vertices of degree at least $\Delta'$.
Then $e(G \setminus \lbrace x_1,\ldots, x_{\ell} \rbrace) \geq e(G)-\Delta\ell \geq 1$.
So $G$ contains an edge $e$ which is not incident to $\lbrace x_1, \ldots, x_{\ell} \rbrace$.
We obtain a contradiction by considering $\lbrace e, x_1z_1 \rbrace \cup \lbrace x_1y_1, \ldots, x_{\ell}y_{\ell} \rbrace$, where $z_1 \in N(x_1)$ avoids $e$ and for $1 \leq i \leq \ell$ the vertices $y_i \in N(x_i)$ are distinct, and avoid $e$, $z_1$ and $x_1, \ldots, x_{\ell}$.

Finally, if $\ell=0$, then any two edges of $G$ satisfy (i).
\end{proof}


Given an even matching $M$ in $G[A,V_1 \cup V_2]$ and a lower bound on $e_G(A)$, we would like to extend $M$ into a path system $\mathcal{P}$ using edges from $G[A]$ so that $\bal_{AB}(\mathcal{P})$ is large.
Lemma~\ref{goodmatching2} gives us two useful structures in $G[A]$ from which we can choose suitable edges to add to $M$ to form $\mathcal{P}$.
The following proposition does this in the case when the consequence~(i) of Lemma~\ref{goodmatching2} holds.

\begin{proposition}\label{casei}
Let $G$ be a graph with vertex partition $X,Y$.
Suppose that $G[Y]$ contains a matching $M'$ of size $\ell+1$ and an edge $uv$ with $u \notin V(M')$.
Let $M$ be a non-empty even matching of size $m$ in $G[X,Y]$.
Then $G$ contains a path system $\mathcal{P}$ such that%
   \COMMENT{DK: added "at least" in (iii)}
\begin{itemize}
\item[(i)] $\mathcal{P}[X,Y]=M$ and $\mathcal{P} \subseteq M \cup M' \cup \lbrace uv \rbrace$;
\item[(ii)] $e_{\mathcal{P}}(Y)=\ell+1$;
\item[(iii)] $\mathcal{P}$ contains at least two $XY$-paths.
\end{itemize}
\end{proposition}

\begin{proof} 
We will extend $M$ by adding edges from $M' \cup \lbrace uv \rbrace$, so (i) automatically holds.
Note that any path system $\mathcal{P}$ obtained in this way contains an even number of $XY$-paths.
So it suffices to find such a $\mathcal{P}$ with at least one $XY$-path.
If $M \cup M'$ contains an $XY$-path, then we are done by setting $\mathcal{P} := M \cup M'$.
So suppose not.
Then $M'[V(M) \cap Y]$ is a perfect matching $M''$.
If $v \in V(M'')$, let $f$ be the edge of $M''$ containing $v$.
Otherwise, let $f \in E(M'')$ be arbitrary.
We take $\mathcal{P} := M \cup M' \cup \lbrace uv \rbrace \setminus \lbrace f \rbrace$.
Now both of the two edges in $M$ which are incident to $f$ lie in distinct $XY$-paths of $\mathcal{P}$, so (iii) holds.
Clearly (ii) holds too.
\end{proof}

Following on from the previous proposition, we now consider how to extend $M$ into $\mathcal{P}$ when instead the consequence~(ii) of Lemma~\ref{goodmatching2} holds in $G[A]$.

\begin{proposition}\label{2paths}
Let $\Delta' \in \mathbb{N}$ and let $\ell,m,r \in \mathbb{N}_0$ with $\Delta' \geq 3\ell + m$.
Let $G$ be a graph with vertex partition $X,Y$ and let $M$ be a matching in $G[X,Y]$ of size $m$.
Let $\lbrace x_1, \ldots, x_\ell \rbrace \subseteq Y$ such that $d_Y(x_i) \geq \Delta'$ and $|\lbrace x_1, \ldots, x_\ell \rbrace \setminus V(M)| \geq r$.
Then there exists a path system $\mathcal{P} \subseteq G[X,Y] \cup G[Y]$ such that $e_{\mathcal{P}}(Y) = \ell+r$, $\mathcal{P}[X,Y] = M$ and every edge of $M$ lies in a distinct $XY$-path in $\mathcal{P}$.  
\end{proposition}

\begin{proof}
Since $\Delta' \geq 3\ell+m$, $G[Y]$ contains a collection of $\ell$ vertex-disjoint paths $P_1, \ldots, P_\ell$ of length two with midpoints $x_1, \ldots, x_\ell$ respectively, such that $V(P_i) \cap V(M) \subseteq \lbrace x_i \rbrace$.
For each $x_i \in V(M)$, delete one arbitrary edge from $P_i$.
Let $\mathcal{P}$ consist of $M$ together with $P_1, \ldots, P_\ell$.
Then $\mathcal{P}$ is a path system, and every edge of $M$ lies in a distinct $XY$-path.
Moreover, $e_{\mathcal{P}}(Y) \geq 2\ell - (\ell-r) = \ell +r$.
Delete additional edges from $\mathcal{P}[Y]$ if necessary.
\end{proof}

\begin{proposition}\label{rounding2}
Let $0 < \eps < 1/3$.
Let $a, b \in \mathbb{R}_{\geq 0}$
and let $x \in \mathbb{N}_0$.
Suppose that
$
2a + b  \geq 2x.
$
Let $a' := \lceil a \rceil_{\eps}$ and let $b'$ be the largest even integer of size at most $\lceil b \rceil_{\eps}$.
Then $a',b' \geq 0$ and
$
2a' + b' \geq 2x.
$
\end{proposition}

\begin{proof}
Note that
$$
2\lceil a \rceil_{\eps} + \lceil b \rceil_{\eps} = 2\lceil a-\eps \rceil + \lceil b - \eps \rceil \geq \lceil 2a-2\eps + b - \eps \rceil \geq \lceil 2x - 3\eps \rceil \geq 2x.
$$
This implies the proposition.%
\COMMENT{
If $\lceil b \rceil_{\eps}$ is even, have $a' := \lceil a \rceil_{\eps}$ and $b' := \lceil b \rceil_{\eps}$.
Otherwise, $2\lceil a \rceil_{\eps} + \lceil b \rceil_{\eps}$ is odd and at least $2x$, so $2\lceil a \rceil_{\eps} + \lceil b \rceil_{\eps} - 1 \geq 2x$.
Note that, since $b \geq 0$, $\lceil b \rceil_{\eps} = \lceil b-\eps \rceil \geq \lceil b-1/3 \rceil \geq \lceil -1/3 \rceil = 0$. But $\lceil b \rceil_{\eps}$ is odd so
$\lceil b \rceil_{\eps} \geq 1$.
In this case we have $a' := \lceil a \rceil_{\eps}$ and $b' := \lceil b \rceil_{\eps} - 1$ (the latter assertion needs $\lceil b \rceil_{\eps} \geq 1$).}
\end{proof}

\begin{proposition}\label{matchingsizes}
Let $D \in \mathbb{N}$ and let $0 < \eps < 1/3$.%
    \COMMENT{DK  deleted $1/D \ll 1$}
Let $G$ be a $D$-regular graph and let $U,A,B$ be a partition of $V(G)$ where $|A| \geq |B|$.
Suppose that $\Delta(G[A,U]),\Delta(G[A]) \leq D/2$ and that $\char_{D/2,\eps}(G) = (\ell,m)$.
Then $\ell,m \geq 0$ and $\ell+m/2 \geq |A|-|B|$.
\end{proposition}

\begin{proof}
The consequence~(ii) of Proposition~\ref{fact2} implies that $4e(A)/D + 2e(A,U)/D \geq 2(|A|-|B|)$.
Apply Proposition~\ref{rounding2} with $2e(A)/D, 2e(A,U)/D, |A|-|B|$ playing the roles of $a,b,x$ to obtain $a',b'$.
Note that $a' = \ell$ and $b' = m$.
\end{proof}

We will first prove Lemma~\ref{aim} in the case when $|A|-|B| \geq 2$.
This constraint arises for the following reason.
We will show that we can find a path system~$\mathcal{P}$ such that $R_{\mathcal{V}}(\mathcal{P})$ is an Euler tour, but $\mathcal{P}$ is `overbalanced'.
More precisely, $\bal_{AB}(\mathcal{P})=\ell+m/2$, which is at least as large as $|A|-|B|$ by Proposition~\ref{matchingsizes}.
We would like to remove edges from $\mathcal{P}$ so that (P2) holds, and $R_{\mathcal{V}}(\mathcal{P})$ is still an Euler tour.
However, there exist path systems $\mathcal{P}_0$ such that $\bal_{AB}(\mathcal{P}_0)=2$, $R_{\mathcal{V}}(\mathcal{P}_0)$ is an Euler tour, but any $\mathcal{P}_0'$ with $E(\mathcal{P}_0') \subsetneqq E(\mathcal{P}_0)$ is such that $R_{\mathcal{V}}(\mathcal{P}_0')$ is not an Euler tour.
(For example, a matching of size two in $G[V_1,A]$ together with a matching of size two in $G[V_2,A]$, such that these edges are all vertex-disjoint.)
So, if $|A|-|B|<2$, we cannot guarantee, simply by removing edges, that we will ever be able to find $\mathcal{P}'$ with $\bal_{AB}(\mathcal{P}') =|A|-|B|$ without violating (P3).

We will split the case when $|A|-|B|\ge 2$ further into the subcases $m \geq 4$ and $m \leq 2$, i.e.~when $e_G(A,V_1 \cup V_2)$ is at least a little larger than $3D/2$, and when it is not.
We will call these the \emph{dense} and \emph{sparse} cases respectively.

\subsection{The proof of Lemma~\ref{aim} in the case when $|A|-|B| \geq 2$ and $m \geq 4$}\label{dense}

This subsection concerns the dense case
when $m \geq 4$, i.e.~when $e_G(A,V_1 \cup V_2)$ is at least slightly larger than $3D/2$.%
\COMMENT{If $e_G(A,U) \geq 3D/2+2\eps$ then $\lceil e_G(A,U) - \eps \rceil \geq \lceil 3+\eps \rceil = 4$.}
Now $G[A,V_1 \cup V_2]$ contains a matching $M$ of size $m$.
We will add edges to $M$ to obtain a path system $\mathcal{P}$ which satisfies (P1)--(P3).
If $M[A,V_i]$ is an even non-empty matching for both $i=1,2$, then $M$ satisfies (P3).
In every other case we must modify $M$ by adding and/or subtracting edges.
We do this separately depending on the relative values of $e_M(A,V_1)$ and $e_M(A,V_2)$.
We thus obtain a path system $\mathcal{P}_0$ which satisfies (P1) and (P3).
Then we obtain $\mathcal{P}$ by adding edges to $\mathcal{P}_0$ from $G[A]$ so that (P2) is also satisfied.
We must pay attention to the way in which these sets of edges interact to ensure that $\mathcal{P}$ still satisfies (P3).%
\COMMENT{I added a mini sketch..}

We begin with the subcase when $e_M(V_1,A),e_M(V_2,A)$ are both even and positive.%
   \COMMENT{DK: added the moreover part of the next lemma - useful for Lemmas~\ref{1,3} and~\ref{0,4}}

\begin{lemma}\label{2,2}
Let $\Delta,\Delta' \in \mathbb{N}$, $\ell \in \mathbb{N}_0$ and $m \in 2\mathbb{N}$ with $\Delta'/\Delta, m/\Delta', \ell/\Delta' \ll 1$.
Let $G$ be a graph with vertex partition $\mathcal{V} = \lbrace V_1,V_2, W:= A \cup B \rbrace$.
Let $M$ be a matching in $G[V_1 \cup V_2,A]$ of size $m$, and let $M_i := M[V_i,A]$ and $m_i := e(M_i)$.
Suppose that $\lbrace m_1,m_2 \rbrace \subseteq 2\mathbb{N}$.
Let $e(A) \geq (\ell-1)\Delta + \Delta'$ and $\Delta(G[A]) \leq \Delta$.
Then $G$ contains a path system $\mathcal{P}$ such that $\mathcal{P} \subseteq G[A] \cup G[A,V_1 \cup V_2]$,
$\mathcal{P}[A,V_1 \cup V_2] = M$, $e(\mathcal{P}) = \ell+m$, $R_{\mathcal{V}}(\mathcal{P})$ is an Euler tour and $\bal_{AB}(\mathcal{P}) = \ell + m/2$.
Moreover, $\mathcal{P}$ contains at least one $V_iA$-path for each $i=1,2$.
\end{lemma}

\begin{proof}
We will find $\mathcal{P}$ by adding suitable edges of $G[A]$ to $M$ such that
$\mathcal{P}$ contains at least one $V_iA$-path for each $i=1,2$.%
   \COMMENT{DK: reformulated first sentence}
Then by Fact~\ref{eulertour} we have that $R_{\mathcal{V}}(\mathcal{P})$ is an Euler tour. 
Apply Lemma~\ref{goodmatching2} to $G[A]$.
Suppose first that the consequence~(i) of Lemma~\ref{goodmatching2} holds.
Let $M'$ be a matching of size $\ell+1$ in $G[A]$ and let $uv \in E(G[A])$ be such that $u \notin V(M')$.
Then
\begin{equation}\label{bal22}
\bal_{AB}(M \cup M') = \ell+m/2+1\ \ \mbox{ and }\ \ e(M \cup M') = \ell + m + 1.
\end{equation}
If $M \cup M'$ contains a $V_iA$-path for both $i=1,2$ we are done by setting $\mathcal{P} := M \cup M' \setminus \lbrace e \rbrace$ where $e \in M'$ is arbitrary.
Suppose now that $M \cup M'$ contains a $V_1A$-path but no $V_2A$-path.
Then $V(M_2) \cap A \subseteq V(M')$.
Choose $e_2 \in E(M')$ with an endpoint in $V(M_2)$.
Then $\mathcal{P} := M \cup M' \setminus \lbrace e_2 \rbrace$ contains a $V_iA$-path for both $i=1,2$, and (\ref{bal22}) implies that $\bal_{AB}(\mathcal{P})=\ell+m/2$ and $e(\mathcal{P})=\ell+m$, as required.
The case when $M \cup M'$ contains a $V_2A$-path but no $V_1A$-path is identical.

So we may assume that $M \cup M'$ contains no $V_iA$-path for both $i=1,2$.
Suppose that there is $a_1a_2 \in E(M')$ with $a_i \in V(M_i)$.
Then $\mathcal{P} := M \cup M' \setminus \lbrace a_1a_2 \rbrace$ contains a $V_iA$-path with endpoint $a_i$ for $i=1,2$.
Moreover, (\ref{bal22}) implies that $\mathcal{P}$ satisfies the other conditions.
Therefore we may assume that $M'_i := M'[V(M_i) \cap A]$ is a (non-empty) perfect matching for $i=1,2$.
Choose $f_i \in E(M_i')$ for $i=1,2$ such that $v \in V(f_1)\cup V(f_2)$ if possible.
We set $\mathcal{P} := M \cup M' \cup \lbrace uv \rbrace \setminus \lbrace f_1,f_2 \rbrace$.
Note that every vertex in $V( f_i ) \setminus \lbrace v \rbrace$ is the endpoint of a $V_iA$-path in $\mathcal{P}$.
Then (\ref{bal22}) implies that $\bal_{AB}(\mathcal{P})=\bal_{AB}(M \cup M') +1-2 = \ell+m/2$ and $e(\mathcal{P})=\ell+m$, as required.

Suppose instead that the consequence~(ii) of Lemma~\ref{goodmatching2} holds and let $x_1, \ldots, x_\ell$ be $\ell$ distinct vertices in $A$ with $d_A(x_i) \geq \Delta'$ for all $1 \leq i \leq \ell$. 
Apply Proposition~\ref{2paths} with $G \setminus B, V_1 \cup V_2, A, M, x_i,0$ playing the roles of $G,X,Y,M,x_i,r$ to obtain a path system $\mathcal{P} \subseteq G[A] \cup G[A,V_1 \cup V_2]$ with $e_{\mathcal{P}}(A) = \ell$, $\mathcal{P}[A,V_1 \cup V_2]=M$ and such that every edge in $M$ lies in a distinct $AV_i$-path in $\mathcal{P}$ for some $i \in \lbrace 1,2 \rbrace$. 
Therefore $R_{\mathcal{V}}(\mathcal{P})$ is an Euler tour, $e(\mathcal{P})=\ell+m$, and since $V(\mathcal{P}) \cap B = \emptyset$ we have that $\bal_{AB}(\mathcal{P})=\ell+m/2$.
\end{proof}

We now consider the case when $e_M(V_1,A), e_M(V_2,A)$ are both odd and at least three.

\begin{lemma}\label{3,3}
Let $\Delta,\Delta' \in \mathbb{N}$, $\ell \in \mathbb{N}_0$ and $m \in 2\mathbb{N}$ with  $\Delta'/\Delta, m/\Delta', \ell/\Delta' \ll 1$.
Let $G$ be a graph with vertex partition $\mathcal{V} = \lbrace V_1,V_2, W:= A \cup B \rbrace$.
Let $m < e_G(V_1 \cup V_2,A)$, $e_G(A) \geq (\ell-1)\Delta + \Delta'$ and $\Delta(G[A]) \leq \Delta$.
Let $M$ be a matching in $G[V_1 \cup V_2,A]$ of size $m$, and let $M_i := M[V_i,A]$, $m_i := e(M_i)$.
Suppose $\lbrace m_1,m_2 \rbrace \subseteq 2\mathbb{N}+1$.
Then $G$ contains a path system $\mathcal{P}$ such that $e(\mathcal{P}) \leq \ell+m$, $R_{\mathcal{V}}(\mathcal{P})$ is an Euler tour and $\bal_{AB}(\mathcal{P}) = \ell + m/2$.
\end{lemma}

\begin{proof}
We will find $\mathcal{P}$ such that $e_{\mathcal{P}}(V_i,A) = e_{\mathcal{P}}(V_i,W)$ is even for $i=1,2$, $e_{\mathcal{P}}(V_1,V_2)=0$ and such that for each $X \in \mathcal{V}$, there exists $X' \in \mathcal{V} \setminus \lbrace X \rbrace$ such that $\mathcal{P}$ contains an $XX'$-path.
Then by Fact~\ref{eulertour} we have that $R_{\mathcal{V}}(\mathcal{P})$ is an Euler tour. 

Let us first suppose that $\ell=0$.
Since $m < e_G(V_1 \cup V_2,A)$, there exists an edge $e^+ \in G[V_1 \cup V_2,A] \setminus E(M)$.
Suppose, without loss of generality, that $e^+ \in G[V_1,A]$.
Let $e^-$ be an arbitrary edge in $M_2$.
Let $\mathcal{P} := M \cup \lbrace e^+ \rbrace \setminus \lbrace e^- \rbrace$.
Then $R_{\mathcal{V}}(\mathcal{P})$ is an Euler tour and $\bal_{AB}(\mathcal{P}) = (m_1+1)/2+(m_2-1)/2=m/2$, as required.

Therefore we assume that $\ell \geq 1$.
Apply Lemma~\ref{goodmatching2} to $G[A]$.
Suppose first that the consequence~(i) of Lemma~\ref{goodmatching2} holds.
So $G[A]$ contains a matching $M'$ of size $\ell+1$.
Note that it suffices to find $e_i \in M_i$ for $i=1,2$ such that $M \cup M' \setminus \lbrace e_1,e_2 \rbrace$ contains a $V_iA$-path for $i=1,2$.
Then it is straightforward to check that we are done by setting $\mathcal{P} := M \cup M' \setminus \lbrace e_1,e_2 \rbrace$.

We say that $xy \in E(G[A])$ is a \emph{connecting edge} if $x \in V(M_1)$ and $y \in V(M_2)$.
Suppose that $M'$ contains no connecting edge.
So $M \cup M'$ contains no $V_1V_2$-paths.
But an even number of edges in $M_i$ lie in $V_iV_i$-paths of $M \cup M'$.
Since $m_i$ is odd, there must be a $V_iA$-path $P_i$ in $M \cup M'$ for $i=1,2$.
We are done by choosing $e_i \in E(M_i) \setminus E(P_i)$ arbitrarily.

Therefore we may assume that there exists a connecting edge $a_1a_2 \in M'$, with $a_i \in V(M_i)$.
Suppose that there exists a second connecting edge $a_1'a_2' \in M'$, with $a_i' \in V(M_i)$.
Then we are done by choosing $e_1 \in M_1$ with endpoint $a_1$ and $e_2 \in M_2$ with endpoint $a_2'$.
Therefore we may suppose that $a_1a_2$ is the only connecting edge in $G$.
Let $P$ be the $V_1V_2$-path containing $a_1a_2$.
Let $\mathcal{P}' := (M \cup M') \setminus \lbrace E(P) \rbrace$.
Then, for each $i=1,2$, either $\mathcal{P}'$ contains a $V_iA$-path $P_{i,A}$, or a $V_iV_i$-path $P_{i,i}$.
In the first case, let $e_i$ be an arbitrary edge of $M_i$ that does not lie in $P_{i,A}$.%
\COMMENT{This is possible because $P_{i,A}$ contains at most one $V_iA$-edge and $e(M_i) \geq 3$.}
In the second case, let $e_i \in E(P_{i,i}) \cap E(M_i)$ be arbitrary.%
\COMMENT{LATE CHANGE: Else $e_i$ could have both endpoints in $A$.}

Suppose instead that the consequence~(ii) of Lemma~\ref{goodmatching2} holds in $G[A]$ and let $x_1, \ldots, x_\ell$ be $\ell$ distinct vertices in $A$ with $d_A(x_i) \geq \Delta'$ for all $1 \leq i \leq \ell$.
Since $\ell \geq 1$, we can choose $e_1 \in M_1$ and $e_2 \in M_2$ so that $\lbrace x_1, \ldots, x_\ell \rbrace \not\subseteq V(M \setminus \lbrace e_1,e_2 \rbrace)$.
Apply Proposition~\ref{2paths} with $G \setminus B, V_1 \cup V_2, A, M \setminus \lbrace e_1,e_2 \rbrace, x_i,1$ playing the roles of $G,X,Y,M,x_i,r$ to obtain a path system $\mathcal{P} \subseteq G[A] \cup G[A,V_1 \cup V_2]$ such that $e_{\mathcal{P}}(A)=\ell+1$, $\mathcal{P}[A, V_1 \cup V_2] = M \setminus \lbrace e_1,e_2 \rbrace$, and every edge in $M \setminus \lbrace e_1,e_2 \rbrace$ lies in a distinct $AV_i$-path in $\mathcal{P}$ for some $i \in \lbrace 1,2 \rbrace$.
Then $e(\mathcal{P})=\ell+m-1$ and $\bal_{AB}(\mathcal{P})=\ell+1+(m-2)/2=\ell+m/2$. 
Since $\mathcal{P}[A,V_i]$ is an even matching for $i=1,2$ and $\mathcal{P}[V_1,V_2]$ is empty, we have that $R_{\mathcal{V}}(\mathcal{P})$ is an Euler tour and we are done.
\end{proof}

We now consider the case when $e_M(V_2,A)$ is odd and at least three, and $e_M(V_1,A) = 1$.

\begin{lemma}\label{1,3}
Let $\Delta,\Delta' \in \mathbb{N}$, $\ell \in \mathbb{N}_0$ and $m \in 2\mathbb{N}$ with $\Delta'/\Delta, m/\Delta',\ell/\Delta' \ll 1$.
Let $G$ be a $3$-connected graph with vertex partition $\mathcal{V} = \lbrace V_1,V_2, W:= A \cup B \rbrace$.
Let $e_G(A) \geq (\ell-1)\Delta + \Delta'$ and $\Delta(G[A]) \leq \Delta$.
Let $M_2$ be a matching in $G[V_2,A]$ of size $m-1$ where $3 \leq m-1 < e_G(V_2,A)$ and let $e_1 \in G[V_1,A]$ be an edge not incident to $M_2$.
Then $G$ contains a path system $\mathcal{P}$ such that $e(\mathcal{P}) \leq \ell+m+2$, $R_{\mathcal{V}}(\mathcal{P})$ is an Euler tour and $\bal_{AB}(\mathcal{P}) = \ell + m/2$.
\end{lemma}

\begin{proof}
We will find a path system $\mathcal{P}$ such that, for each $X \in \mathcal{V}$, $e_{\mathcal{P}}(X,\overline{X})$ is even and there exists $X' \in \mathcal{V} \setminus \lbrace X \rbrace$ such that $\mathcal{P}$ contains an $XX'$-path.
Then by Fact~\ref{eulertour}, $R_{\mathcal{V}}(\mathcal{P})$ is an Euler tour.
We will choose $\mathcal{P}$ such that $\mathcal{P}[V_1 \cup V_2,W]$ is obtained from $M_2 \cup \lbrace e_1 \rbrace$ by adding/removing at most one edge.%
\COMMENT{LATE CHANGE: The previous sentence was wrong.}
Since $G$ is $3$-connected, $G$ contains an edge $v_1v$ with $v_1 \in V_1$ and $v \in V_2 \cup A \cup B$ such that $vv_1$ and $e_1$ are vertex-disjoint.
We consider cases depending on the location of $v$.

\medskip
\noindent
\textbf{Case 1.}
$v \in A$.

\medskip
\noindent
If possible, let $e_2$ be the edge of $M_2$ incident to $v$; otherwise, let $e_2$ be an arbitrary edge of $M_2$.
Then we are done by applying Lemma~\ref{2,2} with $M_2 \cup \lbrace e_1,v_1v \rbrace \setminus \lbrace e_2 \rbrace$ playing the role of $M$.

\medskip
\noindent
\textbf{Case 2.}
$v \in V_2$.

\medskip
\noindent
If possible, choose $e_2 \in E(M_2)$ whose endpoint $v_2 \in V_2$ satisfies $v_2 = v$, otherwise let $e_2 \in E(M_2)$ be arbitrary.
Set $V_1' := V_1 \cup \lbrace v,v_2 \rbrace$ and $V_2' := V_2 \setminus \lbrace v,v_2 \rbrace$.
Observe that $e_{M_2 \cup \lbrace e_1 \rbrace}(A,V_i') \in 2\mathbb{N}$ for $i=1,2$.
Let $\mathcal{V}' := \lbrace V_1',V_2',W \rbrace$.
Apply Lemma~\ref{2,2} with $G \setminus \lbrace v_1 \rbrace, V_1',V_2',A,B,M_2 \cup \lbrace e_1 \rbrace$ playing the roles of $G,V_1,V_2,A,B,M$ to obtain a path system $\mathcal{P}'$ such that $\mathcal{P}' \subseteq G[A] \cup G[A,V_1' \cup V_2']$, $\mathcal{P}'[A,V'_1 \cup V'_2] = M_2 \cup \lbrace e_1 \rbrace$, $e(\mathcal{P}') = \ell+m$, $R_{\mathcal{V}'}(\mathcal{P}')$ is an Euler tour
and $\bal_{AB}(\mathcal{P}')=\ell+m/2$. Moreover, $\mathcal{P}'$ contains at least one $V'_iA$-path for each $i=1,2$.
Let%
   \COMMENT{DK: before we had "Since $\mathcal{P}'[A,V_1 \cup V_2] = M_2 \cup \lbrace e_1 \rbrace$ we have that, for each $i=1,2$, $\mathcal{P}'$ contains at least two distinct $AV_i'$-paths.
Therefore, for each $i=1,2$, $\mathcal{P}'$ contains at least one $AV_i'$-path $P_i$ with $V(P_i) \cap \lbrace v,v_1,v_2 \rbrace = \emptyset$." But I don't see why
$\mathcal{P}'$ contains at least two distinct $AV_i'$-paths.}
$P_i$ be such a path.

Let $\mathcal{P} := \mathcal{P}' \cup \lbrace vv_1 \rbrace$.
Then $e(\mathcal{P}) = \ell+m+1$ and $\bal_{AB}(\mathcal{P}) = \ell+m/2$.
Moreover, each of 
$e_{\mathcal{P}}(V_1,\overline{V_1})=e_{\mathcal{P}'}(V_1',\overline{V_1'})=2$, $e_{\mathcal{P}}(V_2,\overline{V_2}) = e_{\mathcal{P}'}(V_2',\overline{V_2'})+2$
and $e_{\mathcal{P}}(W,\overline{W})=e_{\mathcal{P}'}(W,\overline{W})$ is even.
Now $P_2$ is a $V_2A$-path in $\mathcal{P}$. Similarly, if $P_1$ avoids $e_2$, then $P_1$ is a $V_1A$-path in $\mathcal{P}$.
If $P_1$ contains $e_2$ and $v_2=v$, then $v_1vP_1$ is a $V_1A$-path in $\mathcal{P}$. If $v_2\neq v$ then $v_1v$ is a $V_1V_2$-path in $\mathcal{P}$.%
\COMMENT{Note that $v$ lies in a path of $\mathcal{P}'$ iff $v=v_2$.}
Therefore, by Fact~\ref{eulertour}, $R_{\mathcal{V}}(\mathcal{P})$ is an Euler tour, as required.

\medskip
\noindent
\textbf{Case 3.}
$v \in B$.

\medskip
\noindent
Apply Lemma~\ref{goodmatching2} to $G[A]$.
Suppose first that the consequence~(i) of Lemma~\ref{goodmatching2} holds.
Let $M'$ be a matching of size $\ell+1$ in $G[A]$ and let $uw \in E(G[A])$ with $u \notin V(M')$.
Apply Proposition~\ref{casei} with $G \setminus B, V_1 \cup V_2,A,M_2 \cup \lbrace e_1 \rbrace,M',u,w$ playing the
roles of $G,X,Y,M,M',u,v$ to obtain a path system $\mathcal{P}_0$  such that
$\mathcal{P}_0[V_1 \cup V_2,A] = M_2 \cup \lbrace e_1 \rbrace$; $e_{\mathcal{P}_0}(A) = \ell+1$; and $\mathcal{P}_0$ contains at least two $(V_1 \cup V_2)A$-paths.
But $\mathcal{P}_0$ contains at most one $V_1A$-path, and hence at least one $V_2A$-path $P$.
Now the consequence~(i) of Proposition~\ref{casei} implies that $e_{P}(V_2,A)=1$.
So we can choose $e \in E(\mathcal{P}_0[V_2,A]) \setminus E(P)$.
Let $\mathcal{P} := \mathcal{P}_0 \cup \lbrace v_1v \rbrace \setminus \lbrace e \rbrace$.
Then $e_{\mathcal{P}}(X,\overline{X})$ is even for all $X \in \lbrace V_1,V_2,W \rbrace$ and $\mathcal{P}$ contains a $V_1B$-path and a $V_2A$-path.
Moreover, $\bal_{AB}(\mathcal{P})= e_{\mathcal{P}_0}(A) + e_{\mathcal{P}_0}(A,V_1 \cup V_2)/2 - 1 = \ell+m/2$, as required.

Suppose instead that the consequence~(ii) of Lemma~\ref{goodmatching2} holds.
Then $G[A]$ contains $\ell$ distinct vertices $x_1, \ldots, x_\ell$ such that $d_A(x_i) \geq \Delta'$ for all $1 \leq i \leq \ell$.
Choose $e \in E(G[V_2,A]) \setminus E(M_2)$.
If $\ell=0$ then $\mathcal{P} := M_2 \cup \lbrace e_1,v_1v,e \rbrace$ is as required.
Suppose now that $\ell=1$.
Let $w_1,y_1 \in N_A(x_1) \setminus V(M_2 \cup \lbrace e_1 \rbrace)$ be distinct.
Suppose that $x_1 \notin V(e_1)$.
If possible, choose $e_2$ to be the edge of $M_2$ that contains $x_1$;
otherwise, let $e_2$ be an arbitrary edge of $M_2$.
In this case we let
$\mathcal{P} := M_2 \cup \lbrace e_1,v_1v,w_1x_1y_1 \rbrace \setminus \lbrace e_2 \rbrace$.
Suppose now that $x_1 \in V(e_1)$.
In this case we let
$\mathcal{P} := M_2 \cup \lbrace e_1,v_1v,e \rbrace \cup \lbrace x_1y_1 \rbrace$. 
In all cases, we have that $R_{\mathcal{V}}(\mathcal{P})$ is an Euler tour, $e(\mathcal{P}) \leq \ell+m+2$ and $\bal_{AB}(\mathcal{P})=m/2+1$, as required.

Suppose finally that $\ell \geq 2$.
Then we can choose $e_2 \in M_2$ so that $\lbrace x_1, \ldots, x_\ell \rbrace \not\subseteq V(M_2 \cup \lbrace e_1 \rbrace \setminus \lbrace e_2 \rbrace)$.
Apply Proposition~\ref{2paths} with $G \setminus B, V_1 \cup V_2, A, M_2 \cup \lbrace e_1 \rbrace \setminus \lbrace e_2 \rbrace, x_i,1$ playing the roles of $G,X,Y,M,x_i,r$ to obtain a path system $\mathcal{P}_0$ in $G[A] \cup G[A,V_1 \cup V_2]$ such that $e_{\mathcal{P}_0}(A) = \ell + 1$, $\mathcal{P}_0[A,V_1 \cup V_2]=M_2 \cup \lbrace e_1 \rbrace \setminus \lbrace e_2 \rbrace$, and every edge in $M_2 \cup \lbrace e_1 \rbrace \setminus \lbrace e_2 \rbrace$ lies in a distinct $AV_i$-path in $\mathcal{P}_0$ for some $i \in \lbrace 1,2 \rbrace$. 
Let $\mathcal{P} := \mathcal{P}_0 \cup \lbrace v_1v \rbrace$.
Then $e(\mathcal{P}) = \ell+m+1$ and
$$
\bal_{AB}(\mathcal{P})=e_{\mathcal{P}_0}(A)+e_{\mathcal{P}_0}(A,V_1 \cup V_2)/2 - 1/2 = \ell+1+(m-1)/2 -1/2 = \ell+m/2.
$$
Note finally that $R_{\mathcal{V}}(\mathcal{P})$ is an Euler tour by Fact~\ref{eulertour}.
\end{proof}

We are now ready to prove a more general version of Lemmas~\ref{2,2}--\ref{1,3} in which $G[A,V_1 \cup V_2]$ contains an arbitrary even matching of size at least four.

\begin{lemma}\label{0,4}
Let $\Delta,\Delta' \in \mathbb{N}$, $\ell \in \mathbb{N}_0$ and $m \in 2\mathbb{N}$ with  $\Delta'/\Delta, m/\Delta', \ell/\Delta' \ll 1$ and $m \geq 4$.
Let $\Delta'/\Delta < \eps < 1/3$.
Let $G$ be a $3$-connected graph with vertex partition $\mathcal{V} = \lbrace V_1,V_2, W:= A \cup B \rbrace$.
Suppose that $\Delta(G[A]),\Delta(G[A,V_1 \cup V_2]) \leq \Delta$ and $\char_{\Delta,\eps}(G) = (\ell,m)$.
Then $G$ contains a path system $\mathcal{P}$ such that $e(\mathcal{P}) \leq \ell+m+4$, $R_{\mathcal{V}}(\mathcal{P})$ is an Euler tour and $\bal_{AB}(\mathcal{P}) = \ell + m/2$.
\end{lemma}

\begin{proof}
Write $U := V_1 \cup V_2$.
Proposition~\ref{charedges} implies that
\begin{equation}\label{lmedges}
e_G(A) \geq (\ell-1)\Delta+\Delta'\ \ \mbox{ and }\ \ e_G(A,U) \geq (m-1)\Delta + \Delta'.
\end{equation}
Recall also that $m \leq \lceil e_G(A,U)/\Delta \rceil$ and $m$ is even.
Choose non-negative integers $b_1,b_2$ such that $b_i \leq \lceil e_G(A,V_i)/\Delta \rceil$ for $i=1,2$ and $b_1+b_2=m$.
Apply Lemma~\ref{spreadmatching} with $G[A,U],A,V_1,V_2$ playing the roles of $G,U,V,W$ to obtain a matching $M$ in $G[A,U]$ such that $e_M(A,V_i) = b_i$ for $i=1,2$.
Without loss of generality we assume that $b_1 \leq b_2$.
Suppose first that $b_1,b_2$ are both even and positive.
Then we are done by applying Lemma~\ref{2,2}.
If $b_1,b_2$ are both odd and at least three, then we are done by applying Lemma~\ref{3,3}.%
\COMMENT{$\lceil e(A,U)/\Delta \rceil < e(A,U)$ since $e(A,U) \geq 3\Delta+\Delta'$.}
Suppose that $b_1=1$.
Then $\lceil e_G(A,V_2)/\Delta \rceil \geq b_2 = m-1$ so $m-1 < e_G(A,V_2)$.
Therefore we can apply Lemma~\ref{1,3} with $M$ playing the role of $M_2 \cup \lbrace e_1 \rbrace$.
So we can assume that $b_1=0$, and hence that $M \subseteq G[A,V_2]$.
Suppose that $e_G(A,V_1) > 0$.
Then there is an edge $e \in E(G[A,V_1])$ and $m-1$ edges in $M$ which are not incident with $e$.
We are similarly done by applying Lemma~\ref{1,3}.
The only remaining case is when $e_G(A,V_1)=0$.
Now
(\ref{lmedges}) implies that 
\begin{equation}\label{nowhave}
e_G(A,V_2) \geq (m-1)\Delta + \Delta'.
\end{equation}

Since $G$ is $3$-connected, $G[V_1,\overline{V_1}]$ contains a matching of size three.
So $G[V_1,V_2 \cup B]$ contains a matching of size three.
Then at least one of $G[V_1,V_2]$, $G[V_1,B]$ contains a matching of size two.

\medskip
\noindent
\textbf{Case 1.}
\emph{$G[V_1,V_2]$ contains a matching $M^*$ of size two.}

\medskip
\noindent
Choose two distinct edges $e_2,e_2' \in E(M)$ such that $|V(M^*)\cap \{v_2,v'_2\}|$ is as large
as possible, where $v_2,v'_2$ are the endvertices of $e_2,e'_2$ in $V_2$.%
   \COMMENT{DK: reformulated that sentence}
Set $V_1' := V_1 \cup \lbrace v_2,v_2' \rbrace$ and $V_2' := V_2 \setminus \lbrace v_2,v_2' \rbrace$.
Observe that $e_M(A,V_i') \in 2\mathbb{N}$ for $i=1,2$ since $m \geq 4$.
Let $\mathcal{V}' := \lbrace V_1',V_2',W \rbrace$.
Apply Lemma~\ref{2,2} with $G,V_1',V_2',A,B,M$ playing%
   \COMMENT{DK had $G \setminus E(M^*)$ instead of $G$, don't see why we have to delete $M^*$}
the roles of $G,V_1,V_2,A,B,M$ to obtain a path system $\mathcal{P}'$ such that $\mathcal{P}' \subseteq G[A] \cup G[A,V_1' \cup V_2']$, $\mathcal{P}'[A,V_1' \cup V_2'] = M$, $e(\mathcal{P}')=\ell+m$, $R_{\mathcal{V}'}(\mathcal{P}')$ is an Euler tour and $\bal_{AB}(\mathcal{P}')=\ell+m/2$.%
   \COMMENT{DK: changed the rest of Case 1, had "Since $\mathcal{P}'[A,V_1' \cup V_2'] = M$, we have that,
for each $i=1,2$, $\mathcal{P}'$ contains at least two $AV_i'$-paths $P_i, P_i'$." before. I don't see why this is true.}
Moreover, $\mathcal{P}'$ contains at least one $V'_iA$-path for each $i=1,2$.
Let $P_i$ be such a path. Then $P_1$ contains either $e_2$ or $e'_2$. Without loss of generality we may assume that $P_1$ contains $e_2$.

Let $\mathcal{P} := \mathcal{P}' \cup M^*$.
Then $e(\mathcal{P})=\ell+m+2$ and $\bal_{AB}(\mathcal{P})=\ell+m/2$.
Moreover, each of $e_{\mathcal{P}}(V_1,\overline{V_1})=e_{\mathcal{P}'}(V_1',\overline{V_1'})=2$, $e_{\mathcal{P}}(V_2,\overline{V_2})=e_{\mathcal{P}'}(V_2',\overline{V_2'})+4$
and  $e_{\mathcal{P}}(W,\overline{W})=e_{\mathcal{P}'}(W,\overline{W})$ is even.
Now $P_2$ is an $V_2A$-path in $\mathcal{P}$. If $M^*$ contains an edge $e$ which avoids both $v_2, v'_2$ (and thus is vertex-disjoint from all
edges in $M$), then $e$ is a $V_1V_2$-path in $\mathcal{P}$. If there is no such edge $e$, then $M^*$ contains an edge $e'$ whose endvertex in $V_2$ is $v_2$.
Then $e'\cup P_1$ is a $V_1A$-path in $\mathcal{P}$.
Therefore, by Fact~\ref{eulertour}, $R_{\mathcal{V}}(\mathcal{P})$ is an Euler tour, as required.

\medskip
\noindent
\textbf{Case 2.}
\emph{$G[V_1,B]$ contains a matching $M^*$ of size two.}

\medskip
\noindent
Apply Lemma~\ref{goodmatching2} to $G[A]$.
Suppose first that the consequence~(i) of Lemma~\ref{goodmatching2} holds.
Then $G[A]$ contains a matching $M'$ of size $\ell+1$ and an edge $uv$ with $u \notin V(M')$.
Apply Proposition~\ref{casei} with $G \setminus B, V_1 \cup V_2,A,M,M',u,v$ playing the roles of $G,X,Y,M,M',u,v$ to
obtain a path system $\mathcal{P}_0$  such that $\mathcal{P}_0[V_1 \cup V_2,A] = M$; $\mathcal{P}_0 \subseteq M \cup M' \cup \lbrace uv \rbrace$; $e_{\mathcal{P}_0}(A)=\ell+1$;
and $\mathcal{P}_0$ contains at least two $V_2A$-paths.
Let $\mathcal{P} := \mathcal{P}_0 \cup M^*$.
Then $\mathcal{P}$ contains at least two $V_2A$-paths and two $V_1B$-paths (namely the edges of $M^*$), so $R_{\mathcal{V}}(\mathcal{P})$ is an Euler tour.
Moreover
$
\bal_{AB}(\mathcal{P})= \ell+m/2
$
and $e(\mathcal{P})=\ell+m+3$, as required.

Suppose now that the consequence~(ii) of Lemma~\ref{goodmatching2} holds in $G[A]$.
Assume first that $\ell \geq 2$.
Let $x_1, \ldots, x_\ell$ be $\ell$ distinct vertices in $A$ such that $d_A(x_i) \geq \Delta'$ for $1 \leq i \leq \ell$.
Since $m \geq 4$, we can choose distinct $e_1,e_2 \in M$ such that $|\lbrace x_1, \ldots, x_\ell \rbrace \setminus V(M \setminus \lbrace e_1,e_2 \rbrace)| \geq 2$.
Then Proposition~\ref{2paths} applied with $G \setminus B,V_1 \cup V_2,A,M \setminus \lbrace e_1,e_2 \rbrace,x_i,2$ playing the roles
of $G,X,Y,M,x_i,r$ implies that there is a path system $\mathcal{P}' \subseteq G[A] \cup G[A,V_1 \cup V_2]$ such
that $e_{\mathcal{P}'}(A) = \ell+2$, $\mathcal{P}'[A,V_1 \cup V_2]=M \setminus \lbrace e_1,e_2 \rbrace$,
and such that every edge of $M \setminus \lbrace e_1, e_2 \rbrace$ lies in a distinct $AV_2$-path.
Let $\mathcal{P} := \mathcal{P}' \cup M^*$.
Then $R_{\mathcal{V}}(\mathcal{P})$ is an Euler tour, $e(\mathcal{P})=\ell+m+2$, and
$$
\bal_{AB}(\mathcal{P})=e_{\mathcal{P}'}(A) +e_{\mathcal{P}'}(A,V_1 \cup V_2)/2 - 1 = \ell+2+(m-2)/2 - 1 = \ell+m/2.
$$

Finally we consider the case when $\ell \leq 1$.
Lemma~\ref{goodmatching2} applied to $G[A,V_1 \cup V_2]$ and (\ref{nowhave}) imply that $G[A,V_1 \cup V_2]$ contains a matching $M'$
of size $m$ together with a matching $M^+$ of size two which is edge-disjoint from $M'$, such that both edges in $M^+$ contain a
vertex outside of $V(M')$. Since $e_G(A,V_1)=0$ by our assumption, we have $M' \cup M^+\subseteq G[A,V_2]$.%
   \COMMENT{Deryk added new sentence}
Suppose first that $\ell=0$.
In this case we let $\mathcal{P} := M' \cup M^+ \cup M^*$.
It is clear that $R_{\mathcal{V}}(\mathcal{P})$ is an Euler tour, $e(\mathcal{P})=m+4$ and $\bal_{AB}(\mathcal{P})=m/2$, as required.
The final case is when $\ell=1$. Choose $e\in M^+$ and $e'\in M'$ such that $|V(e)\cap \{x_1\}|+|V(e')\cap \{x_1\}|$ is maximal.%
    \COMMENT{DK: reformulated this sentence}
So $\mathcal{P}' := M' \cup M^+ \setminus \lbrace e,e' \rbrace$ is a matching of size $m-1$ together with an extra edge, and $x_1 \notin V(\mathcal{P}')$.
In particular, $\mathcal{P}'$ contains a $V_2A$-path $P_2$.
Since $m/\Delta' \ll 1$, we can choose distinct vertices $w_1,y_1$ in $N_A(x_1) \setminus V(\mathcal{P}')$.
Let $\mathcal{P} := \mathcal{P}' \cup M^* \cup \lbrace w_1x_1y_1 \rbrace$.
Then $P_2$ is a $V_2A$-path in $\mathcal{P}$ and each edge of $M^*$ is a $V_1B$-path in $\mathcal{P}$.
So Fact~\ref{eulertour} implies that $\mathcal{P}$ is an Euler tour.
Moreover,
$
\bal_{AB}(\mathcal{P})=m/2+1,
$
and $e(\mathcal{P})=m+4$,
as required.
\end{proof}

The proof of Lemma~\ref{aim} in the `dense' case is now just a short step away.

\medskip
\noindent
\emph{Proof of Lemma~\emph{\ref{aim}} in the case when $|A|-|B| \geq 2$ and $m \geq 4$.}
Let $\Delta := D/2$.
Observe that $d_A(a) \leq d_B(a)$ for all $a \in A$ implies that $\Delta(G[A]) \leq \Delta$.
Proposition~\ref{matchingsizes} implies that $\ell+m/2 \geq |A|-|B|$.
Choose non-negative integers $\ell' \leq \ell$ and $m' \leq m$ such that $m'$ is even, $\ell'+m'/2 = |A|-|B|$ and $m' \geq 4$.
This is possible since $|A|-|B| \geq 2$.
Let $\Delta' := \nu n$.
Proposition~\ref{ell} implies that $\ell',m' \leq 12\rho n$.
Then $\Delta'/\Delta \ll 1$, $m'/\Delta' \ll 1$, $\ell'/\Delta' \ll 1$, $\Delta'/\Delta < \eps$.
Apply Lemma~\ref{0,4} with $\ell',m'$ playing the roles of $\ell,m$ to obtain a path system $\mathcal{P}$ such that
$e(\mathcal{P}) \leq \ell'+m'+4 \leq \ell+m+4$,
$R_{\mathcal{V}}(\mathcal{P})$ is an Euler tour, and $\bal(\mathcal{P})=\ell'+m'/2 = |A|-|B|$.
So (P1)--(P3) hold.
\hfill$\square$

\medskip


\subsection{The proof of Lemma~\ref{aim} in the case when $|A|-|B| \geq 2$ and $m \leq 2$}\label{sparse}

We now deal with the sparse case, i.e.~when the largest even matching we can guarantee between $A$ and $V_1 \cup V_2$ has size at most two.
For this, we need to introduce some notation which
will be used in all of the remaining cases.

\subsubsection{More notation and tools}

\COMMENT{ALLAN NEW}
In the previous case when $m \ge 4 $, $G[A,V_1 \cup V_2]$ has a large matching which we used to suitably connect components. 
In the case $m \le 2$ we cannot rely on this.
So we use a `basic connector'. 
Roughly speaking, a basic connector $\mathcal{P}$ is a path system with few edges such that $R_{\mathcal{V}}(\mathcal{P})$ is an Euler tour.
So $\mathcal{P}$ satisfies (P1) and (P3), but not necessarily (P2) (i.e.~we might have $\bal_{AB}(\mathcal{P}) \ne |A| - |B|$).
We find a basic connector by Proposition~\ref{BC} and then adjust it to satisfy~(P2).
Basic connectors will also be useful in the final two subsections, i.e.~Sections~\ref{+1} and~\ref{equal}, which concern the case when $|A| - |B| \le 1$.

Given a path system $\mathcal{P}$, recall the definition of $F_{\mathcal{P}}(A)$ in~(\ref{F}).
We say that $\mathcal{P}$ is a \emph{basic connector}%
\COMMENT{when $G[V_1 \cup _2, A \cup B]$ is very sparse, we need to start with such an object (guaranteed by $3$-connectivity) and adjust it as appropriate. When this graph is dense, we start with the large matching it contains and extend this into an appropriate path system.}
(for $\mathcal{V} = \lbrace V_1,V_2,W := A \cup B \rbrace$)
if
\begin{itemize}
\item[(BC1)] $R_{\mathcal{V}}(\mathcal{P})$ is an Euler tour;
\item[(BC2)] $e(\mathcal{P}) \leq 4$ and $|\bal_{AB}(\mathcal{P})| \leq 2$;
\item[(BC3)] $e_{\mathcal{P}}(A \cup B)=0$;
\item[(BC4)] if $F_{\mathcal{P}}(A) = (a_1,a_2)$ then $\bal_{AB}(\mathcal{P}) \in \lbrace a_1 + 2a_2-2, a_1+2a_2-1 \rbrace$ and $a_2 \leq 1$.%
\end{itemize}
It can be shown that (BC1)--(BC3) imply (BC4) (cf.~the proof of Proposition~\ref{BC}).
Observe (BC3) implies that if $\mathcal{P}$ is a basic connector, then \begin{equation}\label{BCeq}
2\bal_{AB}(\mathcal{P}) = e_{\mathcal{P}}(A,V_1 \cup V_2) - e_{\mathcal{P}}(B,V_1 \cup V_2) = a_1+2a_2 - e_{\mathcal{P}}(B,V_1 \cup V_2).
\end{equation}
Roughly speaking, the existence of a basic connector $\mathcal{P}$ follows from $3$-connectivity.
We would like to modify/extend $\mathcal{P}$ into a path system $\mathcal{P}'$ which balances the sizes of $A,B$, i.e.~for
which $\bal_{AB}(\mathcal{P}')= |A|-|B|$.%
    \COMMENT{DK: changed this sentence}
The following notion will be very useful for this. 
Given a graph $G$, disjoint $A_1,A_2 \subseteq V(G)$ and $t \in \mathbb{N}_0$,
we say that%
    \COMMENT{Deryk introduced this def instead of ${\rm acc}(G; c_1,c_2)$ since there was a mistake in the proof of Prop~\ref{build2paths}
(and in the proof of Lemma~\ref{accommodation})}
\begin{equation*}
{\rm acc}(G; A_1,A_2) \geq t
\end{equation*}
if $G$ contains a path system $\mathcal{P}$ such that
\begin{itemize}
\item[(A1)] $e(\mathcal{P}) = t$;
\item[(A2)] $d_{\mathcal{P}}(x_2)=0$ for each $x_2 \in A_2$;
\item[(A3)] $d_{\mathcal{P}}(x_1) \leq 1$ for each $x_1 \in A_1$, and no path of $\mathcal{P}$ has both endpoints in $A_1$.
\end{itemize}
We say that such a $\mathcal{P}$ \emph{accommodates $A_1,A_2$}, where `acc' is chosen for `accomodating'.%
\COMMENT{NEW sentence}

In%
   \COMMENT{DK: changed this para}
a typical application of this notion, we have already constructed a path system $\mathcal{P}_0$.
We let $A_1$ be the set of all those vertices in $A$ which have degree one in $\mathcal{P}_0$ and $A_2$
be the set of all those vertices in $A$ which have degree two in $\mathcal{P}_0$.
Then, if ${\rm acc}(G[A];A_1,A_2) \geq t$, we can find a path system $\mathcal{P}$ in $G[A]$ with $t$ edges such
that $\mathcal{P}_0 \cup \mathcal{P}$ is also a path system.

We now collect some tools which will be used to prove Lemma~\ref{aim} in the case when $|A|-|B| \geq 2$ and $m \leq 2$.
The next proposition uses Lemma~\ref{cliquetour} to show that $G$ contains a basic connector.

\begin{proposition}\label{BC}
Let $G$ be a $3$-connected graph with vertex partition $\mathcal{V} = \lbrace V_1,V_2,W := A \cup B \rbrace$.
Then $G$ contains a basic connector $\mathcal{P}$.
\end{proposition}

\begin{proof}
Apply Lemma~\ref{cliquetour} to $G$ and $\mathcal{V}$ to obtain a path system $\mathcal{P}$ satisfying the conditions (i)--(iii).
We claim that $\mathcal{P}$ is a basic connector.
Write $F_{\mathcal{P}}(A)=(a_1,a_2)$ and $F_{\mathcal{P}}(B)=(b_1,b_2)$.
In particular, (iii) implies that
\begin{equation}\label{degsum}
a_1+b_1+2(a_2+b_2) \in \lbrace 2, 4 \rbrace
\end{equation}
and $a_2+b_2 \leq 1 $.
Note that (BC1) and (BC3) are immediate from (ii) and (i) respectively.
Moreover, (i) implies $e_{\mathcal{P}}(A \cup B)=0$.
So $e_{\mathcal{P}}(A,V_1 \cup V_2) = a_1+2a_2$ and $e_{\mathcal{P}}(B,V_1 \cup V_2) = b_1 +2b_2$.
So (\ref{degsum}) implies that 
$$
2\bal_{AB}(\mathcal{P}) = a_1+2a_2 - b_1 - 2b_2 \in \lbrace 2a_1+4a_2 - 4, 2a_1+4a_2-2 \rbrace
$$
and $|2\bal_{AB}(\mathcal{P})| \leq 4$,
so (BC2) and (BC4) hold.
\end{proof}

By Proposition~\ref{BC}, we can find a basic connector $\mathcal{P}_0$ in $G$, which may not satisfy (P2).
Our aim now is to find a suitable path system $\mathcal{P}_A$ in $G[A]$ so that $\mathcal{P}_0 \cup \mathcal{P}_A$ satisfies (P1)--(P3).
Let $A_i$ be the collection of all those vertices of $A$ with degree $i$ in $\mathcal{P}_0$.
The next result shows%
\COMMENT{I merged the two propositions that used to be here since the former was only ever used to prove the latter.}
that it suffices to show that ${\rm acc}(G[A];A_1,A_2)\ge |A|-|B|-\bal_{AB}(\mathcal{P}_0)$.%
     \COMMENT{Deryk changed this sentence and rewrote prop, replacing ${\rm acc}(G[A];a_1,a_2)$ by ${\rm acc}(G[A];A_1,A_2)$.
We need this strengthening in the proof of Prop~\ref{2paths2}}

\begin{proposition}\label{addpaths}
Let $G$ be a graph with vertex partition $\mathcal{V} = \lbrace V_1,V_2,W := A \cup B \rbrace$.
Let $\mathcal{P}_0$ be a basic connector in $G$ and for $i=1,2$ let $A_i$ be the collection of all those vertices of $A$
with degree $i$ in $\mathcal{P}_0$.
Then, for any integer $0 \leq t \leq {\rm acc}(G[A];A_1,A_2)$, we have that $G$ contains a path system $\mathcal{P}$ such that $R_{\mathcal{V}}(\mathcal{P})$ is an Euler tour, $\bal_{AB}(\mathcal{P}) = \bal_{AB}(\mathcal{P}_0)+t$ and $e(\mathcal{P}) \leq t+4$.
\end{proposition}

\begin{proof}
Let $\mathcal{P}_A$ be a path system in $G[A]$ which accommodates $A_1,A_2$ such that $e(\mathcal{P}_A) = t$.
Let $\mathcal{P} := \mathcal{P}_0 \cup \mathcal{P}_A$.
Properties (A2) and (A3) imply that $\mathcal{P}$ is a path system.
It is straightforward to check that (BC1) implies that $R_{\mathcal{V}}(\mathcal{P})$ is an Euler tour.
Moreover, $\bal_{AB}(\mathcal{P})=\bal_{AB}(\mathcal{P}_0)+e(\mathcal{P}_A)$, as required.
Finally, (BC2) gives the required bound on $e(\mathcal{P})$.
\end{proof}

\subsubsection{Building a basic connector from a matching}

The next lemma shows that in the case when $G[A,V_1 \cup V_2]$ contains a matching of size at least three,
we can obtain a basic connector with additional useful properties.

\begin{lemma}\label{3matching}
Let $G$ be a $3$-connected graph with vertex partition $\mathcal{V} = \lbrace V_1,V_2,W := A \cup B \rbrace$.
Suppose that $G[A,V_1 \cup V_2]$ contains a matching $M$ of size three.
Then one of the following holds:
\begin{itemize}
\item[(i)] $G$ contains a basic connector $\mathcal{P}$ with $\bal_{AB}(\mathcal{P}) \geq 1$, and if $F_{\mathcal{P}}(A) = (a_1,a_2)$,
then $a_1 \geq 2$;
\item[(ii)] $e_G(A,V_i)=0$ for some $i\in \{1,2\}$, and for each $a \in A$, $G$ contains matchings $M_{a,A},M_{a,B}$ in $G[A \setminus \lbrace a \rbrace,V_j], G[B,V_i]$ respectively, where $j \in \lbrace 1,2 \rbrace \setminus \lbrace i \rbrace$, each of which has size two.
In particular, $\mathcal{P}_a := M_{a,A} \cup M_{a,B}$ is a basic connector with $\bal_{AB}(\mathcal{P}_a)=0$, $a \notin V(\mathcal{P}_a)$ and $F_{\mathcal{P}}(A)=(2,0)$.%
\COMMENT{This statement has been strengthened, and the proof has changed accordingly (last few lines). DK: changed "moreover" to "in particular"}
\end{itemize}
\end{lemma}

\begin{proof}
Without loss of generality we may assume that $e_M(A,V_2) \geq e_M(A,V_1)$.
Suppose first that $e_G(A,V_1)>0$.
We claim that $G[A,V_1 \cup V_2]$ contains a matching $M'$ of size three such that $e_{M'}(A,V_1)=1$ and $e_{M'}(A,V_2)= 2$.
To see this, we may assume that we cannot set $M' := M$, so $M \subseteq G[A,V_2]$.
Let $e_1 \in E(G[A,V_1])$.
Then $V(e_1) \cap V(M) \subseteq A$.
If possible, let $e'$ be the edge of $M$ incident to $e_1$, otherwise let $e' \in E(M)$ be arbitrary.
Let $M' := M \cup \lbrace e_1 \rbrace \setminus \lbrace e' \rbrace$, proving the claim.

Since $G$ is $3$-connected, there exists $e \in E(G[V_1,\overline{V_1}])$ that is not incident with the unique edge $e_1 \in M'[A,V_1]$.
Let $x$ be the endpoint of $e$ that does not lie in $V_1$.
If $x \in V_2$ then we can choose $e_2 \in M'[A,V_2]$ which is not incident with $e$ and then $\mathcal{P} := \lbrace e,e_1,e_2 \rbrace$ is a path system with $\bal_{AB}(\mathcal{P})=1$ and $F_{\mathcal{P}}(A)=(2,0)$.
It is easy to check that $\mathcal{P}$ is a basic connector, so (i) holds.
If $x \in A \cup B$ then similarly $\mathcal{P} := M' \cup \lbrace e \rbrace$ satisfies (i).

Suppose now that $e_G(A,V_1)=0$.
Thus $e_M(A,V_2)=3$.
Since $G$ is $3$-connected, there is a matching $M'$ of size three in $G[V_1,\overline{V_1}]$.
Let $E(M') = \lbrace e_1,e_2,e_3 \rbrace$ and let $x_1,x_2,x_3$ respectively be the endpoints of $e_1,e_2,e_3$ which do not lie in $V_1$.
Note that $\lbrace x_1,x_2,x_3 \rbrace \subseteq  B \cup V_2$.
Suppose first that $|V(M') \cap B| \leq 1$.
Without loss of generality we assume that $\lbrace x_1,x_2 \rbrace \subseteq V_2$.
Let $e,e' \in E(M)$ be such that $\lbrace x_1,x_2 \rbrace \not\subseteq V(\lbrace e,e' \rbrace)$.
Then $\mathcal{P} := \lbrace e,e',e_1,e_2 \rbrace$ is such that $\bal_{AB}(\mathcal{P})=1$ and $F_{\mathcal{P}}(A)=(2,0)$.
Moreover, $\mathcal{P}$ is a basic connector, so (i) holds.
So without loss of generality we may assume that $|V(M') \cap B| \geq 2$ and $\lbrace x_1,x_2 \rbrace \subseteq B$.
Given an arbitrary $a \in A$, choose $e,e' \in E(M)$ such that $a \notin V(\lbrace e,e' \rbrace)$.
Let $M_{a,A} := \lbrace e,e' \rbrace$ and $M_{a,B} := \lbrace e_1,e_2 \rbrace$.
So (ii) holds.%
\COMMENT{Let $\mathcal{P}_a := M_1 \cup M_2$.
Then $\bal_{AB}(\mathcal{P}_a)=0$ and $a \notin V(\mathcal{P}_a)$ and $F_{\mathcal{P}_a}=(2,0)$.
Moreover, $\mathcal{P}$ is a basic connector, so (ii) holds.}
\end{proof}

We now show how this result implies that, whenever $G[A,V_1 \cup V_2]$ contains a matching of size two, we are again able to find a basic connector with additional useful properties (though not as useful as those in Lemma~\ref{3matching}).

\begin{lemma}\label{2matchingcor}
Let $G$ be a $3$-connected graph with vertex partition $\mathcal{V} = \lbrace V_1,V_2,W := A \cup B \rbrace$.
Let $M$ be a matching in $G[A, V_1 \cup V_2]$ of size two.
Then $G$ contains a basic connector $\mathcal{P}$ with $\bal_{AB}(\mathcal{P}) \geq 0$, and if $F_{\mathcal{P}}(A)=(a_1,a_2)$, then $a_1 \geq 1$.
\end{lemma}

\begin{proof}
Write $U := V_1 \cup V_2$.
Since $G$ is $3$-connected, $G[A \cup B,U]$ contains a matching $M'$ of size three.
We claim that $M \cup M'$ contains a matching $M^*$ of size three such that at least two of the edges in $M^*$ lie in $G[A,U]$.
To see this, assume that $e_{M'}(A,U) \leq 1$ (or we could take $M^* := M'$).
Assume further that there is no edge $e \in E(M')$ without an endpoint in $V(M)$ (or we could take $M^* := M \cup \lbrace e \rbrace$).
Then, if we write $M := \lbrace au,a'u' \rbrace$ where $a,a' \in A$ and $u,u' \in U$, we have that $M'$ consists of distinct edges $e_u,e_{u'},e$ incident with $u,u'$ and $\lbrace a,a' \rbrace$ respectively.
Suppose that $a \in V(e)$.
Then $e \in E(G[A,U])$ and so $e_u,e_{u'} \in E(G[B,U])$.
Moreover, neither $e$ nor $e_u$ is incident with $a'u'$.
We can set $M^* := \lbrace a'u', e,e_u \rbrace$.
If instead $a' \in V(e)$, then we can set $M^* := \lbrace au,e,e_{u'} \rbrace$.
This proves the claim.

If $M^* \subseteq G[A,U]$, we are done by Lemma~\ref{3matching}.
Otherwise, let $bu$ be the unique edge in $M^*[B,U]$ with $u \in U$ and $b \in B$.
Let $A' := A \cup \lbrace b \rbrace$ and $B' := B \setminus \lbrace b \rbrace$.
Apply Lemma~\ref{3matching} with $G,M^*,A',B'$ playing the roles of $G,M,A,B$.
Suppose first that (i) holds.
Then $G$ contains a basic connector $\mathcal{P}$ with $\bal_{A'B'}(\mathcal{P}) \geq 1$.
But $\bal_{AB}(\mathcal{P}) = \bal_{A'B'}(\mathcal{P})- d_{\mathcal{P}}(b)$ if $b \in V(\mathcal{P})$ and $\bal_{AB}(\mathcal{P}) = \bal_{A'B'}(\mathcal{P})$ otherwise.
If $d_{\mathcal{P}}(b) = 1$ then $\bal_{AB}(\mathcal{P}) \geq 0$, as required.
Suppose that $d_{\mathcal{P}}(b) = 2$.
Write $F_{\mathcal{P}}(A') := (a_1',a_2')$.
Thus $a_2' = 1$ by (BC4).
Moreover,
the consequence~(i) of Lemma~\ref{3matching} implies that $a_1' \geq 2$.
Now $a_1'+2a_2' \leq \bal_{A'B'}(\mathcal{P})+2 \leq 4$ by (BC2) and (BC4), so $(a_1',a_2')=(2,1)$ and $\bal_{A'B'}(\mathcal{P})=2$.
Then $\bal_{AB}(\mathcal{P}) \geq 0$, as required.
Let $F_{\mathcal{P}}(A) =: (a_1,a_2)$.
As above, $(a_1,a_2) \in \lbrace (a_1'-1,a_2'),(a_1',a_2'-1), (a_1',a_2') \rbrace$.
So $a_1 \geq a_1'-1 \geq 1$ by the consequence~(i) of Lemma~\ref{3matching}.
Suppose instead that the consequence~(ii) of Lemma~\ref{3matching} holds.
The `in particular' part implies that $G$ contains a basic connector $\mathcal{P}_b$
with $\bal_{A'B'}(\mathcal{P}_b) = 0$, $F_{\mathcal{P}_b}(A) = (2,0)$ and $b \notin V(\mathcal{P}_b)$.
Then $\bal_{AB}(\mathcal{P}_b)=\bal_{A'B'}(\mathcal{P}_b)$,
and $F_{\mathcal{P}_b}(A)=F_{\mathcal{P}_b}(A')$
as required. 
\end{proof}

\subsubsection{Accommodating path systems}

The following proposition gives a lower bound for ${\rm acc}(G;A_1,A_2)$ whenever $G$ contains several vertices of degree much larger
than $|A_1|+|A_2|$ (i.e.~when the consequence~(ii) of Lemma~\ref{goodmatching2} holds in $G$).%
    \COMMENT{Deryk replaced ${\rm acc}(G;a_1,a_2)$ by ${\rm acc}(G;A_1,A_2)$ in the next prop and the rest of the section (the next prop was wrong before).}

\begin{proposition}\label{build2paths}
Let $\Delta' \in \mathbb{N}$ and let $\ell, a_1,a_2 \in \mathbb{N}_0$ be such that $\Delta' \geq 3\ell+a_1+a_2$.
Let $G$ be a graph and let $X$ be a collection of $\ell$ vertices in $G$ such that $d_G(x) \geq \Delta'$ for all $x \in X$.
Then for all disjoint $A_1,A_2 \subseteq V(G)$ with $|A_i|=a_i$ for $i=1,2$, we have
$$
{\rm acc}(G;A_1,A_2) \geq 2\ell-|X \cap A_1|-2|X \cap A_2|.
$$
\end{proposition}

\begin{proof}
Write $X := \lbrace x_1, \ldots, x_\ell \rbrace$.
Since $\Delta' \geq 3\ell+a_1+a_2$ we can choose distinct vertices $w_1, \ldots, w_\ell, y_1, \ldots, y_\ell$ such that $\lbrace w_i, y_i \rbrace \subseteq N(x_i) \setminus (A_1 \cup A_2 \cup X)$.
For each $1 \leq i \leq \ell$, define
\begin{equation}\label{Picases}
P_i := \begin{cases}
x_iy_i
&\mbox{if } x_i \in A_1;\\
\emptyset
&\mbox{if } x_i \in A_2;\\
w_ix_iy_i
&\mbox{otherwise}.
\end{cases}
\end{equation}
Then $\mathcal{P} := \bigcup_{1 \leq i \leq \ell}P_i$ is a path system which accommodates $A_1,A_2$.
Clearly
\begin{align}\label{accbound2}
{\rm acc}(G;A_1,A_2) &\geq e(\mathcal{P}) = 2\ell - |X \cap A_1| - 2|X \cap A_2|,
\end{align}
as required.
\end{proof}

The following proposition shows that, if $A$ contains a collection $X$ of vertices of high degree and $G$ contains a basic connector $\mathcal{P}_0$ which does not interact too much with $X$, then we can extend $\mathcal{P}_0$ such that it still induces an Euler tour but $\bal_{AB}(\mathcal{P}_0)$ has increased.

\begin{proposition}\label{2paths2}
Let $\Delta' \in \mathbb{N}$ and let $\ell,r \in \mathbb{N}_0$ be such that $\Delta' \geq 3\ell+4$.
Let $G$ be a graph with vertex partition $\mathcal{V} = \lbrace V_1,V_2,W := A \cup B \rbrace$ and let $\mathcal{P}_0$ be a basic connector in $G$.
For $i=1,2$, let $A_i$ be the collection of all those vertices in $A$ with degree $i$ in $\mathcal{P}_0$.
Let $X := \lbrace x_1, \ldots, x_\ell \rbrace \subseteq A$ where $d_A(x_i) \geq \Delta'$ for all $1 \leq i \leq \ell$.
Suppose that $X \cap A_2 = \emptyset$ and $|X \setminus A_1| \geq r$.
Then $G$ contains a path system $\mathcal{P}$ such that $R_{\mathcal{V}}(\mathcal{P})$ is an Euler tour, $\bal_{AB}(\mathcal{P})=\bal_{AB}(\mathcal{P}_0) + \ell+r$ and $e(\mathcal{P}) \leq \ell+r+4$.
\end{proposition}

\begin{proof}
Write $F_{\mathcal{P}_0}(A) := (a_1,a_2)$.
So $|A_i|=a_i$ and hence $a_1+a_2 = |V(\mathcal{P}_0) \cap A| \leq 4$ by (BC2) and (BC3).
Therefore we can apply Proposition~\ref{build2paths} to see that
$$
{\rm acc}(G[A];A_1,A_2) \geq 2\ell - |X \cap A_1| - 2|X \cap A_2| \geq 2\ell-(\ell-r)=\ell+r.
$$
Then Proposition~\ref{addpaths} implies that there exists a path system $\mathcal{P}$ as required.
\end{proof}

The following lemma gives lower bounds for ${\rm acc}(G[A];A_1,A_2)$. Together with Proposition~\ref{addpaths}, this will enable us to see `how far' we can extend a basic connector.
We show that ${\rm acc}(G[A];A_1,A_2)$ is `sufficiently large' unless we are in one of two special cases.

\begin{lemma}\label{accommodation}
Let $k \in \lbrace 0,1 \rbrace$, $\Delta, \Delta',\ell \in \mathbb{N}$ be such that $\ell+k \geq 2$.
Suppose that $\Delta'/\Delta, \ell/\Delta' \ll 1$.
Let $G$ be a graph with vertex partition $U,A$ and suppose that
$e_G(A) \geq (\ell-1)\Delta+\Delta'$ and $\Delta(G[A]),\Delta(G[A,U]) \leq \Delta$.
Let $a_1,a_2\in\mathbb{N}_0$ with $a_1 \geq k$ and $\Delta' \geq 3\ell + a_1+a_2$.
Let $A_1,A_2 \subseteq A$ be disjoint such that $|A_i|=a_i$ for $i=1,2$.
Then one of the following holds.
\begin{itemize}
\item[(I)] ${\rm acc}(G[A]; A_1,A_2) \geq \ell-a_1-2a_2+k+2$;
\item[(II)] $k=1$, $(a_1,a_2)=(1,0)$ and ${\rm acc}(G[A];A_1,A_2) \geq \ell+1$;
\item[(III)] $k=1$, $1 \le \ell, a_1 + a_2 \le 2$, $e_G(A) \le \ell \Delta$ and ${\rm acc}(G[A];A_1,A_2) \geq \ell-a_2$.
Moreover, let $X := \lbrace x \in A : d_A(x) \ge \Delta' \rbrace$.
Then $|X| = \ell$ and all edges of $G[A]$ are incident with $X$.
\end{itemize}
\end{lemma}

\begin{proof}
Apply Lemma~\ref{goodmatching2} to $G[A]$.
Suppose first that (i) holds.
Let $M$ be a matching in $G[A]$ of size $\ell+1$ and let $uv \in E(G[A])$ be such that $u \notin V(M)$.
Obtain $M'$ from $M$ by deleting all those edges with both endpoints in $A_1$ or at least one endpoint in $A_2$.
Then $M'$ accommodates $A_1,A_2$ by construction, so
\begin{equation}\label{acc1}
{\rm acc}(G[A];A_1,A_2) \geq e(M') \geq \ell+1 - \lfloor a_1/2 \rfloor - a_2.
\end{equation}
If $\lceil a_1/2 \rceil + a_2 \geq k+1$,
then (\ref{acc1}) implies that (I) holds.

So suppose instead that $\lceil a_1/2 \rceil + a_2 \leq k$.
First consider the case $k=0$.
Then $\lceil a_1/2 \rceil + a_2 = 0$ and hence $(a_1,a_2)=(0,0)$.
Now $A_1=A_2=\emptyset$, so $M \cup \lbrace uv \rbrace$ is a path system which
accommodates $A_1,A_2$, and $e(M \cup \lbrace uv \rbrace) = \ell+2$, so (I) holds.

Now consider the case $k=1$.
We have $\lceil a_1/2 \rceil + a_2 \leq 1$.
But $a_1 \geq k \ge 1$ so $(a_1,a_2)=(1,0)$.
Observe that ${\rm acc}(G[A];A_1,A_2) \geq \ell+1$ by (\ref{acc1}).
So (II) holds.

Suppose now that the consequence~(i) of Lemma~\ref{goodmatching2} does not hold in $G[A]$.
Since $\ell \ge 1$, we have $e_G(A) \le \ell \Delta$ by the final assertion in Lemma~\ref{goodmatching2}.
Let $X := \lbrace x \in A : d_A(x) \ge \Delta' \rbrace$.
Then $|X| \geq \ell$. Since the consequence~(i) of Lemma~\ref{goodmatching2} does not hold, we must have that $|X| = \ell$ and
that all edges of $G[A]$ are incident with $X$.%
   \COMMENT{DK reformulated this sentence}

Apply Proposition~\ref{build2paths} to see that
\begin{align}\label{accheavy}
\nonumber {\rm acc}(G[A];A_1,A_2) &\geq 2\ell - |X \cap A_1| - 2|X \cap A_2| \geq 2\ell - \min \lbrace a_1,\ell-a_2 \rbrace - 2a_2\\
&= \ell - a_1 - 2 a_2 + \max \lbrace \ell,a_1+a_2 \rbrace \geq \ell-a_2.
\end{align}%
\COMMENT{$\ell - a_1 - 2 a_2 + \max \lbrace \ell,a_1+a_2 \rbrace \geq \ell - (a_1+a_2) - a_2 + (a_1+a_2)$}
In particular, if $\max \lbrace \ell,a_1+a_2 \rbrace \geq k+2$, (\ref{accheavy}) implies that (I) holds.
So we may suppose that $\max \lbrace \ell,a_1+a_2 \rbrace \leq k+ 1$.
Recall that $k+ \ell \geq 2$ and $ a_1 \ge k$ in the hypothesis.
Hence, we have $k =1$ and so $1 \le \ell, a_1 + a_2 \le 2$.
So (III) holds.
\end{proof}

We are now ready to prove Lemma~\ref{aim} in the case when $|A|-|B| \geq 2$ and $m \leq 2$.
Roughly speaking, the approach is as follows.
Proposition~\ref{BC} implies that $G$ contains a basic connector $\mathcal{P}_0$.
When $m=2$, Lemmas~\ref{3matching} and~\ref{2matchingcor} allow us to assume that $\bal_{AB}(\mathcal{P}_0)$ is non-negative.
We would like to extend $\mathcal{P}_0$ to a path system $\mathcal{P}$ in such a way that $R_{\mathcal{V}}(\mathcal{P})$ is an Euler tour and $\bal_{AB}(\mathcal{P})=\ell+m/2 \geq |A|-|B|$. 
Proposition~\ref{addpaths} implies that, in order to do this, it suffices to find a path system $\mathcal{P}_A$ in $G[A]$ which accommodates $A_1,A_2$ (where $A_i$ is the collection of all those vertices in $A$ with degree $i$ in $\mathcal{P}_0$) and has enough edges.
Now Lemma~\ref{accommodation} implies that we can do this unless $m=2$, $\ell$ is small and $(|A_1|,|A_2|)$ takes one of a small number of special values.
Some additional arguments are required in these cases.

\medskip
\noindent
\emph{Proof of Lemma~\emph{\ref{aim}} in the case when $|A|-|B| \geq 2$ and $m \leq 2$.}%
    \COMMENT{DK changed quite a bit in this proof...}
Let $k := m/2$.
Since $m \in 2 \mathbb{N}_0$ we have $k \in \lbrace 0,1 \rbrace$.
Let $\Delta := D/2$, $\Delta' := \nu n$ and $U := V_1 \cup V_2$.
Proposition~\ref{matchingsizes} implies that
\begin{equation}\label{l+k}
\ell + k \geq |A|-|B| \geq 2.
\end{equation}
Proposition~\ref{ell} implies that $\ell,m \leq 12\rho n$.
Then $\Delta'/\Delta, \ell/\Delta', m/\Delta' \ll 1$, $\Delta'/\Delta \ll \eps$.
Proposition~\ref{charedges} implies that
\begin{equation}\label{lm}
e_G(A) \geq (\ell-1)\Delta+\Delta'\ \ \mbox{ and }\ \ e_G(A,U) \geq (m-1)\Delta+\Delta'.
\end{equation}
By Proposition~\ref{BC}, $G$ contains a basic connector~$\mathcal{P}_0$.
Further assume that  $\bal_{AB}(\mathcal{P}_0)$ is maximal, and given $\bal_{AB}(\mathcal{P}_0)$, $a_1$ is maximal where $F_{\mathcal{P}_0}(A):=(a_1,a_2)$.
Let
\begin{equation}
\nonumber
t := |A|-|B|-\bal_{AB}(\mathcal{P}_0).
\end{equation}
Then (BC2) implies that $t \geq 0$.
In fact we may assume that $t \geq 1$ as otherwise $\mathcal{P}_0$ satisfies (P1)--(P3). For $i=1,2$ let $A_i$ be the set of all those
vertices in $A$ which have degree $i$ in $\mathcal{P}_0$. So $|A_i|=a_i$.
Proposition~\ref{addpaths} implies that, to prove Lemma~\ref{aim}, it suffices to show that
\begin{equation*}
{\rm acc}(G[A];A_1,A_2) \ge t.
\end{equation*}
(To check (P1), note that (BC2) and (\ref{l+k}) imply $t \le |A| - |B| +2 \le \ell +k + 2 \le \ell +m + 2$.)

\medskip
\noindent
\textbf{Claim A.} \begin{itemize}
\item[(i)] \emph{Suppose that $k=1$. Then $\bal_{AB}(\mathcal{P}_0)\ge 0$, and if $\bal_{AB}(\mathcal{P}_0)=0$ then $a_1\ge 1$.}
\item[(ii)] $a_1 \geq k$.%
	\COMMENT{KS: We do need to prove the claim like this. Lemma~\ref{2matchingcor} implies there exists $\mathcal{P}$ with $\bal_{AB}(\mathcal{P}) \geq 0$ and $a_1 \geq 1$ (for this $\mathcal{P}$).
So this could only give us some $\mathcal{P}$ with $\bal_{AB}(\mathcal{P})=0$ and $a_1 = 1$ (for this $\mathcal{P}$).
By our choice of $\mathcal{P}_0$, we have that $\bal_{AB}(\mathcal{P}_0) \geq 0$. 
So we could have $\bal_{AB}(\mathcal{P}_0) \geq 1$.
But then the existence of $\mathcal{P}$ tells us nothing about $a_1$ for $\mathcal{P}_0$, so this needs to be checked separately.} 
\end{itemize}

\medskip
\noindent
To prove Claim~A(i), note that if $k=1$ (and so $m=2$), then (\ref{lm}) and Lemma~\ref{goodmatching2} imply
that $G[A,U]$ contains a matching of size two.
Together with Lemma~\ref{2matchingcor} and our choice of $\mathcal{P}_0$ this in turn implies Claim~A(i).
Claim~A(ii) clearly holds if $k=0$, so assume $k=1$.
If $\bal_{AB}(\mathcal{P}_0) = 2$, then $a_1 \ge 1$ by~(BC4). Together with Claim~A(i) this shows that we 
may assume that $\bal_{AB}(\mathcal{P}_0) = 1$.
By~(BC4), we may further assume that $(a_1,a_2) = (0,1)$.
Then (\ref{BCeq}) implies that $e_{\mathcal{P}_0}(B,U) = 0$.
But then $\mathcal{P}_0$ has no endpoints in $W=A \cup B$, contradicting (BC1).
This completes the proof of Claim~A.

\medskip

\noindent
Apply Lemma~\ref{accommodation} with $G \setminus B, A , U, F_{\mathcal{P}_0}(A),\ell,k$ playing the roles of $G,A,U,(a_1,a_2),\ell,k$.
Suppose first that (I) holds, so
$$
{\rm acc}(G[A];A_1,A_2) \geq \ell-a_1-2a_2+k+2 \stackrel{{\rm (BC4)},(\ref{l+k})}{\geq} |A|-|B| - \bal_{AB}(\mathcal{P}_0) = t,
$$
as required.
Therefore we may assume that one of the consequences~(II) or (III) of Lemma~\ref{accommodation} holds.
So $k=1$ and therefore $\bal_{AB}(\mathcal{P}_0)\ge 0$ by Claim~A(i).
Suppose first that (II) holds. Then
$$
{\rm acc}(G[A];A_1,A_2) \geq \ell + 1 \stackrel{(\ref{l+k})}{\geq} |A|-|B| \geq t,
$$
as required.
Therefore we may assume that (III) holds.
So $1 \leq \ell,a_1+a_2 \leq 2$, $e_G(A) \leq \ell\Delta$ and ${\rm acc}(G[A];A_1,A_2) \geq \ell-a_2$.
Let $X := \lbrace x \in A: d_A(x) \geq \Delta' \rbrace$.
Then the consequence~(III) of Lemma~\ref{accommodation} also implies that $|X|=\ell$ and all edges of $G[A]$ are incident with $X$.

We claim that we are done if $\bal_{AB}(\mathcal{P}_0) \neq a_2$.
To see this, suppose first that $\bal_{AB}(\mathcal{P}_0) \leq a_2-1$.
Since $\bal_{AB}(\mathcal{P}_0) \geq 0$ this implies that $a_2 = 1$ and $\bal_{AB}(\mathcal{P}_0) =0$.
But $a_1 \geq k\ge 1$ by Claim~A(ii) and $a_1+a_2 \leq 2$, so $a_1=a_2=1$.
This is a contradiction to (BC4).
Suppose instead that $\bal_{AB}(\mathcal{P}_0) \geq a_2+1$.
Then
\begin{align*}
{\rm acc}(G[A];A_1,A_2) \geq \ell-a_2 \geq \ell + 1 - \bal_{AB}(\mathcal{P}_0) = \ell+1 - (|A|-|B|) + t
\stackrel{(\ref{l+k})}{\geq} t.
\end{align*}
Therefore we may assume that $\bal_{AB}(\mathcal{P}_0)=a_2$.
In particular, this together with (BC4) implies that $\bal_{AB}(\mathcal{P}_0) \in \lbrace 0,1 \rbrace$.
We claim that we can further assume that
\begin{equation}\label{ellequal}
\ell=|A|-|B|-1.
\end{equation}
Indeed, to see this, note that by (\ref{l+k}), it suffices to show that we are done if $\ell \geq |A|-|B|$.
But in this case we have
${\rm acc}(G[A];A_1,A_2) \geq \ell-a_2 \geq |A|-|B| - a_2 = t$, as required.

We will now distinguish two cases.

\medskip
\noindent
\textbf{Case 1.}
\emph{$G[A,U]$ contains a matching of size three.}%
\COMMENT{KS: we can't assume that $\bal_{AB}(\mathcal{P}_0) \leq 0$, as we could have $(a_1,a_2)=(1,1)$.}

\medskip
\noindent
Recall that $\bal_{AB}(\mathcal{P}_0) \in \lbrace 0,1\rbrace$.
So Lemma~\ref{3matching} and our choice of $\mathcal{P}_0$ imply that $a_1 \geq 2$.
Since $a_1+a_2 \leq 2$ we have that $(a_1,a_2)=(2,0)$.
Therefore $\bal_{AB}(\mathcal{P}_0) = a_2=0$.
Now, by Lemma~\ref{3matching} and our choice of $\mathcal{P}_0$ we deduce that there is some
$i\in \{1,2\}$ such that for $j\in \{1,2\}\setminus \{i\}$ and
for each $a \in A$, there are matchings $M_{a,A},M_{a,B}$ in $G[A \setminus \lbrace a \rbrace,V_i],G[B,V_j]$ respectively,
each of which has size two. Moreover, $\mathcal{P}_a := M_{a,A} \cup M_{a,B}$ is a basic connector
with $\bal_{AB}(\mathcal{P}_{a})=0$.

Let $x \in X$ be arbitrary. (Recall that $|X|=\ell \geq 1$.)
Apply Proposition~\ref{2paths2} with $\mathcal{P}_{x},V(M_{x,A}) \cap A,\emptyset,X,\ell,1$ playing the
roles of  $\mathcal{P}_0,A_1,A_2,X,\ell,r$ to obtain a path system $\mathcal{P}$ in $G$ such that
$R_{\mathcal{V}}(\mathcal{P})$ is an Euler tour, $\bal_{AB}(\mathcal{P})=\bal_{AB}(\mathcal{P}_{x})+\ell+1=|A|-|B|$
(using (\ref{ellequal})), and $e(\mathcal{P}) \leq \ell+5$. 
Thus, $\mathcal{P}$ satisfies (P1)--(P3).

\medskip
\noindent
\textbf{Case 2.}
\emph{$G[A,U]$ does not contain a matching of size three.}

\medskip
\noindent
Together with K\"onig's theorem on edge-colourings this implies that $e_G(A,U) \leq 2\Delta$.

\medskip
\noindent
\textbf{Claim B.}
\emph{$X \cap V(\mathcal{P}_0) = \emptyset$.}

\medskip
\noindent
Since $e_G(A,U) \leq 2\Delta$, the consequence~(ii) of Proposition~\ref{fact2} implies that
\begin{equation*}
e_G(A) \geq \Delta(|A|-|B|) - e_G(A,U)/2 \stackrel{(\ref{ellequal})}{\ge } \ell\Delta.
\end{equation*}
In fact, equality holds since $e_G(A) \leq \ell\Delta$ by the consequence~(III) of Lemma~\ref{accommodation}.
Since all edges of $G[A]$ are incident with $X$ and $|X|=\ell$ it follows that $d_A(x) = \Delta =D/2$ for all $x \in X$.
For all $x \in X$,  $d_U(x)=D-d_A(x)-d_B(x) \leq D-2d_A(x) = D-2\Delta=0$.
The claim follows by~(BC3).

\medskip
\noindent
Recall that we assume that $t\ge 1$.
Observe that, since $\bal_{AB}(\mathcal{P}_0) \in \lbrace 0,1 \rbrace$, the definition of $t$ and
(\ref{ellequal}) imply that $1\le t \leq |A|-|B|=\ell+1$. Choose an arbitrary $X' \subseteq X$ with $|X'|=t-1$.
Apply Proposition~\ref{2paths2} with $\mathcal{P}_{0},X',t-1,1$ playing the roles of  $\mathcal{P}_0,X,\ell,r$ to obtain a path system $\mathcal{P}$ in $G$ such that $R_{\mathcal{V}}(\mathcal{P})$ is an Euler tour, $\bal_{AB}(\mathcal{P})=\bal_{AB}(\mathcal{P}_{0})+t =|A|-|B|$,
and $e(\mathcal{P}) \leq \ell+5$. Thus, $\mathcal{P}$ satisfies (P1)--(P3). 
\hfill$\square$

\subsection{The proof of Lemma~\ref{aim} in the case when $|A|=|B|+1$.}\label{+1}

Note that the extremal example in Figure~\ref{fig:exactex}(i) satisfies the conditions of this case.
Therefore the degree bound $D \geq n/4$ is essential here.
We will follow a similar strategy as in Section~\ref{sparse}.
We first find a basic connector $\mathcal{P}_0$ and then modify it to obtain a path system $\mathcal{P}$ satisfying (P1)--(P3).
To be more precise, $\mathcal{P}$ will satisfy $e(\mathcal{P}) \leq 6$ and $\bal_{AB}(\mathcal{P})=1$.
Throughout this section, we will assume that the basic connector $\mathcal{P}_0$ is chosen so that $|\bal_{AB}(\mathcal{P}_0)-1|$ is minimal.
We will distinguish cases depending on the value of $\bal_{AB}(\mathcal{P}_0)$.

Let $G$ be a $D$-regular graph with vertex partition $A,B,U$ where $|A|=|B|+1$.
Then the consequence~(i) of Proposition~\ref{fact2} implies that
\begin{equation}\label{balance1}
2e_G(A) + e_G(A,U) = 2e_G(B) + e_G(B,U) + D.
\end{equation}

We will need the following simple facts for the case when $|\bal_{AB}(\mathcal{P}_0)|= 2$.%
   \COMMENT{DK: changed (ii) and its proof}

\begin{proposition}\label{balmin}
Let $G$ be a $3$-connected graph with vertex partition $\mathcal{V} = \lbrace V_1,V_2,W := A \cup B \rbrace$.
Then the following holds:
\begin{itemize}
\item[(i)] if $\mathcal{P}_0$ is a basic connector in $G$ with $\bal_{AB}(\mathcal{P}_0)=2$, then $V(\mathcal{P}_0) \cap B = \emptyset$ and $\mathcal{P}_0[A,V_i]$ is a matching of size two for each $i=1,2$. In particular, $\mathcal{P}_0[A,V_1 \cup V_2]$ contains a matching of size three.
\item[(ii)] if $e_G(B,U)>0$ and $G$ contains a basic connector $\mathcal{P}'_0$ with $\bal_{AB}(\mathcal{P}'_0)= 2$,
then $G$ also contains a basic connector $\mathcal{P}_0$ with $\bal_{AB}(\mathcal{P}_0)= 1$;
\item[(iii)] if $e_G(A,U)>0$ then $G$ contains a basic connector $\mathcal{P}_0$ with $\bal_{AB}(\mathcal{P}_0) \geq -1$;
\item[(iv)] if $e_G(A,U),e_G(B,U)>0$ then $G$ contains a basic connector $\mathcal{P}_0$ with $|\bal_{AB}(\mathcal{P}_0)| \leq 1$.
\end{itemize}
\end{proposition}

\begin{proof}
(i) follows immediately from (BC1)--(BC4).
To prove (ii), note that by (i), for both $i=1,2$ there are matchings $M_i \subseteq G[A,V_i]$ of size two such that
$\mathcal{P}'_0 = M_1 \cup M_2$.
Let $e \in E(G[B,U])$ be arbitrary.
Without loss of generality, suppose that $e \in E(G[B,V_1])$.
If possible, let $e' \in E(M_1)$ be the edge incident with $e$; otherwise let $e' \in E(M_1)$ be arbitrary.
Then $\mathcal{P}_0 := (\mathcal{P}'_0 \cup \lbrace e \rbrace )\setminus \lbrace e' \rbrace$ is a basic connector with
$\bal_{AB}(\mathcal{P}_0)=1$, as required.
(iii) and (iv) follow from Proposition~\ref{BC} together with an argument similar to the one for (ii).%
\COMMENT{I don't think (iv) follows from the statements of (ii) and (iii)...}
\end{proof}

The next lemma concerns the case when $G[A,V_1 \cup V_2]$ contains a matching of size three.
This extra condition ensures the existence of a basic connector with useful properties of which we can take advantage.%
   \COMMENT{DK: changed statement + proof of next lemma}

\begin{lemma}\label{3okay}
Let $n,D \in \mathbb{N}$ be such that $D \geq n/4$ and $1/n\ll 1$.%
    \COMMENT{DK added $1/n\ll 1$ since we need $2/D=1/\Delta\ll 1$ in order to apply Lemma~\ref{goodmatching2}}
Let $G$ be a $3$-connected $D$-regular graph with vertex partition $\mathcal{V} = \lbrace V_1,V_2,W := A \cup B \rbrace$,
where $|V_i| \geq D/2$ for $i=1,2$.
Suppose that $|A|=|B|+1$, that $\Delta(G[A,V_1 \cup V_2])\le D/2$ and that $G[A,V_1 \cup V_2]$ contains a matching of size three.
Then $G$ contains a path system $\mathcal{P}$ which satisfies \emph{(P1)--(P3)}.
\end{lemma}

\begin{proof}
Let $U:=V_1\cup V_2$. Without loss of generality we may assume that $e_G(A,V_1)\le e_G(A,V_2)$.
We will obtain $\mathcal{P}$ by adding at most two edges to a basic connector $\mathcal{P}_0$.
Therefore $e(\mathcal{P}) \leq 6$ so (P1) will hold.
We may assume that there does not exist a basic connector $\mathcal{P}_0'$ with $\bal_{AB}(\mathcal{P}_0') = 1$
(otherwise we can take $\mathcal{P} := \mathcal{P}_0'$).
Apply Lemma~\ref{3matching} to obtain a basic connector in~$G$ which satisfies (i) or (ii).

\medskip
\noindent
\textbf{Case 1.}
\emph{The consequence~(i) of Lemma~\ref{3matching} holds.}

\medskip
\noindent
So $G$ contains a basic connector $\mathcal{P}_0$ such that $\bal_{AB}(\mathcal{P}_0) \geq 1$ and,
if $F_{\mathcal{P}_0}(A) = (a_1,a_2)$, then $a_1 \geq 2$.
Thus $\bal_{AB}(\mathcal{P}_0)=2$ by our assumption.
The consequence~(i) of Proposition~\ref{balmin} implies that $V(\mathcal{P}_0) \cap B = \emptyset$.
Furthermore, the consequence~(ii) of Proposition~\ref{balmin} implies that $e_G(B,U)=0$.
Suppose that $e_G(B) \geq 1$.
For arbitrary $e \in E(G[B])$ we have that $\mathcal{P} := \mathcal{P}_0 \cup \lbrace e \rbrace$ satisfies (P1)--(P3).
So we may assume that $e_G(B)=0$.
So (\ref{balance1}) implies that
\begin{equation}\label{AUsum}
2e_G(A) + e_G(A,U) = D.
\end{equation}
Moreover, for each $b \in B$ we have that $N_G(b) \subseteq A$ and thus $|A| \geq D$.
So $|B| \geq D-1$ and since $D \geq n/4$ we have that $|U| \leq 2D+1$.
We will only prove the case when $|V_1| = D-s$ for some $s \in \mathbb{N}_0$.
(The same argument also works for $|V_2|=D-s$.)
Recall that $s \leq D/2$ by assumption.
Then every vertex in $V_1$ has at least $s+1$ neighbours in $\overline{V_1}$.
Since $e_G(B,U)=0$ and $e_G(A,V_1)\le e_G(A,V_2)$ we have that%
\COMMENT{for $0\le s\le D/2$ the function $(s+1)(D-s) - D/2=-s^2+sD+D/2$ is minimized if $s=0$}
\begin{eqnarray*}
e_G(V_1,V_2) \geq e_G(V_1,\overline{V_1}) - e_G(A,V_1)
\stackrel{(\ref{AUsum})}{\geq} (s+1)(D-s) - D/2\geq D/2.
\end{eqnarray*}  
Suppose that $\mathcal{P}_0$ is a matching of size four in $G[A,U]$.
Then, given any $e \in E(G[V_1,V_2])$, we can choose $e_i \in \mathcal{P}_0[A,V_i]$ such that $e,e_1,e_2$ is a matching of size three.
Otherwise, the consequence~(i) of Proposition~\ref{balmin} implies that $\mathcal{P}_0$ consists of vertex-disjoint paths $u_1a_1,u_2a_2,v_1av_2$, where $v_i,u_i \in V_i$ and $a,a_1,a_2 \in A$.
Since $e_G(V_1,V_2) \geq 2$, we can pick $e \in E(G[V_1,V_2]) \setminus \lbrace u_1u_2 \rbrace$.
It is easy to see that we can similarly find $e_i \in E(\mathcal{P}_0[A,V_i])$ such that $e,e_1,e_2$ is a matching of size three. 
In both cases, $\mathcal{P} := \lbrace e,e_1,e_2 \rbrace$ satisfies (P1)--(P3).

\medskip
\noindent
\textbf{Case 2.}
\emph{The consequence~(ii) of Lemma~\ref{3matching} holds.}

\medskip
\noindent
Since $e_G(A,V_1)\le e_G(A,V_2)$ this implies that $e_G(V_1,A)=0$.
Moreover, the consequence~(ii) of Lemma~\ref{3matching} also implies that, for each $a \in A$, there are matchings $M_{a,A},M_{a,B}$
in $G[A \setminus \lbrace a \rbrace,V_2], G[B,V_1]$ respectively, each of which has size two.
In particular $e_G(B,U) \geq 2$.%
\COMMENT{LATE CHANGE: new sentence.}
Suppose that $e_G(A) > 0$.
Let $aa' \in E(G[A])$.
Then $\mathcal{P} := M_{a,A} \cup M_{a,B} \cup \lbrace aa' \rbrace$ satisfies (P1)--(P3).
So we may assume that $e_G(A) = 0$.
Then (\ref{balance1}) implies that $e_G(A,V_2) =e_G(A,U) \geq D + e_{G}(B,U) \geq D+2$.%
\COMMENT{LATE CHANGE: New calculation, to save defining $\mathcal{P}_a$.}
The `moreover' part of Lemma~\ref{goodmatching2} with $G[A,V_2],D/2,2$ playing the roles of $G,\Delta, \ell$ implies that $G[A,V_2]$ contains a matching $M_A$ of size three and an edge $xy$ with $x \notin V(M_A)$.
Let $a \in A$ be arbitrary.
Then $\mathcal{P} := M_{a,B} \cup M_A \cup \lbrace xy \rbrace$ satisfies (P1)--(P3).
\end{proof}

The following proposition will be used to find edges in $G[A]$ which can be added to
a basic connector $\mathcal{P}_0$ so that it is still a path system and $R_{\mathcal{V}}(\mathcal{P}_0)$ is still an Euler tour.
For example, if $a \in A$ is such that $d_{\mathcal{P}_0}(a)=2$, then we cannot add any edges in $G[A]$ which are incident with $a$.
(Recall that the partition given in Lemma~\ref{aim} satisfies $d_A(a) \leq d_B(a)$ for all $a \in A$.)

\begin{proposition}\label{sumfact}
Let $G$ be a $D$-regular graph with vertex partition $A,B,U$ where $|A|=|B|+1$.
Let $a \in A$ be such that $d_A(a) \leq d_B(a)$.
Then
$$
2e_G(A \setminus \lbrace a \rbrace) + e_G(A \setminus \lbrace a \rbrace,U) \geq e_G(B,U).
$$
\end{proposition}

\begin{proof}
Note that%
\COMMENT{LATE CHANGE: Removed (i) from Prop.}
\begin{align*}
2e_G(A \setminus \lbrace a \rbrace) + e_G(A \setminus \lbrace a \rbrace,U) &= 2e_G(A) + e_G(A,U) - 2d_A(a) - d_U(a)\\
&\geq 2e_G(A) + e_G(A,U) - d_A(a) - d_B(a) - d_U(a)\\
&= 2e_G(A) + e_G(A,U) - D \stackrel{(\ref{balance1})}{\geq} e_G(B,U),
\end{align*}
as required. 
\end{proof}

By Lemma~\ref{3okay}, we may assume that $G[A,V_1 \cup V_2]$ contains no matching of size three.
Then the consequence~(i) of Proposition~\ref{balmin} allows us to assume that $\bal_{AB}(\mathcal{P}_0) \leq 0$ (or we are done). 
In the next lemma, we consider the case when $\bal_{AB}(\mathcal{P}_0)=0$.

\begin{lemma}\label{bal0}
Let $D \in \mathbb{N}$.
Let $G$ be a $3$-connected $D$-regular%
    \COMMENT{Deryk added $D$-regular}
graph with vertex partition $\mathcal{V} = \lbrace V_1,V_2,W := A \cup B \rbrace$.
Suppose that $|A|=|B|+1$, $\Delta(G[A,V_1 \cup V_2]) \leq D/2$ and $d_A(a) \leq d_B(a)$ for all $a \in A$.
Suppose further that $G[A,V_1 \cup V_2]$ does not contain a matching of size three.
Let $\mathcal{P}_0$ be a basic connector in $G$ with $\bal_{AB}(\mathcal{P}_0) = 0$.
Then $G$ contains a path system $\mathcal{P}$ which satisfies \emph{(P1)--(P3)}.
\end{lemma}

\begin{proof}
Let $U := V_1 \cup V_2$.
Since $G[A,U]$ does not contain a matching of size three, K\"onig's theorem on edge-colourings implies that
\begin{equation}\label{AUedges}
e_G(A,U) \leq D.
\end{equation}
Property (BC4) implies that $a_1+2a_2 \in \lbrace 1,2 \rbrace$ and so $F_{\mathcal{P}_0}(A) \in \lbrace (2,0),(1,0),(0,1) \rbrace$.
We will distinguish cases based on the value of $F_{\mathcal{P}_0}(A)$.

\medskip
\noindent
\textbf{Case 1.}
\emph{$F_{\mathcal{P}_0}(A) = (2,0)$.}

\medskip
\noindent
Then (\ref{BCeq}) implies that $e_{\mathcal{P}_0}(A,U) = e_{\mathcal{P}_0}(B,U)=2$.
Since $\mathcal{P}_0$ is an Euler tour and $e(\mathcal{P}_0) \leq 4$ by (BC1) and (BC2), there are distinct vertices $a,a' \in A$, a collection of distinct vertices $X := \lbrace u,u',v,v' \rbrace \subseteq U$ with $|X \cap V_i| = 2$ for $i=1,2$ and $b,b' \in B$ which are not necessarily distinct, such that
$\mathcal{P}_0 := \lbrace au,a'u',bv,b'v' \rbrace$.%
\COMMENT{The old proof didn't work because I was assuming that $b,b'$ were distinct. So Case 1.b. below is essentially the extra part.}

Observe that we are done if there exists $e \in E(G[A]) \setminus \lbrace aa' \rbrace$ since then $\mathcal{P}_0 \cup \lbrace e \rbrace$ satisfies (P1)--(P3).
So we may assume that $E(G[A]) \subseteq \lbrace aa' \rbrace$.
Now
$$
2 = e_{\mathcal{P}_0}(B,U) \leq e_G(B,U) \stackrel{(\ref{AUedges})}{\leq} 2e_G(B) + e_G(B,U) + D - e_G(A,U) \stackrel{(\ref{balance1})}{=} 2e_G(A) \leq 2.
$$
Therefore we have $e_G(B)=0$, $e_G(A)=1$, $e_G(A,U)=D$ and $e_G(B,U)=2$, so
$E(G[B,U]) = \lbrace bv, b'v' \rbrace$ and $E(G[A]) = \lbrace aa' \rbrace$.

We will assume that either $\lbrace u,u' \rbrace \subseteq V_1$ and $\lbrace v,v' \rbrace \subseteq V_2$; or $\lbrace u,v \rbrace \subseteq V_1$ and $\lbrace u',v' \rbrace \subseteq V_2$ since the other cases are similar.

\medskip
\noindent
\textbf{Case 1.a.}
\emph{$\lbrace u,u' \rbrace \subseteq V_1$ and $\lbrace v,v' \rbrace \subseteq V_2$.}

\medskip
\noindent
Suppose that $e_G(V_1,V_2) \neq 0$.
Let $v_1v_2 \in E(G[V_1,V_2])$ with $v_i \in V_i$.
Choose $e_1 \in \mathcal{P}_0[A,V_1 \setminus \lbrace v_1 \rbrace]$ and $e_2 \in \mathcal{P}_0[B,V_2 \setminus \lbrace v_2 \rbrace]$.
Then $\mathcal{P} := \lbrace e_1,e_2,v_1v_2,aa' \rbrace$ satisfies (P1)--(P3).
Suppose that $e_G(A,V_2) \neq 0$.
Let $a''x_2 \in E(G[A,V_2])$ with $a'' \in A$ and $x_2 \in V_2$.
Choose $e_2 \in \mathcal{P}_0[B,V_2 \setminus \lbrace x_2 \rbrace]$.
Then $\mathcal{P} := \lbrace au,a'u',a''x_2,e_2 \rbrace$ satisfies (P1)--(P3).
Therefore $e_G(A \cup V_1,V_2)=0$.
So $E(G[V_2,\overline{V_2}]) = \lbrace bv,b'v' \rbrace$, contradicting the $3$-connectivity of $G$.

\medskip
\noindent
\textbf{Case 1.b.}
\emph{$\lbrace u,v \rbrace \subseteq V_1$ and $\lbrace u',v' \rbrace \subseteq V_2$.}

\medskip
\noindent
We may assume that $b=b'$ since otherwise $\mathcal{P} := \mathcal{P}_0 \cup \lbrace aa' \rbrace$ satisfies (P1)--(P3).
Since $G[A,U]$ does not contain a matching of size three, every edge in $G[A,U]$ is incident with at least one of $a,a',u,u'$.
Suppose that there exists $a'' \in A \setminus \lbrace a,a' \rbrace$ such that $ua'' \in E(G)$.
Then $\mathcal{P} := \mathcal{P}_0 \cup \lbrace ua'',aa' \rbrace \setminus \lbrace ua \rbrace$ satisfies (P1)--(P3).
A similar deduction can be made with $u'$ playing the role of $u$.
Therefore every edge in $G[A,U]$ is incident with $a$ or $a'$.
Since $e_G(A,U) = D$ we have $d_U(a),d_U(a') =D/2$.

Suppose that $e_G(V_1,V_2) \neq 0$.
Let $v_1v_2 \in E(G[V_1,V_2])$ with $v_i \in V_i$.
If $v_1 \neq u$ and $v_2 \neq u'$ then $\mathcal{P} := \lbrace au,a'u',v_1v_2 \rbrace$ satisfies (P1)--(P3).
Therefore we may suppose, without loss of generality, that $v_1=u$.
Suppose that $v_2 \neq u'$.
Then $\mathcal{P} := \lbrace a'u',v_1v_2,bv,aa' \rbrace$ satisfies (P1)--(P3).
Therefore we may suppose that $v_2=u$.
Thus $uu' \in E(G)$.
Since $d_U(a) \geq D/2$, we can choose $w \in N_U(a) \setminus \lbrace v,v',u,u' \rbrace$.
Suppose that $w \in V_1$.
Then $\mathcal{P} := \lbrace aw,uu',aa',bv' \rbrace$ satisfies (P1)--(P3).
If $w \in V_2$ then $\mathcal{P} := \lbrace aw,uu',aa',bv \rbrace$ satisfies (P1)--(P3).

Thus we may assume that $e_G(V_1,V_2)=0$.
Choose $Y_a \in \lbrace V_1,V_2 \rbrace$ such that $d_{Y_a}(a) \geq D/4$.
Note that there is always such a $Y_a$.
Define $Y_{a'}$ analogously.
Suppose that $Y_{a'} = V_1$.
Choose $w' \in N_{V_1}(a') \setminus \lbrace u,v \rbrace$.
Then $\mathcal{P} := \mathcal{P}_0 \cup \lbrace a'w' \rbrace \setminus \lbrace bv \rbrace$ satisfies (P1)--(P3).
We can argue similarly if $Y_a = V_2$.

Therefore we may assume that $Y_{a'}=V_2$ and $Y_a=V_1$.
Suppose that $d_{V_1}(a') \neq 0$.
Let $w' \in N_{V_1}(a')$.
Since $d_{V_1}(a) \geq D/4$, we can choose $w \in N_{V_1}(a) \setminus \lbrace w'\rbrace$.%
    \COMMENT{DK replaced $w \in N_{V_1}(a) \setminus \lbrace w', v\rbrace$ by $w \in N_{V_1}(a) \setminus \lbrace w'\rbrace$}
Then $\mathcal{P} := \mathcal{P}_0 \cup \lbrace aw, a'w' \rbrace \setminus \lbrace au, bv \rbrace$ satisfies (P1)--(P3).%
    \COMMENT{DK replaced $\mathcal{P} := \mathcal{P}_0 \cup \lbrace a'w' \rbrace \setminus \lbrace bv \rbrace$
by $\mathcal{P} := \mathcal{P}_0 \cup \lbrace aw, a'w' \rbrace \setminus \lbrace au, bv \rbrace$}
So $d_{V_1}(a') = 0$.
Since every edge of $G[A,U]$ is incident with $a$ or $a'$, we have that 
every edge in $G[A,V_1]$ is incident with $a$.
We have shown that every edge in $G[V_1,\overline{V_1}]$ is incident with $a$ or $b$, contradicting the $3$-connectivity of $G$.

\medskip
\noindent
\textbf{Case 2.}
\emph{$F_{\mathcal{P}_0}(A) = (1,0)$.}

\medskip
\noindent
Then (\ref{BCeq}) implies that
$
e_G(B,U) \geq e_{\mathcal{P}_0}(B,U)=1.
$
So (\ref{balance1}) and (\ref{AUedges}) give $2e_G(A) = D + 2e_G(B) + e_G(B,U) - e_G(A,U) \geq 1$.
Let $e \in E(G[A])$ be arbitrary.
Then $\mathcal{P} := \mathcal{P}_0 \cup \lbrace e \rbrace$ satisfies  (P1)--(P3).

\medskip
\noindent
\textbf{Case 3.}
\emph{$F_{\mathcal{P}_0}(A) = (0,1)$.}

\medskip
\noindent
Now (\ref{BCeq}) implies that $e_{\mathcal{P}_0}(B,U)=e_{\mathcal{P}_0}(A,U)=2$.
So (BC2) implies that $e_{\mathcal{P}_0}(V_1,V_2)=0$ and that there exist distinct $v_i,u_i \in V_i$ for $i=1,2$, and $b,b' \in B$ and $a \in A$ such that $\mathcal{P}_0 = \lbrace v_1b, v_2b', u_1au_2 \rbrace$.
Proposition~\ref{sumfact} implies that $2e_G(A \setminus \lbrace a \rbrace) + e_G(A \setminus \lbrace a \rbrace,U) \geq 2$.
Suppose first that $e_G(A \setminus \lbrace a \rbrace) \geq 1$.
Choose $e \in E(G[A \setminus \lbrace a \rbrace])$. 
Then $\mathcal{P} := \mathcal{P}_0 \cup \lbrace e \rbrace$ satisfies (P1)--(P3).
Therefore we may assume that $e_G(A \setminus \lbrace a \rbrace,U) \geq 2$.
Suppose there exists $e' \in E(G[A \setminus \lbrace a \rbrace,U \setminus \lbrace u_1,u_2 \rbrace])$.
Without loss of generality, suppose that $e'$ has an endpoint in $V_1$.
Then $\mathcal{P} := \mathcal{P}_0 \cup \lbrace e' \rbrace \setminus \lbrace v_1b \rbrace$ satisfies (P1)--(P3).
Therefore we may assume that $G$ contains an edge $a'u_1$ where $a' \in A \setminus \lbrace a \rbrace$.
Let $\mathcal{P}_0' := \mathcal{P}_0 \cup \lbrace a'u_1 \rbrace \setminus \lbrace au_1 \rbrace$.
Then $\mathcal{P}'_0$ is a basic connector with $\bal_{AB}(\mathcal{P}_0')=0$ and $F_{\mathcal{P}_0'}(A)=(2,0)$.
So we are in Case 1.
\end{proof}

The next lemma concerns the case when $\bal_{AB}(\mathcal{P}_0)=-1$.

\begin{lemma}\label{bal-1}
Let $D \in \mathbb{N}$ where $D \geq 12$.
Let $G$ be a $3$-connected $D$-regular graph with vertex partition $\mathcal{V} = \lbrace V_1,V_2,W := A \cup B \rbrace$.
Suppose that $|A|=|B|+1$, $\Delta(G[A,V_1 \cup V_2] \leq D/2$ and $d_A(a) \leq d_B(a)$ for all $a \in A$.
Let $\mathcal{P}_0$ be a basic connector in $G$ such that $|\bal_{AB}(\mathcal{P}_0)-1|$ is minimal.
Suppose that $\bal_{AB}(\mathcal{P}_0) = -1$.
Then $G$ contains a path system $\mathcal{P}$ which satisfies \emph{(P1)--(P3)}.
\end{lemma}

\begin{proof}
Let $U := V_1 \cup V_2$.
Observe that $G[A,U]$ does not contain a matching of size two since otherwise Lemma~\ref{2matchingcor} would imply that $\bal_{AB}(\mathcal{P}_0) \geq 0$.
Therefore $e_G(A,U) \leq D/2$, and so (\ref{balance1}) implies that
\begin{equation}\label{D/4bound}
e_G(A) \geq D/4.
\end{equation}
Write $F_{\mathcal{P}_0}(A) := (a_1,a_2)$.
Then (BC4) implies that $a_1+2a_2 \in \lbrace 0,1 \rbrace$.
So $(a_1,a_2) \in \lbrace (0,0),(1,0) \rbrace$.
Suppose first that $(a_1,a_2) = (0,0)$.
Then by (\ref{D/4bound}), we can choose distinct $e,e' \in E(G[A])$.
In this case $\mathcal{P} := \mathcal{P}_0 \cup \lbrace e,e' \rbrace$ satisfies (P1)--(P3).

Now suppose that $(a_1,a_2)=(1,0)$.
Then (\ref{BCeq}) implies that
\begin{equation}\label{BUedges}
e_G(B,U) \geq e_{\mathcal{P}_0}(B,U)=3.
\end{equation}
Let $au$ be the single edge in $\mathcal{P}_0[A,U]$, where $a \in A$ and $u \in U$.
Note that any edge in $E(G[A \setminus \lbrace a \rbrace,U])$ is incident with $u$ since $G[A,U]$ contains no matching of size two.
So $e_G(A \setminus \lbrace a \rbrace,U) = d_{A \setminus \lbrace a \rbrace}(u)$.
Thus Proposition~\ref{sumfact} and (\ref{BUedges}) imply that
\begin{equation}\label{3edges}
2e_G(A \setminus \lbrace a \rbrace) + d_{A \setminus \lbrace a \rbrace}(u) \geq 3.
\end{equation}
Suppose first that $d_A(a) \leq 1$.
In this case, (\ref{D/4bound}) implies that $e_G(A \setminus \lbrace a \rbrace) \geq D/4-1 \geq 2$.
Let $e,e' \in E(G[A \setminus \lbrace a \rbrace])$ be distinct.
Then $\mathcal{P} := \mathcal{P}_0 \cup \lbrace e,e' \rbrace$ satisfies (P1)--(P3).

Now suppose that $d_A(a) \geq 2$.
Let $a',a'' \in N_A(a)$ be distinct.
Suppose that $e_G(A \setminus \lbrace a \rbrace) \neq 0$.
Then we can choose $e \in E(G[A \setminus \lbrace a \rbrace])$, and $\mathcal{P} := \mathcal{P}_0 \cup \lbrace aa',e \rbrace$ satisfies (P1)--(P3).
Suppose instead that $e_G(A \setminus \lbrace a \rbrace) = 0$.
Then $d_{A \setminus \lbrace a \rbrace}(u) \geq 3$ by~(\ref{3edges}), so there exists $a^* \in A \setminus \lbrace a,a',a'' \rbrace$ such that $ua^* \in E(G[A,U])$.%
\COMMENT{
Clearly $\mathcal{P}_0' := \mathcal{P}_0 \cup \lbrace ua^* \rbrace \setminus \lbrace ua \rbrace$ is a basic connector with $\bal_{AB}(\mathcal{P}_0) = -1$ and $a \notin V(\mathcal{P}_0')$.}
We have that $\mathcal{P} := \mathcal{P}_0 \cup \lbrace ua^*, a'aa'' \rbrace \setminus \lbrace ua \rbrace$ satisfies (P1)--(P3).
\end{proof}

We are now ready to combine the preceding lemmas to prove Lemma~\ref{aim} fully in the case when $|A|=|B|+1$.

\medskip
\noindent
\emph{Proof of Lemma~\emph{\ref{aim}} in the case when $|A|=|B|+1$.}
Let $U := V_1 \cup V_2$.
Suppose first that $G[A,U]$ contains a matching of size three.
Then we are done by Lemma~\ref{3okay}, so assume not.
Proposition~\ref{BC} implies that $G$ contains a basic connector.
Choose a basic connector $\mathcal{P}_0$ in $G$ such that $|\bal_{AB}(\mathcal{P}_0)-1|$ is minimal.
Recall that (BC2) implies $|\bal_{AB}(\mathcal{P}_0)| \leq 2$.
Since $G[A,U]$ does not contain a matching of size three, the consequence~(i) of Proposition~\ref{balmin} implies that $\bal_{AB}(\mathcal{P}_0) \leq 1$.
We may assume that $\bal_{AB}(\mathcal{P}_0) \leq 0$ or we are done.
Lemmas~\ref{bal0} and~\ref{bal-1} prove the lemma in the case when $\bal_{AB}(\mathcal{P}_0) = 0,-1$ respectively.
So we may assume that $\bal_{AB}(\mathcal{P}_0) = -2$.
Thus, by (\ref{BCeq}), we have $e_G(B,U) \ge 4$.
Moreover, by the consequence~(iii) of Proposition~\ref{balmin} we may assume that $e_G(A,U)=0$.
Now (\ref{balance1}) implies $e_G(A) \geq D/2+2$.
The `moreover' part of Lemma~\ref{goodmatching2} with $G[A],D/2,1$ playing the roles of $G,\Delta,\ell$ implies that $G[A]$ contains a matching $M_A$ of size two and an edge $aa'$ with $a \notin V(M_A)$.
So $\mathcal{P} := \mathcal{P}_0 \cup M_A \cup \lbrace aa' \rbrace$ satisfies (P1)--(P3).
\hfill$\square$

\subsection{The proof of Lemma~\ref{aim} in the case when $|A|=|B|$}\label{equal}

In this subsection we consider the only remaining case of Lemma~\ref{aim}: when the bipartite vertex classes $A$ and $B$ have equal size.
Our aim is to find a path system $\mathcal{P}$ such that $R_{\mathcal{V}}(\mathcal{P})$ is an Euler tour, and $\bal_{AB}(\mathcal{P})=0$.
As in the previous section, we will appropriately modify a basic connector guaranteed by Proposition~\ref{BC}.
The degree bound $D \geq n/4$ is used again here.

\medskip
\noindent
\emph{Proof of Lemma~\emph{\ref{aim}} in the case when $|A|=|B|$.}
Let $U := V_1 \cup V_2$.
The consequence~(i) of Proposition~\ref{fact2} implies that
\begin{equation}\label{balance0}
2e_G(A) + e_G(A,U) = 2e_G(B) + e_G(B,U).
\end{equation}
Proposition~\ref{BC} implies that $G$ contains a basic connector.
Choose a basic connector $\mathcal{P}_0$ in $G$ such that $|\bal_{AB}(\mathcal{P}_0)|$ is minimal.
Write $F_{\mathcal{P}_0}(A):=(a_1,a_2)$.

Suppose first that $e_G(B,U)=0$.%
    \COMMENT{DK changed this para}
Then 
$$2\bal_{AB}(\mathcal{P}_0) \stackrel{(\ref{BCeq})}{=} a_1+2a_2 = e_{\mathcal{P}_0}(A,U) \leq e_G(A,U)
\stackrel{(\ref{balance0})}{\le} 2e_G(B).$$
(In particular, $\bal_{AB}(\mathcal{P}_0)\ge 0$.)
Let $E' \subseteq E(G[B])$ be a collection of $\bal_{AB}(\mathcal{P}_0)$ distinct edges (so $|E'| \leq 2$ by (BC2)).
Then $\mathcal{P} := \mathcal{P}_0 \cup E'$ satisfies (P1)--(P3).
Thus we may assume that $e_G(B,U) \geq 1$ and a similar argument allows us to assume that $e_G(A,U) \geq 1$.

Together with the $3$-connectivity of $G$, this implies that $G[W,U]$ contains a matching $M$ of size two such that one edge is incident with $A$ and one edge is incident with $B$.%
\COMMENT{$G[W,U]$ contains a matching of size three by $3$-connectivity. If it is contained in $G[A,U]$ then replace one edge by the edge in $G[B,U]$. and vice versa.}
The consequence~(iv) of Proposition~\ref{balmin} and our choice of $\mathcal{P}_0$ together imply that $|\bal_{AB}(\mathcal{P}_0)| \leq 1$.
Without loss of generality we suppose that $\bal_{AB}(\mathcal{P}_0) = -1$ (otherwise $\bal_{AB}(\mathcal{P}_0)=1$
and we could swap $A$ and $B$, or $\bal_{AB}(\mathcal{P}_0)=0$ and we are done by taking $\mathcal{P}:=\mathcal{P}_0$).
Then (BC4) implies that $(a_1,a_2) \in \lbrace (0,0),(1,0) \rbrace$.
If $e_G(A) \geq 1$ then, for any $e \in E(G[A])$ we have that $\mathcal{P} := \mathcal{P}_0 \cup \lbrace e \rbrace$ satisfies (P1)--(P3).
So we may assume that
\begin{equation}\label{eGA}
e_G(A)=0.
\end{equation}

\medskip
\noindent
\textbf{Claim 1.}
\emph{$G[A,U]$ does not contain a matching of size two.}

\medskip
\noindent
To prove the claim, suppose not.
We will show that if $G[A,U]$ contains a matching of size two, then the minimality of $|\bal_{AB}(\mathcal{P}_0)|$ will be contradicted.
First consider the case when $(a_1,a_2)=(1,0)$.
So $e_{\mathcal{P}_0}(A,U)=1$ and therefore $e_{\mathcal{P}_0}(B,U)=3$ by (\ref{BCeq}).
But (BC2) implies that $e(\mathcal{P}_0) \leq 4$, so $e_{\mathcal{P}_0}(V_1,V_2)=0$.
Now by (BC1) we have that $|V(\mathcal{P}_0) \cap V_i|=2$ for $i=1,2$, and $d_{\mathcal{P}_0}(v)=1$ for all $v \in V(\mathcal{P}_0) \cap V_i$.
In particular, $e_{\mathcal{P}_0}(V_i,B)>0$ for both $i=1,2$.
Let $e$ be the single edge in $\mathcal{P}_0[A,U]$.
Without loss of generality, we may assume that $G[A,U]$ contains an edge $e'$ which is vertex-disjoint from $e$.
(Otherwise, $G[A,U]$ contains a matching $av,a'v'$ such that $e = av'$. Then $\mathcal{P}_0' := \mathcal{P}_0 \cup \lbrace a'v' \rbrace \setminus \lbrace e \rbrace$ is a basic connector with $\bal_{AB}(\mathcal{P}_0')=\bal_{AB}(\mathcal{P}_0)$ and $a'v'$ is the single edge in $\mathcal{P}_0'[A,U]$; and $av$ is an edge which is vertex-disjoint from $a'v'$.)%
\COMMENT{this bracket is more detailed.}
Suppose first that $e'$ has an endpoint in $V_1$.
If possible, choose $f \in E(\mathcal{P}_0[V_1,B])$ which is incident with $e'$; otherwise let $f \in E(\mathcal{P}_0[V_1,B])$ be arbitrary.
Then $\mathcal{P} := \mathcal{P}_0 \cup \lbrace e' \rbrace \setminus \lbrace f \rbrace$ contradicts the minimality of $|\bal_{AB}(\mathcal{P}_0)|$.
The case when $e'$ has an endpoint in $V_2$ is similar.

Suppose now that $(a_1,a_2)=(0,0)$.
Then $e_{\mathcal{P}_0}(A,U)=0$ and hence $e_{\mathcal{P}_0}(B,U)=2$.
Moreover, $\mathcal{P}_0[B,U]$ is a matching $e, e'$ since $\mathcal{P}_0$ is an Euler tour by (BC1).
Now $d_{R_{\mathcal{V}}(\mathcal{P}_0)}(V_i) \geq 2$ for $i=1,2$, so $e_{\mathcal{P}_0}(V_1,V_2)  \geq 1$.
But (BC2) implies that $e(\mathcal{P}_0) \leq 4$, so $e_{\mathcal{P}_0}(V_1,V_2) \leq 2$.
Suppose that $e_{\mathcal{P}_0}(V_1,V_2)=1$ and let $f \in E(\mathcal{P}_0[V_1,V_2])$.
Then $\mathcal{P}_0 = \lbrace e,e', f \rbrace$ is a matching of size three.
Moreover $e_{\mathcal{P}_0}(B,V_i)=1$ for $i=1,2$.
If there exists $e_A \in E(G[A,U] \setminus V(f))$ then we can replace one of $e,e'$ by $e_A$ to contradict the minimality of $|\bal_{AB}(\mathcal{P}_0)|$.
Therefore there is a matching $\lbrace e_A,e_A' \rbrace \subseteq E(G[A,U])$ such that both $e_A,e_A'$ are incident to $V(f)$.
Then they are vertex-disjoint from $\lbrace e,e' \rbrace$, so $\mathcal{P} := \lbrace e,e',e_A,e_A' \rbrace$ contradicts the minimality of $|\bal_{AB}(\mathcal{P}_0)|$.
Suppose now that $e_{\mathcal{P}_0}(V_1,V_2)=2$.
Then $\mathcal{P}_0[B,U] \subseteq G[B,V_i]$ for some $i=1,2$.
Without loss of generality we assume that $i=2$.
Suppose that there exists $e_A \in E(G[A,V_1])$.
Choose $f \in E(\mathcal{P}_0[V_1,V_2])$ that is not incident to $e_A$.
Choose $e_B \in E(\mathcal{P}_0[B,V_2])$ that is not incident to $f$.
Then $\mathcal{P} := \lbrace e_A,f,e_B \rbrace$ contradicts the minimality of $|\bal_{AB}(\mathcal{P}_0)|$.
Therefore we may assume that there is a matching $M_A \subseteq G[A,V_2]$ of size two.
There is at least one $V_1V_2$-path in $\mathcal{P}_0$ (which consists of a single edge $f'$).
Choose $e \in M_A$ which is not incident to $f'$.
If possible, let $e_B$ be the edge of $\mathcal{P}_0[B,V_2]$ which is incident to $e$; otherwise let $e_B \in E(\mathcal{P}_0[B,V_2])$ be arbitrary.
Then $\mathcal{P} := \mathcal{P}_0 \cup \lbrace e \rbrace \setminus \lbrace e_B \rbrace$ contradicts the minimality of $|\bal_{AB}(\mathcal{P}_0)|$.
This completes the proof of the claim.

\medskip
\noindent
Therefore $e_G(A,U) \leq D/2$ since $\Delta(G[A,U)] \leq D/2$.
So (\ref{balance0}) and (\ref{eGA}) together imply that
\begin{equation}\label{eWU}
e_G(W,U) = e_G(B,U) - e_G(A,U) + 2e_G(A,U) \leq D.
\end{equation}
Suppose first that $|A|=|B| = D-k$ for some $k \in \mathbb{N}$.
Then (\ref{eGA}) implies that, for all $a \in A$, we have $d_U(a)=D-d_A(a)-d_B(a) \geq D-|B| = k$.
So $e_G(A,U) \geq k|A| = k(D-k) \geq D-1$, a contradiction.
So $|A|=|B| \geq D$ and hence $|U| = n-|A|-|B| \leq n-2D \leq 2D$ since $D \geq n/4$.

\medskip
\noindent
\textbf{Claim 2.}
\emph{There exists a matching $M'$ of size three in $G[V_1,V_2]$.}

\medskip
\noindent
To prove the claim, assume without loss of generality that $|V_1| \leq |V_2|$.
Then there exists $s \in \mathbb{N}_0$ such that $|V_1|=D-s$.
Recall from our assumption in Lemma~\ref{aim} that%
   \COMMENT{DK rewrote this sentence}
$|V_1| \geq D/2$. Suppose first that $s \geq 2$.
Then
\begin{align}
e_G(V_1,V_2) &\geq D|V_1| - e_G(U,W) - 2\binom{|V_1|}{2} \stackrel{(\ref{eWU})}{\geq} |V_1|(D-|V_1|+1)-D\\
\nonumber &\geq \min \lbrace D^2/4-D/2, 2D-6 \rbrace \geq D+1.
\end{align}
Recall that $d_{V_i}(x_i) \geq d_{V_j}(x_i)$ for all $x_i \in V_i$ and $\{i,j\}=\{1,2\}$.
So $\Delta(G[V_1,V_2]) \leq D/2$.
Therefore we are done by K\"onig's theorem on edge-colourings.

Thus we may assume that $s \in \lbrace 0,1 \rbrace$.
Let $H := G[V_1,V_2]$.
Suppose that $H$ contains no matching of size three.
By K\"onig's theorem on vertex covers, $H$ contains a vertex cover $\lbrace v_i,v_j \rbrace$
where $v_i \in V_i$, $v_j \in V_j$ and $i,j$ are not necessarily distinct.
So $e(H) \leq d_H(v_i)+d_H(v_j)$.
Note that the complement $\overline{G}$ of $G$ satisfies%
\COMMENT{$-1$: if $v_iv_j \in E(\overline{G})$ and $i=j$.}
\begin{align}\label{missingedges}
\nonumber e_{\overline{G}}(V_1) + e_{\overline{G}}(V_2) &\geq d_{\overline{G}[V_i]}(v_i) + d_{\overline{G}[V_j]}(v_j)-1 = |V_i|-d_{V_i}(v_i) + |V_j| - d_{V_j}(v_j)-3\\
\nonumber &\geq D - d_{V_i}(v_i) + D - d_{V_j}(v_j) - 5
\geq d_H(v_i) + d_H(v_j) - 5\\
&\geq e(H) - 5.
\end{align}
Therefore by counting the degrees in $G$ of the vertices in $U$, we have that
\begin{eqnarray*}
e_G(U,W) &=& \sum_{v \in V_1}d_G(v) + \sum_{v \in V_2}d_G(v) - 2e(H) -2e_G(V_1)-2e_G(V_2)\\
&=& D(|V_1|+|V_2|) - 2e(H) - 2 \left( \binom{|V_1|}{2} - e_{\overline{G}}(V_1)  + \binom{|V_2|}{2} - e_{\overline{G}}(V_2) \right)\\~~
&\stackrel{(\ref{missingedges})}{\geq}& D(|V_1|+|V_2|)-10 - 2\binom{|V_1|}{2} - 2\binom{|V_2|}{2}\\
&=& |V_1|(D-|V_1|) + |V_2|(D-|V_2|) + |V_1|+|V_2|-10 \geq 2D-14, 
\end{eqnarray*}
a contradiction to (\ref{eWU}).%
\COMMENT{$|V_1|(D-|V_1|) + |V_2|(D-|V_2|) = 0$ if $s=0$ and is at least $(D-1)-(D+1)=-2$ when $s=1$. Moreover, $|V_1|+|V_2| \geq 2D-2$}
This proves the claim.

\medskip
\noindent
Recall that $M$ is a matching of size two in $G[W,U]$ with one edge incident to $A$ and one edge incident to $B$.
Assume without loss of generality that $e_M(V_2,W) \geq e_M(V_1,W)$.
There exists $e \in E(M')$ which is vertex-disjoint from $M$.
Suppose first that $e_M(V_2,W)=2$.
Let $e' \in E(M') \setminus \lbrace e \rbrace$ be arbitrary.
Then $\mathcal{P} := M \cup \lbrace e,e' \rbrace$ satisfies (P1)--(P3).
Suppose instead that $e_M(V_2,W)=e_M(V_1,W)=1$.
Then $\mathcal{P} := M \cup \lbrace e \rbrace$ satisfies (P1)--(P3).
This completes the proof of Lemma~\ref{aim} in all cases.
\hfill$\square$

\section{The proof of Theorem~\ref{exact}}\label{proof}

We are now ready to prove Theorem~\ref{exact}.
It is a consequence of Theorem~\ref{stability} and Lemma~\ref{HES} (both proved in~\cite{klos}), as well as Lemmas~\ref{(4,0)},~\ref{(0,2)} and~\ref{(2,1)}.

\medskip
\noindent
\emph{Proof of Theorem~\emph{\ref{exact}}.}
Choose a non-decreasing function $g: (0,1) \rightarrow (0,1)$ with $g(x) \leq x$ for all $x \in (0,1)$
such that the requirements of
Proposition~\ref{WRSD-RD} and
Lemmas~\ref{HES},~\ref{(4,0)},~\ref{(0,2)},~\ref{(2,1)}
(each applied, where relevant, with $1/32,1/4$ playing the roles of $\eta,\alpha$) 
are satisfied whenever $n,\rho,\gamma,\nu,\tau$ satisfy
\begin{align} \label{hierarchy}
1/n &\leq g(\rho), g(\gamma);\ \ \rho,\gamma \leq g(\nu);\ \ \nu \leq g(\tau);\ \ \tau \leq g(1/32).
\end{align}
Choose $\tau,\tau'$ so that
$$
0 \leq \tau' \leq \tau \leq g(1/32),40^{-3}\ \ \mbox{ and }\ \ \tau' \leq g(\tau).
$$
Define a function $g' : (0,1) \rightarrow (0,1)$ by $g'(x) = (g(x))^3$.%
\COMMENT{Note that $g'(x) < g(x) \leq g(x^{1/3})$ for all $x \in (0,1)$, where the second inequality holds since $g$ is non-decreasing.
Note also that $2\tau'^{1/3} \leq \eps$.}
Apply Theorem~\ref{stability} with $g',\tau',1/20$ playing the roles of $g,\tau,\eps$ to obtain an integer $n_0$.%
\COMMENT{Have $\gamma=\rho^{1/3}$ in Lemma~\ref{(0,2)} so use $g^3$ to ensure that here $\gamma = \rho^{1/3} \leq g(\nu)$.}
Let $G$ be a 3-connected $D$-regular graph on $n \geq n_0$ vertices where $D \geq n/4$.
We may assume that the consequence~(ii) of Theorem~\ref{stability} holds or we are done.
Thus there exist
$\rho,\nu$ with $1/n_0 \leq \rho \leq \nu \leq \tau'$, $1/n_0 \leq g'(\rho)$
 and $\rho \leq g'(\nu)$; and $(k,\ell) \in \lbrace (4,0),(2,1),(0,2) \rbrace$ such that $G$ has a robust partition $\mathcal{V}$ with parameters $\rho,\nu,\tau',k,\ell$ (and thus also a robust partition with parameters $\rho,\nu,\tau,k,\ell$).

Let $\gamma:=\rho^{1/3}$. Note that $n,\rho,\gamma,\nu,\tau$ satisfy (\ref{hierarchy}).%
\COMMENT{Then $1/n \leq 1/n_0 \leq g'(\rho) \leq g(\rho^{1/3})=g(\gamma)$, and $\gamma = \rho^{1/3} \leq (g')^{1/3}(\nu) = g(\nu)$}
Apply Lemmas~\ref{(4,0)},~\ref{(0,2)} in the cases when $(k,\ell)$ equals $(4,0),(0,2)$ respectively to obtain a $\mathcal{V}$-tour of $G$
with parameter $\gamma$.%
\COMMENT{(Note that a $\mathcal{V}$-tour with parameter $\gamma$ is a $\mathcal{V}$-tour with parameter $\gamma'$ whenever $\gamma \leq \gamma'$.)}
Proposition~\ref{WRSD-RD} implies that $\mathcal{V}$ is a weak robust partition with parameters $\rho,\nu,\tau,1/32,k,\ell$.
Then Lemma~\ref{HES} implies that $G$ contains a Hamilton cycle.
Apply Lemma~\ref{(2,1)} in the case when $(k,\ell)=(2,1)$ to obtain a Hamilton cycle in $G$.
This completes the proof of the theorem.
\hfill$\square$

\section*{Acknowledgements}
The authors would like to thank Peter Allen for his helpful comments.%
\COMMENT{NEW and updated refs}

\medskip

{\footnotesize \obeylines \parindent=0pt

Daniela K\"uhn, Allan Lo, Deryk Osthus
School of Mathematics
University of Birmingham
Edgbaston
Birmingham
B15 2TT
UK

\medskip

Katherine Staden
Mathematics Institute
University of Warwick
Coventry
CV4 7AL
UK

}
\begin{flushleft}
{\it{E-mail addresses}:
\tt{\{d.kuhn, s.a.lo, d.osthus\}@bham.ac.uk, k.l.staden@warwick.ac.uk}}
\end{flushleft}

\end{document}